\documentclass[a4paper, 10pt, twoside, notitlepage]{amsart}

\usepackage[utf8]{inputenc}
\usepackage{color}
\usepackage{amsmath} 
\usepackage{amssymb} 
\usepackage{amsrefs}
\usepackage{amsthm}
\usepackage{geometry}
\usepackage{graphicx}
\usepackage{esint}
\usepackage[ocgcolorlinks,linkcolor=blue]{hyperref}
\usepackage{diagbox}
\usepackage{tabu}
\usepackage{pgf,tikz,pgfplots}
\pgfplotsset{compat=1.13}
\usepackage{mathrsfs}
\usetikzlibrary{arrows}

\definecolor{ffqqqq}{rgb}{1,0,0}
\definecolor{qqqqff}{rgb}{0,0,1}
\definecolor{cqcqcq}{rgb}{0.75,0.75,0.75}
\definecolor{zzttqq}{rgb}{0.6,0.2,0.}

\theoremstyle{plain}
\newtheorem{thm}{Theorem}
\newtheorem{prop}{Proposition}[section]
\newtheorem{lem}[prop]{Lemma}

\newtheorem{rmk}[prop]{Remark}

\newcommand {\R} {\mathbb{R}} \newcommand {\Z} {\mathbb{Z}}
\newcommand {\T} {\mathbb{T}} \newcommand {\N} {\mathbb{N}}
\newcommand {\C} {\mathbb{C}} 
\newcommand {\p} {\partial}

\newcommand {\D} {\Delta}

\newcommand {\supp} {\text{supp}}

\newcommand{\dis}{\mathrm{d}}

\DeclareMathOperator{\F} {\mathcal{F}}

\allowdisplaybreaks
\date{\today}

\title[(Global) Unique Continuation Properties of the Fractional Discrete Laplacian]{On (Global) Unique Continuation Properties of the Fractional Discrete Laplacian}

\author[A. Fern\'andez-Bertolin]{Aingeru Fern\'andez-Bertolin}
\address{Facultad de Ciencia y Tecnología, Universidad del País Vasco /Euskal Herriko Unibertsitatea (UPV/EHU), Departamento de Matemáticas, UPV/EHU, Apartado 644, 48080 Bilbao, Spain}
\email{aingeru.fernandez@ehu.eus}
\author[L.\ Roncal]{Luz Roncal}
\address[Luz Roncal]{
BCAM - Basque Center for Applied Mathematics,
48009 Bilbao, Spain
and
Ikerbasque, Basque Foundation for Science,
48011 Bilbao, Spain
and
Universidad del Pa\'is Vasco / Euskal Herriko Unibertsitatea,
48080 Bilbao, Spain}
\email{lroncal@bcamath.org}
\author[A. R\"uland]{Angkana R\"uland}
\address{Ruprecht-Karls-Universität Heidelberg, Institut für Angewandte Mathematik, Im Neuenheimer Feld 205, 69120 Heidelberg, Germany}
\email{Angkana.Rueland@uni-heidelberg.de}


\begin{document}

\begin{abstract}
We study various qualitative and quantitative (global) unique continuation properties for the fractional discrete Laplacian. We show that while the fractional Laplacian enjoys striking rigidity properties in the form of (global) unique continuation properties, the fractional discrete Laplacian does not enjoy these in general. While discretization thus counteracts the strong rigidity properties of the continuum fractional Laplacian, by discussing quantitative forms of unique continuation, we illustrate that these properties can be recovered if exponentially small (in the lattice size) correction terms are added. This in particular allows us to deduce uniform stability properties for a discrete, linear inverse problem for the fractional Laplacian. We complement these observations with a transference principle and the discussion of these properties on the discrete torus.
\end{abstract}

\maketitle

\section{Introduction}

The fractional Laplacian is a nonlocal elliptic operator which enjoys striking unique continuation properties (UCP). These are substantially stronger than for local elliptic operators: The fractional Laplacian is an ``antilocal operator" \cite{L82,V93,SVW02,GFR20} and thus satisfies \emph{global} unique continuation properties \cite{GSU16}. Fractional Schrödinger equations satisfy the weak and strong unique continuation properties as well as the UCP from measurable sets \cite{FF14,Rue15,S15,Yu17,GRSU18,GFR19}. Moreover, the associated Caffarelli--Silvestre extension \cite{CS07} enjoys quantitative boundary-bulk doubling properties \cite{Rue21,RS20}. These properties have played a major role in the study of nonlocal inverse problems such as the fractional Calder\'on problem \cite{GSU16} (see also \cite{S17, Rue18} for surveys on these).
It is well-known that discrete counterparts of local elliptic operators often exhibit different, and in terms of local properties often weaker, unique continuation properties \cite{BHR10, BHLR10a, FB19, FBM21,FBV17,GM13,GM14,LM15,LM17, FBRRS20}.
It is hence the main objective of this article to investigate this phenomenon for the fractional discrete Laplacian. In particular, we study discrete counterparts for various types of qualitative and quantitative unique continuation properties for the fractional Laplacian on the lattice $(h\Z)^d$ and the discrete torus $(h\T)^d$.

\subsection{Global unique continuation}

For the fractional Laplacian a remarkable \emph{global} unique continuation property holds which is not available for local elliptic operators such as the Laplacian:

\begin{thm}[Global UCP, \cite{GSU16}]
\label{prop:unique}
Let $s\in (0,1)$ and let $\Omega \subset \R^d$ with $d\geq 1$ be open.
Let $u\in H^{r}(\R^d)$ for some $r\in \R$ and assume that 
\begin{align*}
u=0=(-\D)^s u \mbox{ in } \Omega.
\end{align*}
Then $u\equiv 0$ in $\R^d$.
\end{thm}
A robust proof of this can, for instance, be obtained by using the unique continuation and Carleman estimates from \cite{Rue15} (cf. also \cite{FF14} and \cite{Yu17}), but also potential theoretic approaches \cite{R38} are possible.
In the physics community this property is known as the \emph{antilocality} of the fractional Laplacian \cite{V93, L82,GFR20,CGFR21}.

As a first observation, we illustrate that for \emph{discrete} versions of the fractional Laplacian, the global unique continuation property fails on the discrete lattice $(h\Z)^d$ with $h\in \R_+$. Given a function $u:(h\Z)^d\to \R$, we define the discrete Laplacian as  
$$
(-\D_{\dis})u(hj):=\frac{1}{h^2}\sum_{i=1}^d\big(u(h(j+e_i))-2u(hj)+u(h(j-e_i))\big), \quad hj\in (h\Z)^d.
$$
Based on this definition of the Laplacian on the lattice $(h\Z)^d$, it is possible to define the \emph{fractional discrete Laplacian} by means of its heat semigroup representation or the associated semi-discrete Caffarelli--Silvestre extension (see Section \ref{sec:pre}). We denote the corresponding operator by $(-\D_{\dis})^s$.
Let us further write $u_j:=u(hj)$ to denote the value of $u$ at the mesh point $hj\in (h\Z)^d$. With this notation fixed, we observe that on the lattice, the direct counterpart of Theorem \ref{prop:unique} fails:

\begin{thm}
\label{thm:exterior_UCP}
Let $h \in \R_+$ and let $X \subset (h\Z)^d$ be a finite set of cardinality $M \in \N$. Then there exists a 
non-zero function $u\in \ell_{ s}$ such that $u_j=0=(-\D_{\dis})^s u_j = 0 $ for $j\in X$. 
\end{thm}

We refer to Section \ref{sec:pre} for a definition of the space $\ell_s$. The result of Theorem \ref{thm:exterior_UCP} essentially reduces to the solvability of an underdetermined homogeneous linear system which is a consequence of the fact that the function $u$ can be chosen arbitrarily outside of $X$.

While the strongest version of the global unique continuation property of Theorem \ref{prop:unique} fails for discrete operators, we highlight that a weaker (still qualitative) counterpart of it persists in the form of global unique continuation properties from the exterior:

\begin{thm}
\label{prop:exterior_UCP1}
Let $d\geq 1$, $h\in \R_+$ and $u \in H^{r}((h\Z)^d)$ for some $r\in \R$. Let $s\in (0,1)$ and assume that for some $R\in \R_+$, $R\geq h$,
\begin{align*}
u = 0 = (-\D_{\dis})^s u \mbox{ in }(h\Z)^d \setminus B_{R}.
\end{align*}
Then, $u \equiv 0$ in $(h\Z)^d$.
\end{thm}

Here for $R>0$, we have set $B_R=B_R(0)\cap(h\Z)^d$, with $h>0$; we refer to Section \ref{sec:pre} for a definition of the function space $H^r((h\Z)^d)$.

\begin{rmk}
We remark that while in Theorem \ref{prop:exterior_UCP1} we have required the condition $R\geq h$, this is strictly speaking not necessary. In this case however the statement would be true by definition.
\end{rmk}

Let us comment on the result of Theorem \ref{prop:exterior_UCP1}. It is an example of a result showing that while in the discrete setting the \emph{local}, qualitative variants of unique continuation can in many cases be easily violated due to the introduction of the discretization length scale, the more \emph{global} qualitative vanishing behavior is substantially more rigid. In our specific setting of the fractional discrete Laplacian the vanishing result from the exterior follows along similar lines as in the continuum and is a consequence of the fact that the analytic extension of the function $\xi \mapsto |\xi|^{s}$ for $s\in \R \setminus \Z$ must have a branch cut. Similar arguments had earlier been used in \cite{Isakov,L82,RS20a,RSV19}.
Exploiting this idea further, it is also possible to prove unique continuation for the fractional discrete Laplacian from the upper half-plane as a consequence of a similar Paley--Wiener type argument:

\begin{prop}
\label{prop:exterior_UCP}
Let $d\geq 1$ and $u \in H^{r}((h\Z)^d)$ for some $r\in \R$. Let $s\in (0,1)$ and assume that 
\begin{align*}
u = 0 = (-\D_{\dis})^s u \mbox{ in } \{x\in (h\Z)^d: \ x_{d}\geq 0\}.
\end{align*}
Then, $u \equiv 0$ in $(h\Z)^d$.
\end{prop}

\begin{rmk}
We emphasize that here we consider the special case of weak unique continuation from a fixed half-plane. In the local case, $s=1$, it is known that if the half-plane is tilted by $45^{\circ}$, it is possible to construct harmonic functions in $(h\Z)^2$ vanishing in this tilted half-plane, \cite{BLMS17}.
\end{rmk}

\subsection{Weak unique continuation for fractional discrete Schrödinger equations in slab domains}

While the main focus of our article is the investigation of the degree of the failure of the global unique continuation property for the fractional discrete Laplacian, we also briefly consider the weak unique continuation property for fractional discrete Schrödinger equations in slab domains. Although this is an in one direction global property and although we require the global validity of an equation, we show that for the fractional discrete Laplacian the weak unique continuation property from (thin) slab domains fails.

For simplicity, we only formulate and discuss the failure of weak unique continuation for slab domains in two dimensions and refer to Remark \ref{rmk:nD} for comments on extensions to the higher dimensional setting. More precisely, in the two-dimensional setting we record the following example:

\begin{prop}[Failure of weak unique continuation from slab domains]
\label{prop:slab_domain}
Let $s\in (0,1)$. 
Then there exist non-trivial sequences $\{u_j\}_{j\in \Z^2} \in \ell_s$ and $\{V_j\}_{j\in \Z^2} \in \ell^{\infty}_{(h\Z)^2}$ such that
\begin{align*}
(-\D_{\dis})^s u_j = V_j u_j \mbox{ for all } j=(j_1,j_2) \in \Z^2 \mbox{ and }
u_j = 0 \mbox{ for all } j_1 \in\{-1,0,1\}. 
\end{align*}
\end{prop}

\begin{rmk}
\label{rmk:wucp}
A few remarks are in order.
\begin{itemize}
\item[(i)] \emph{Comparison with the discrete Laplacian in two dimensions.} For the discrete (five-point) Laplacian in two dimensions, examples of the above type are \emph{not} possible since the vanishing of $u$ in the slab would be propagated to the full domain. The existence of such an example thus is a consequence of the stronger nonlocality of the fractional discrete Laplacian. However, as we pointed out above, if the slab domain is tilted by $45^{\circ}$ degrees, then the discrete Laplacian does not satisfy weak unique continuation from that domain.
\item[(ii)] \emph{Arbitrarily thick slab domains.} It would be desirable to extend the above example to an example in which the slab is arbitrarily thick (and not only of a thickness of three points $j_1\in \{0,\pm 1\}$). While for a finite thickness such examples can be constructed (for relatively large values this is still possible with the aid of e.g. symbolic Mathematica computations), for an arbitrarily large thickness this in general reduces to an invertibility condition for a matrix with entries given by the kernel $K_s^h$ (see Section \ref{sec:pre} for its definition) evaluated at suitable points. If it is possible to ensure the invertibility of this matrix and the non-vanishing of the components of the solution to an associated inhomogeneous equation, one would obtain a general example of the failure of the UCP in discrete slab domains.
\end{itemize}
\end{rmk}

We will prove Proposition \ref{prop:slab_domain} in Section \ref{sec:wucp}, where we will reduce the result to a corresponding one-dimensional example.

\subsection{Quantitative (global and boundary-bulk) unique continuation}

A class of results which is strongly related to the unique continuation properties of the fractional Laplacian consists of boundary-bulk inequalities for (weighted) elliptic equations \cite{JL99,BL15,RS20,Rue21}. Indeed, the relation between boundary-bulk unique continuation properties and unique continuation estimates for the fractional Laplacian follows from the characterization of the fractional Laplacian by means of the Caffarelli--Silvestre extension \cite{CS07}: In the continuum setting, for $s\in (0,1)$, the fractional Laplacian can be characterized as a (weighted, weakly defined) Dirichlet-to-Neumann map:
\begin{align*}
(-\D)^s u(x) = c_{s}\lim\limits_{t\rightarrow 0} t^{1-2s} \p_{t} \tilde{u}(x,t),
\end{align*}
where $c_{s}\neq 0$ is a real-valued explicit constant and $\tilde{u}(x,t)$ is a solution to
\begin{align}
\label{eq:CS1}
\begin{split}
\nabla \cdot t^{1-2s} \nabla \tilde{u} & = 0 \mbox{ in } \R^{d+1}_+,\\
\tilde{u} & = u \mbox{ on } \R^d \times \{0\}.
\end{split}
\end{align}
Here $\nabla := (\nabla_{x},\p_t)^T$.
This representation in particular allows one to formulate (weak, strong and global) unique continuation properties for the fractional Laplacian by means of \emph{boundary} unique continuation for the solutions to the (degenerate) elliptic equation \eqref{eq:CS1} (see for instance \cite{FF14,FF15,GFR19,Rue15,Rue17,Yu17}). In the sequel, in order to avoid additional technical difficulties, we focus on the case $s= \frac{1}{2}$ for which \eqref{eq:CS1} turns into the harmonic extension of the function $u$. In this setting, it is well-known \cite{JL99} that the following boundary-bulk inequality holds:

\begin{thm}[Quantitative boundary-bulk unique continuation in the continuum, \cite{JL99}] 
\label{thm:bdry_bulk}
Let $u\in H^{1}(\R^d)$ and let $\tilde{u}$ be a weak solution to
\begin{align*}
(\D +  \p^2_t)\tilde{u} & = V \tilde{u} \mbox{ in } \R^{d+1}_+,\\
\tilde{u} & = u \mbox{ on } \R^d \times \{0\},
\end{align*}
for $V\in L^{\infty}(\R^{d+1}_+)$. Then, there exist constants $C>0$ (depending only on $d,\|V\|_{L^{\infty}(\R^{d+1}_+)}$) and $\alpha \in (0,1)$ (depending only on $d$) such that
\begin{align}
\label{eq:bdry_bulk}
\|\tilde u\|_{L^2(B_1^+)} \leq C \|\tilde{u}\|_{L^2(B_4^+)}^{1-\alpha}\left( \|u\|_{H^1(B_4')} + \|\p_{t} \tilde{u}\|_{L^2(B_4')} \right)^{\alpha}.
\end{align}
\end{thm}

Here and in the following sections, for $x_0 \in \R^d \times \{0\}$, we use the notation $B_r^+(x_0):=\{(x,t)\in \R^{d+1}_+: \ |(x,t)-x_0|<r\}$ and $B_r'(x_0):= \{(x,0)\in \R^d \times \{0\}: \ |x-x_0|<r\}$ and with an abuse of notation, we will also denote $B_r^+(x_0)=B_r^+(x_0)\cap ((h\Z)^d\times\R_+)$ and $B_r'(x_0)=B_r'(x_0)\cap (h\Z)^d$ for $x_0 \in (h\Z)^d \times \{0\}$. If $x_0 =0$, we also omit the center point for the ease of notation.

If combined with trace inequalities and global assumptions on the function $u$, a quantitative estimate as in \eqref{eq:bdry_bulk} can be transferred into a boundary doubling inequality for the fractional Laplacian with $s= \frac{1}{2}$ and into the global unique continuation property \cite{BL15, GSU16, Rue21, RS20}. Equation \eqref{eq:bdry_bulk} thus provides a central quantitative unique continuation estimate for the half-Laplacian. 

Seeking to study the discrete fractional Laplacian and its rigidity and flexibility properties, we here investigate the analogous question on the lattice: Consider a function $\tilde{u}$ solving the equation
\begin{align}
\label{eq:CS_discr}
\begin{split}
(\p_t^2 + \D_{\dis}) \tilde{u} &= V \tilde{u} \mbox{ in } (h \Z)^d \times \R_+,\\
\tilde{u} & =   u \mbox{ on } (h \Z)^d \times \{0\},
\end{split}
\end{align}
where $V: (h \Z)^d \times \R_+ \rightarrow \R$ is a bounded potential. 

The observations from Theorem \ref{thm:exterior_UCP} imply that the estimate \eqref{eq:bdry_bulk} fails in general in the discrete setting. However, building on the results in \cite{BHLR10a, BHR10, GM13,GM14,LM15,LM17, FBRRS20}, we expect that it only fails up to exponentially small correction terms in the lattice spacing $h>0$ and that in the limit $h \rightarrow 0$ the result of the continuum is recovered. We prove that this is indeed the case:

\begin{thm}
\label{thm:boundary_bulk_discrete}
Let $u\in H^{1}((h\Z)^d)$ and $\tilde{u}: (h \Z)^d \times \R_+ \rightarrow \R$ be a solution to 
\eqref{eq:CS_discr}.
Then, there exist $h_0>0$, $C>1$ (depending on $\|V\|_{L^{\infty}((h\Z)^d \times \R_+)}$ and $d$) and $r_0\in (0,1)$, $\alpha \in (0,1)$ (depending only on $d$) such that for all $h\in (0,h_0)$ it holds that
\begin{align}
\label{eq:discrete_bdry_bulk}
\begin{split}
\|\tilde{u}\|_{L^2(B_{r_0}^+)} 
&\leq C \max\{\|\tilde{u}\|_{L^2(B_{1}^+)}, \|u\|_{H^1(B_1')} + \|\p_t \tilde{u}\|_{L^2(B_{1}')} \}^{1-\alpha} (\|u\|_{H^1(B_{1}')} + \|\p_t \tilde{u}\|_{L^2(B_{1}')})^{\alpha} \\
& \quad + C e^{-Ch^{-1}}\|\tilde{u}\|_{L^2(B_{1}^+)}.
\end{split}
\end{align}
\end{thm}

 Theorem \ref{thm:boundary_bulk_discrete} illustrates that the boundary-bulk unique continuation estimates only ``barely'' fail with correction terms that decay exponentially in the lattice size. This is analogous to the bulk doubling estimates from \cite{FBRRS20} (to which this could be reduced if additional vanishing assumptions for $u$ or $\p_t \tilde{u}$ were assumed). 
As in \cite{FBRRS20} a key ingredient towards obtaining this result is given by a robust Carleman inequality (see Theorem \ref{thm:Carl_disc}) which, in particular, allows us to treat equations with potentials. In studying unique continuation properties for (semi-)discrete elliptic equations, as in the continuum, alike in \cite{FBRRS20} and contrary to many earlier works on control theory for (semi-)discrete elliptic equations, a technical challenge arises in that we are forced to consider Carleman weight functions which are not convex but only \emph{pseudoconvex}.

As an application of the boundary-bulk doubling inequality from above, we prove that the global unique continuation property from Theorem \ref{prop:unique} persists in a certain sense if one is ``sufficiently close'' to the continuum setting and if global information on the data is present. To this end, we consider the inverse problem of recovering a function $f\in C_c^{\infty}(W)$ from partial measurements of its half-Laplacian $(-\D_{\dis})^{\frac{1}{2}}f|_{\Omega}$ on the open domain $\Omega$ which we assume to be disjoint from the open set $W$ (see Figure \ref{fig:inverse}). 
Although the global UCP fails in the discrete setting, we can use the observations from Theorem \ref{thm:boundary_bulk_discrete} to still infer a stability estimate for this inverse problem:

\begin{thm}
\label{thm:application}
Let $d\geq 1$ and let $W, \Omega \subset (h\Z)^d$ be non-empty, open sets with $\overline{W}\cap \overline{\Omega} = \emptyset$. Let $h_0>0$ be as in Theorem \ref{thm:boundary_bulk_discrete} and assume that $h\in (0,h_0)$.
Let $f\in C_c^{\infty}(W)$. Then there exists $\nu \in (0,1)$ (depending only on the dimension $d$ and the domains $W,\Omega$) such that if
\begin{align}
\label{eq:continuum_aprox}
0<h_0 \leq 10^{-1}  \Big|\log\Big(\frac{\|(-\D_{\dis})^{\frac{1}{2}} f\|_{L^2(\Omega)}}{\|f\|_{H^1(W)}} \Big) \Big|^{-1+\nu}\Big| \log\Big(-C \log\Big( \frac{\|(-\D_{\dis})^{\frac{1}{2}} f\|_{L^2(\Omega)}}{\|f\|_{H^1(W)}}  \Big) \Big) \Big|^{-1},
\end{align}
the following estimate holds:
\begin{align*}
\|f\|_{L^2(W)}& \leq C \|f\|_{H^1(W)} \Big|\log\Big(\frac{\|(-\D_{\dis})^{\frac{1}{2}} f\|_{L^2(\Omega)}}{\|f\|_{H^1(W)}} \Big) \Big|^{-\nu} \\
& \quad +  C\exp\Big(- (Ch)^{-1} \Big|\log\Big(\frac{\|(-\D_{\dis})^{\frac{1}{2}} f\|_{L^2(\Omega)}}{\|f\|_{H^1(W)}} \Big) \Big|^{-1+\nu} \Big)\|f\|_{H^1(W)}.
\end{align*} 
\end{thm}

\begin{rmk}
Theorem \ref{thm:application} provides a stability estimate (and thus also a uniqueness result) for the linear inverse problem of recovering $f$ from the partial data $(-\D_{\dis})^{\frac{1}{2}}f$ in a regime which is ``sufficiently close to the continuum''. It can be viewed as a linear analogue of the estimates from \cite{RS20,Rue21} and a higher-dimensional, discrete analogue of the one-dimensional linear bounds from \cite{GFR20} (see also \cite{Rue17,RS20a}).
Indeed, since the global unique continuation property fails, we do not expect to be able to uniquely recover $f$ from $(-\D_{\dis})^{\frac{1}{2}}f|_{\Omega}$ for arbitrary choices of $W,\Omega$. However, under condition \eqref{eq:continuum_aprox} on the size of the data, its oscillation and the discrete scale $h_0>0$, the problem is ``sufficiently close to the continuum'': In this case the continuum stability holds up to a correction term. 

A main observation here is the \emph{uniformity} of the estimate in the parameter $h>0$. While for fixed $h>0$ only finitely many degrees of freedom are present and thus Lipschitz stability results would be available \cite{AV05, BDHQ13,RSi19,S07,AdHGS17,AS22}, these would deteriorate exponentially in the limit $h\rightarrow 0$ \cite{R06}. In order to capture the transition from the discrete to the continuum, it is thus central to obtain estimates as in Theorem \ref{thm:application} which are \emph{uniform} in $h>0$ and which allow to pass to the continuum limit.

In the continuum, the linear inverse problem under investigation would be severely ill-posed \cite{RS18} leading to an estimate as given by the first right hand side term in the bound from Theorem~\ref{thm:application}. The bound in Theorem \ref{thm:application} thus shows that this bound remains \emph{uniformly} valid in the discrete set-up and for the limit $h\rightarrow 0$ up to a correction term which is dependent on the lattice size and the data and which is encoded in the second right hand side contribution of the estimate. As the lattice size decreases, it approaches the continuum estimate \emph{exponentially} in the lattice spacing.
\end{rmk}

We remark that in particular, for $h_0\in (0,1)$ sufficiently small (depending on the size of the measurement data $\|(-\D_{\dis})^{\frac{1}{2}}f\|_{L^2(\Omega)}$ and the oscillation of $f$ measured in terms of $\|f\|_{H^1(W)}$), we obtain a stability estimate for the discrete inverse problem of recovering $f \in C_c^{\infty}(W)$ from the data $\|(-\D_{\dis})^{\frac{1}{2}}f\|_{L^2(\Omega)}$ under a priori oscillation control for $f$. In the limit $h_0 \rightarrow 0$ this matches the analogous continuous estimates (see, for instance, \cite{GFR20, RS20}). We hope that this eventually also allows one to obtain similar stability estimates for nonlinear discrete inverse problems such as the discrete fractional Calder\'on problem \cite{RS20}. We refer to \cite{E11} for such estimates for the discrete classical Calder\'on problem. It is also interesting to notice Lipschitz and logarithmic stability results for the inverse problem of recovering a potential in a semi-discrete wave equation \cite{BEO15}.

\subsection{Fractional Laplacian on the discrete torus: transference and global unique continuation}

We also explore the problem of global unique continuation for powers of the Laplacian on the discrete torus. The discrete torus on a mesh of size $h>0$ is defined as the set of points 
$$
A_{h,d}^N:=(h\Z_{2N+1})^d:=(h\{-N,\ldots,N\})^d
$$ 
and the Laplacian on the discrete torus is given by
\begin{equation}
\label{eq:deltaTorus}
\Delta_{A_{h,d}^N}u_j:=\frac{1}{h^2}\sum_{i=1}^d\big(u(h(j+e_i))-2u(hj)+u(h(j-e_i))\big), \quad j\in \{-N,\ldots, N\}^d.
\end{equation}
We will use the notation $A_h^N$ for $A_{h,1}^N$.

In Subsection \ref{subsec:transf} we will show a transference formula which allows us to obtain the pointwise formula for the fractional Laplacian on the discrete torus from the fractional Laplacian on $(h\Z)^d$ and to relate the corresponding extension problems. Such an identity and the consequences are very much in the spirit of  \cite{RS_trans}.

\begin{thm}[Pointwise formula for the discrete torus]
\label{thm:transP}
Let $v\in \ell_s(A_{h,d}^N)$ and $h=\dfrac{2\pi}{2N+1}$. Then
$$
(-\Delta_{A_{h,d}^N})^sv_j=\sum_{\substack{m\in \{-N,\ldots,N\}^d\\
m\neq j}}\big(v_j-v_m\big)K_s^{A_{h,d}^N}(j-m), \quad j\in \{-N,\ldots, N\}^d,
$$
where, for $j\in \{-N,\ldots,N\}^d$, $j\neq (0,\ldots,0)$,
$$
K_s^{A_{h,d}^N}(j)=\sum_{k\in \Z^d}K_s^h(j+k(2N+1)).
$$
In particular, for the case $d=1$, we have
$$
(-\Delta_{A_{h}^N})^sv_j=\sum_{\substack{m\in \{-N,\ldots,N\}\\
m\neq j}}\big(v_j-v_m\big)K_s^{A_{h}^N}(j-m), \quad j\in \{-N,\ldots, N\},
$$
where, for $j\in \{-N,\ldots,N\}$, $j\neq 0$,
$$
K_s^{A_{h}^N}(j)=\frac{4^{s}\Gamma(1/2+s)}{h^{2s}\sqrt{\pi}|\Gamma(-s)|}\sum_{k=-\infty}^\infty \frac{\Gamma(|j+k(2N+1)|-s)}{\Gamma(|j+k(2N+1)|+1+s)}.
$$
\end{thm}

As a further application we deduce the validity of the Caffarelli--Silvestre extension property on the torus:

\begin{thm}[Extension problem for the discrete torus]
\label{thm:transEO}

Let $v\in \operatorname{Dom}((-\Delta_{A_{h,d}^N})^s)$ and $h=\dfrac{2\pi}{2N+1}$. Let $V=V(hj,t)$ be the solution to the boundary problem
$$
\begin{cases}
\Delta_{A_{h,d}^N}V+\frac{1-2s}{t}V_t+V_{tt}=0, \quad &\text{ in } A_{h,d}^N\times (0,\infty),\\
V(hj,0)=v(hj), \quad &\text{ on } A_{h,d}^N. 
\end{cases}
$$
Then, for $c_s=\frac{4^{s}\Gamma(1/2+s)}{h^{2s}\sqrt{\pi}|\Gamma(-s)|}>0$, we have that 
\begin{equation}
\label{eq:DtoN}
-\lim_{t\to 0^+}t^{1-2s}V_t(hj,t)=c_s(-\Delta_{A_{h,d}^N})^sv(hj), \quad j\in \{-N,\ldots, N\}.
\end{equation}
\end{thm}

The pointwise formula in Theorem \ref{thm:transP}  allows us to obtain, for $d=1$, the following result concerning the failure of global unique continuation, which can be seen as an analogous result to Theorem~\ref{thm:exterior_UCP}, illustrating a similar failure of the global unique continuation property in the discrete setting.

\begin{thm}
\label{prop:UCP_torus_new}
Let $h \in \R_+$ and $N\in \N$. Let $X \subset A_h^N$ be a finite set of cardinality $M\le 2N+1$. If $M\le N$, there exists a 
non-zero bounded function $u=\{u_j\}_{j=-N}^N$ such that $u_j=0=(-\D_{A_h^N})^s u_j $ for $j\in X$. 
\end{thm}

\subsection{Outline of the article}

The following sections are organized as follows: After briefly recalling the definition and some of the main properties of the fractional Laplacian in Section \ref{sec:pre}, in Section \ref{sec:globalUC} we first give examples of the failure of the global UCP for the fractional discrete Laplacian. Here we also discuss how the stronger a priori information from Theorem \ref{prop:exterior_UCP1} allows us to recover the global UCP. In Section \ref{sec:wucp} we extend these examples to examples in slab domains for which the weak UCP for the fractional Schrödinger equation fails. After these illustrations of the effects of discretization, in Section \ref{sec:bdry_bulk} we show that quantitative unique continuation estimates only fail up to exponentially small (in the lattice size) corrections. It is in this section that we also discuss the linear inverse problem from Theorem \ref{thm:application} and the claimed stability estimate.
 Finally, in Section \ref{sec:torus} we discuss transference principles and the above results on the torus.

\medskip
\noindent{\textbf{Acknowledgments}.}
We are indebted to Diana Stan for helpful discussions at various stages of the project.
\smallskip

\textbf{A.\ Fern\'andez-Bertolin} was partially supported by ERCEA Advanced Grant 2014 669689 - HADE, 
project PGC2018-094528-B-I00 (AEI/FEDER, UE) and acronym “IHAIP”, and
the Basque Government through the project IT1247-19.
\textbf{L.\ Roncal} was supported by the Spanish Government through the projects SEV-2017-0718, PID2020-113156GB-100,
funded by MCIN/AEI /10.13039/501100011033 and by FSE ``invest in your future'' and RYC2018-025477-I, 
and by the Basque Government through the BERC 2018-2021 program. She also acknowledges IKERBASQUE fundings.
\textbf{A.\ R\"uland} was supported by the German Research Foundation (DFG) under Germany’s Excellence Strategy -- EXC-
2181/1-390900948 (the Heidelberg STRUCTURES Cluster of Excellence).

\section{Preliminaries}
\label{sec:pre}

Before turning to the unique continuation results from the introduction, we recall the precise definition and some of the properties of the fractional discrete Laplacian.

\subsection{Definition by means of the heat kernel and an explicit representation formula}
\label{sub:heat}
We begin by relating the fractional Laplacian and the semidiscrete heat equation: A formal solution $w_j(t)$ to the semidiscrete heat equation 
$$
\partial_tw_j= \Delta_{\dis}w_j, \,\, \text{ in } (h\Z)^d\times (0,\infty),\qquad w_j(0)=u_j, \,\, \text{ on } (h\Z)^d,
$$ 
is given by 
$$
e^{t\Delta_{\dis}}u_j=\sum_{m\in \Z^d}G\Big(j-m,\frac{t}{h^2}\Big)u_m,\quad t>0,\,\, hj\in (h\Z)^d
$$
with 
\begin{equation}
\label{eq:heatn}
G(m,t)=e^{-2dt}\prod_{i=1}^dI_{m_i}(2t), \quad m=(m_1, \ldots, m_d)\in \Z^d.
\end{equation}
Here, $I_{k}$ is the modified Bessel function of the first kind and order $m_i$, defined as 
$$
I_k(t)=\sum_{\ell=0}^\infty\frac{1}{\ell!\Gamma(\ell+k+1)}\Big(\frac{t}{2}\Big)^{2\ell+k}.
$$ 
From the very definition of $I_k$ we deduce that $I_k(t)\ge0$, for every $k\in \Z$ and $t\ge 0$. We refer to \cite[Chapter 5]{Lebedev} and \cite{OlMax} for the properties of $I_k$. Observe that $G(-m,t)=G(m,t)$, since $I_{-m_i}=I_{m_i}$ (and even $G((m_1,\ldots, -m_i,\ldots,m_d),t)=G((m_1,\ldots, m_i,\ldots,m_d),t)$, componentwise). Moreover, $\sum_{m\in \Z^d}G(m,t)=1$. 
The above facts follow from a generalization of the one dimensional case in \cite{CGRTV17}.

With this, we can define the positive fractional powers of the discrete Laplacian, namely, for $0<s<1$,
\begin{align}
\label{eq:fracn}
\notag
(-\Delta_{\dis})^{s}u_{j}&= \frac{1}{\Gamma(-s)}
\int_0^{\infty} \big(e^{t\Delta_{\dis}}u_{j} -u_j\big)\frac{dt}{t^{1+s}}\\
\notag&=\frac{1}{\Gamma(-s)}
\int_0^{\infty}\sum_{m\in \Z^d, m\neq j}G\Big(j-m,\frac{t}{h^2}\Big)\big(u_m-u_j\big)\frac{dt}{t^{1+s}}\\
&=\sum_{m\in \Z^d, m\neq j}\big(u_j-u_m\big)K_s^h(j-m),
\end{align}
 where the discrete kernel $K_s^h$ is given by 
\begin{equation}
\label{Ksh}
K_s^h(m)=\frac{1}{h^{2s}}\frac{1}{|\Gamma(-s)|}
\int_0^{\infty}G(m,t)\frac{dt}{t^{1+s}}=\frac{1}{h^{2s}}\frac{1}{|\Gamma(-s)|}
\int_0^{\infty}e^{-2dt}\prod_{i=1}^dI_{m_i}(2t)\frac{dt}{t^{1+s}},
\end{equation}
for $m\neq (0,\ldots,0)$ and $K_s^h(0,\ldots,0)=0$.
 In particular, the kernel has an even symmetry, in the sense that $K_s^h(m_1,\ldots, m_i,\ldots, m_n)=K_s^h(m_1,\ldots, -m_i,\ldots, m_n)$, for $i=1,\ldots,d$. Note also that $K_s^h(m)>0$ for $m\neq (0,\ldots,0)$. 
 
 In constructing examples for the failure of the global and weak unique continuation properties, we will rely on this representation formula.

When $d=1$, it was shown in \cite[Theorem 1.1]{CRSTV16} that the kernel $K_s^h$ has an explicit expression, namely
\begin{equation}
\label{eq:kernelKs}
K_s^h(m)= \frac{4^{s}\Gamma(1/2+s)}{\sqrt{\pi}|\Gamma(-s)|} \frac{\Gamma(|m|-s)}{h^{2s}\Gamma(|m|+1+s)}\quad \text{ for } \quad m \in \Z \setminus \{ 0\}\quad \text{ and } \quad  K_s^h(0,\ldots,0)=0.
\end{equation}
Bear in mind the well-known asymptotics for the ratio of two Gamma functions (see for instance \cite[Chapter 4, (5.05)]{Olver}): for $z\in \mathbb{R}$, $z\to \infty$, $
\frac{\Gamma(z+a)}{\Gamma(z+b)}\sim z^{a-b}$, so for $m\in \Z$,
\begin{equation}
\label{eq:asymG}
K_s^h(m)\sim C_{s}h^{-2s}|m|^{-2s-1}, \quad |m|\to \infty.
\end{equation}

\subsection{Function spaces}

In the sequel, we will consider solutions to the fractional Laplacian in different function spaces. 

Mimicking the corresponding continuum setting, we introduce the following spaces: For $0\le s \le 1$ we let
$$
\ell_{\pm s}:=\Big\{u: (h\Z)^d \rightarrow \R: \ \|u\|_{\ell_{\pm s}} := \sum\limits_{m\in \Z^d} \frac{|u_m|}{(1+|m|)^{d\pm 2s}}<\infty\Big\}.
$$
As in the continuum, these are function spaces which allow us to define the fractional Laplacian under minimal decay conditions. 

We use $\mathcal{S}((h\Z)^d)=:\mathcal{S}$ to denote the Schwartz functions on the discrete lattice $(h \Z)^d$.
The fractional discrete Laplacian on $(h\Z)^d$ does not preserve the Schwartz class, but instead $(-\D_{\dis})^s: \mathcal{S}\to \mathcal{S}_s$, where
$$
\mathcal{S}_s:=\{\varphi\in C^{\infty}((h\Z)^d): (1+|l|)^{-(d+2s)}\nabla_{\dis}^k\varphi_l\in \ell^{\infty}_{(h\Z)^d}, k\in \N_0\}
$$
where $\nabla_{\dis}$ denotes the symmetric discrete gradient and $ \ell^{\infty}_{(h\Z)^d}$ is the $L^{\infty}$ norm on the lattice $(h\Z)^d$. This can be proven analogously as in \cite[pp. 72--73]{Sil07}. The symmetry of the fractional discrete Laplacian allows us to define $(-\D_{\dis})^s$ for $u$ in the dual space $\mathcal{S}_s'$. For summable functions $u\in \mathcal{S}_s'$ we let
$$
\langle (-\D_{\dis})^su,\varphi\rangle_{\mathcal{S}_s}:=\sum_{k\in \Z^d}u_k(-\D_{\dis})^s\varphi_k, \quad \varphi\in \mathcal{S}((h\Z)^d).
$$
The sum above is absolutely convergent when $u\in \ell_s$.

In addition to these function spaces, we also use standard Sobolev spaces. We define these through the Fourier transform:
Given $u:(h\Z)^d\to \R$, its Fourier transform is a function defined on $(h^{-1}\T)^d:=[-\pi/h,\pi/h)^d$ whose Fourier coefficients are given by the sequence $\{u_k\}_{k\in \Z^d}$. In other words, if $u\in \ell^1_{(h\Z)^d}$ then we define
$$
\mathcal{F}_{(h\Z)^d}u(\xi)=\sum_{k\in \Z^d}u_ke^{-i\xi\cdot k}, \quad \xi\in 
[-\pi/h,\pi/h)^d.
$$
Now, building on this, for $u: (h\Z)^d \rightarrow \R$ and $r\in \R$ we define
\begin{align}
\label{eq:Hr}
\|u\|_{H^r((h\Z)^d)} := \|(1+|\nabla_{\dis}|^2)^{\frac{r}{2}}u\|_{L^2((h\Z)^d)} = \|\mathcal{F}_{(h\Z)^d}^{-1}(1+h^{-2}\sum_{k=1}^d|\sin(h\xi_k)|^2)^{\frac{r}{2}}\mathcal{F}_{(h\Z)^d}u\|_{L^2((h\Z)^d)},
\end{align}
where $\nabla_{\dis}$ denotes the symmetric discrete gradient and
\begin{align*}
H^{r}((h\Z)^d):= \overline{(C_c^{\infty}(h\Z)^d)}^{\|\cdot\|_{H^{r}}}.
\end{align*}

\subsection{The discrete Caffarelli--Silvestre extension}

Similarly as in the continuum setting, the fractional discrete Laplacian can also be shown to be related to a Caffarelli--Silvestre type extension.
This observation is just an application of the
general extension problem of \cite{ST10}, see also \cite{CRSTV16}.
Given $u\in\operatorname{Dom}((-\D_{\dis})^s)$, $s\in (0,1)$, the semidiscrete function $\tilde{u}$ defined as
$$
\tilde{u}_j(t)= \frac{t^{2s}}{4^s \Gamma(s)} \int_0^\infty e^{-t^2/(4z)}e^{z\D_{\dis}}u_j\,\frac{dz}{z^{1+s}}
$$
is the unique solution (weakly vanishing as $t\to\infty$) to the Dirichlet problem
\begin{align}
\label{eq:discre_frac_semi_CF}
\begin{split}
(\p_t t^{1-2s} \p_t + t^{1-2s} \D_{\dis} )\tilde{u} &= 0 \mbox{ in } (h\Z)^d \times \R_+,\\
\tilde{u} & = u \mbox{ on }  (h\Z)^d \times \{0\}.
\end{split}
\end{align}
Moreover,
$$
-\lim_{t\rightarrow 0^+}t^{1-2s}\partial_t\tilde{u}_j(t)=
-2s\lim_{t\to0^+}\frac{\tilde{u}_j(y)-\tilde{u}_j(0)}{t^{2s}}=
\frac{\Gamma(1-s)}{4^{s-1/2}\Gamma(s)}(-\D_{\dis})^su_j.
$$

In Section \ref{sec:bdry_bulk} we will use quantitative estimates for the Caffarelli--Silvestre extension for the special case $s=\frac{1}{2}$ (in which case the Caffarelli--Silvestre extension turns into the harmonic extension). To this end, we mainly rely on energy estimates and thus use the following notation: For a set $E\subset \R^{d+1}_+$ we consider
\begin{align}
\label{eq:Sob}
\|\tilde{u}\|_{H^1(E)} := \|\nabla_{\dis} \tilde{u}\|_{L^2(E)} + \|\p_t \tilde{u} \|_{L^2(E)} + \|\tilde{u}\|_{L^2(E)}.
\end{align}

Further,  we will use a number of auxiliary results for the discrete Caffarelli--Silvestre extension, which we thus briefly collect here. Most of these are proved analogously as their continuous counterparts, hence we only present the basic ideas of their proofs.

\begin{lem}
\label{lem:L2_bd}
Let $\tilde{u}$ be a solution to \eqref{eq:discre_frac_semi_CF} with $s=\frac{1}{2}$ and $u\in C_c^{\infty}((h\Z)^d)$ with $K:=\supp(u)\subset (h\Z)^d$. Then, for any compact set $E:=E_{d} \times [0,C_{d+1}]$ where $E_d \subset (h\Z)^d$ and $C_{d+1}>0$ there exists a constant $C=C(d,K, E_d, C_{d+1})>1$ such that
\begin{align*}
\|\tilde{u}\|_{L^2(E)} \leq C \|u\|_{H^1(K)}.
\end{align*} 
\end{lem}

Here the Sobolev spaces in the upper half-space are defined similarly as in \eqref{eq:Sob}, containing a mixture of standard Sobolev spaces in the normal direction and the ones from the lattice in the tangential direction, see \eqref{eq:Hr}.

\begin{proof}
By mollification arguments, we may assume that $u$ and $\tilde{u}$ are arbitrarily regular. From energy estimates we obtain that there exists $C>1$ such that
\begin{align}
\label{eq:energy}
\|\nabla_{\dis} \tilde{u}\|_{L^2(E)} + \|\p_t \tilde{u}\|_{L^2(E)} \leq C \|u\|_{H^1((h\Z)^d)}.
\end{align}
Next, by the fundamental theorem of calculus in the normal directions, we obtain that
\begin{align}
\label{eq:aux1}
\begin{split}
\|\tilde{u}\|_{L^2(E)} 
& \leq C\|\p_t \tilde{u}\|_{L^2(E_d \times [0,2C_{d+1}])} + C\|\tilde{u}\|_{L^2(E_d \times [C_{d+1},2 C_{d+1}])}.
\end{split}
\end{align}
By the fundamental theorem in tangential directions, we further obtain that
\begin{align}
\label{eq:aux2}
\|\tilde{u}\|_{L^2(E_d \times [C_{d+1},2 C_{d+1}])} \leq C\|\nabla_{\dis} \tilde{u}\|_{L^2(\text{conv}(E_d,K)\times [C_{d+1},2C_{d+1}])} + C\|\tilde{u}\|_{L^2(K \times [C_{d+1},2C_{d+1}])}.
\end{align}
Finally, applying the fundamental theorem in normal directions again, we obtain that
\begin{align}
\label{eq:aux3}
\|\tilde{u}\|_{L^2(K \times [C_{d+1},2C_{d+1}])} \leq C \|\p_t \tilde{u}\|_{L^2(K \times [0,2C_{d+1}])} + C\|u\|_{L^2(K)}.
\end{align}
Inserting \eqref{eq:aux2}-\eqref{eq:aux3} into \eqref{eq:aux1}, using \eqref{eq:energy} and the support assumption for $u$ the claim follows.
\end{proof}

Next, we recall a trace estimate (see for instance, \cite[Lemma 2.5]{CR20} for an argument in the continuum).

\begin{lem}
\label{lem:trace_aux}
Let $\tilde{u} \in \dot{H}^1((h\Z)^d \times \R_+)$ and $u:=\tilde{u}|_{(h\Z)^d \times \{0\}}$.
Let $u \in L^2(\R^d)$ and $\tilde{u}$ a solution to \eqref{eq:discre_frac_semi_CF} with $s= \frac{1}{2}$. Let $W\subset (h\Z)^d$ be an open, bounded set and let $\mu \in (0,1)$. Then,
\begin{align*}
\|u\|_{L^2(W)} \leq C \mu^{\frac{1}{2}} \|\p_t \tilde{u}\|_{L^2(W \times [0,2\mu])} + C \mu^{-\frac{1}{2}} \|\tilde{u}\|_{L^2(W \times [\mu,2\mu])}.
\end{align*} 
\end{lem}

\begin{proof}
Without loss of generality, we assume that $u$ is arbitrarily regular.
The argument then follows from the fundamental theorem of calculus. Indeed, for $x\in (h\Z)^d$,
\begin{align*}
u(x,0) = \int_{0}^s \p_t \tilde{u}(x,r)dr + u(x,s).
\end{align*}
Integrating in $s\in [0,2\mu]$ we obtain
\begin{align*}
\mu |u(x,0)| \leq C \mu  \int_{0}^{2\mu} |\p_t \tilde{u}(x,r)|dr + \|u(x,\cdot)\|_{L^1([\mu,2\mu])}. 
\end{align*}
With the Cauchy-Schwarz inequality and after dividing by $\mu>0$ we obtain
\begin{align*}
 |u(x,0)| \leq C \mu^{\frac{1}{2}}  \|\p_t \tilde{u}(x,\cdot)\|_{L^2([0,2\mu])} + C \mu^{-\frac{1}{2}}\|u(x,\cdot)\|_{L^2([\mu,2\mu])}. 
\end{align*}
Squaring this expression and integrating in $x\in W$ we conclude the desired result.
\end{proof}

\subsection{Notation}

Finally, we collect some further notational conventions which will be used throughout the article.

We will use the following set notation:
\begin{itemize}
\item For some $r>0$ and $x_0 \in \R^d$, the set $B_{r}(x_0)$ we will refer to the set $\{x\in \R^d: \ |x-x_0|<r\}\cap (h\Z)^d$.
\item In the case that we are working in the upper half-plane, for some $r>0$ and $x_0 \in \R^{d+1}_+$, the set $B_{r}^+(x_0)$ we will refer to the set $\{x\in \R^{d+1}_+: \ |x-x_0|<r\}\cap ((h\Z)^d \times \R_+)$.
\item In the case that we are working in the upper half-plane, for some $r>0$ and $x_0 \in \R^{d}_+$, the set $B_{r}'(x_0)$ we will refer to the set $\{x\in \R^{d}: \ |x-x_0|<r\}\cap ((h\Z)^d \times \{0\})$.
\item For a point $(x,t)\in (h\Z)^d \times \R_+$ we use the notation $|(x,t)|:= \big(|x|^2 + t^2 \big)^{\frac{1}{2}}$ and $|x|= \big(\sum\limits_{j=1}^{d}x_j^2\big)^{\frac{1}{2}}$ to refer to its Euclidean norm.
\item We denote by $\ell^p_{(h\Z)^d}$ the $L^p$ norms on the lattice $(h\Z)^d$, for $1\le p\le \infty$.
\end{itemize}
In all these notations, we will also omit the center point if $x_0 =0$ for convenience of notation.

\section[Fractional Discrete  Unique Continuation]{Failure of Global Unique Continuation and Global UCP from the Exterior}
\label{sec:globalUC}

We begin by discussing the global unique continuation property of the fractional discrete Laplacian. On the one hand, using the nonlocal summation formula for the fractional Laplacian, we show that, in contrast to the continuous operator, the fractional discrete Laplacian does not enjoy the global unique continuation property. On the other hand, we prove that with a ``global unique continuation property from the exterior'' a remnant of the global unique continuation properties of the fractional Laplacian persists on the level of the discrete operator. 

\subsection{Failure of global unique continuation}
\label{sub:failuregUC}

We first discuss the failure of the global unique continuation property for the discrete fractional Laplacian. 

\begin{proof}[Proof of Theorem \ref{thm:exterior_UCP}]
If $j\in X$, the requirement $u_j=0$ implies
$$
(-\D_{\dis})^s u_j=\sum_{\substack{m\in \Z^d \\ m \neq j}} u_m K_s^h(j-m).
$$
On the one hand, we compute $(-\D_{\dis})^s u_j$ for all $j\in X$, which yields $M$ equations. On the other hand, we set $u_j=0$ for all $j\in \Z^d\setminus N$, where $N$ is a set of cardinality $M+1$ such that $N\cap X=\emptyset$. The requirement that $(-\D_{\dis})^su_j=0$ for all $j\in X$ gives a homogeneous system of $M$ equations and for the $M+1$ unknown values of $u_j$ for $j \in Y$. Since there are more unknowns than equations, there are infinitely many nontrivial solutions.
\end{proof}

\subsection{Global behaviour for exterior domains}
\label{sec:global}

In \cite[Corollary 5.2]{RS20a}, building on the idea from \cite[Lemma 3.5.4]{Isakov}, it was shown that a large class of nonlocal operators are \emph{antilocal from the exterior}. This was based on using the singularity (in the complex plane) of the symbol of the operator. For instance, the symbol $|\xi|^s$ of the fractional Laplacian $(-\D)^s$ has a branch cut singularity in the complex plane. A similar argument is given in \cite{RSV19} in the context of the qualitative fractional Landis conjecture. We also refer to \cite{L82, GFR20}.

In this section we prove that, in contrast to the failure of the global UCP with \emph{local} assumptions, this behaviour persists for the fractional discrete Laplacian as stated in Theorem \ref{prop:exterior_UCP1}. Notice that a similar conclusion holds under suitable decay conditions for $u$ and $(-\D_{\dis})^s u$: If we assume that $|u(x)| + |(-\D_{\dis})^s u(x)|\leq C e^{-C|x|^{1+}}$ for $|x|\geq R$, then the conclusion of Theorem~\ref{prop:exterior_UCP1} remains true. We refer to \cite{RSV19} and the argument below.

\begin{proof}[Proof of Theorem \ref{prop:exterior_UCP1}]
The proof follows the same ideas as in \cite{RS20a} and \cite{RSV19}, reducing the claim to the Paley--Wiener theorem and the presence of a branch cut in the symbol for $(-\D_{\dis})^s$. To this end, we note that the symbol of $(-\D_{\dis})^s$ is given by
\begin{align*}
h^{-2s}\Big(\sum_{j=1}^{d} \big( 2 - e^{i \xi_j h}-e^{-i \xi_j h} \big) \Big)^{s}=h^{-2s}\Big(\sum_{j=1}^d 4\sin^2(h \xi_j/2)\Big)^s .
\end{align*} 
Further, for contradiction, assume that $u \neq 0$ in $(h\Z)^d$.

By the compact support assumptions in the theorem, the functions 
$$
\F_{(h\Z)^d} u ,\quad \F_{(h\Z)^d} (-\D_{\dis})^s u: (h^{-1} \mathbb{T})^d \rightarrow \R
$$ 
are analytic (and non-trivial since $u \neq 0$ by assumption). They have analytic extensions to $([-\pi/h, \pi/h]\times \R)^d \subset \C^d$ (where we have identified $\T$ with $[-\pi,\pi]$ and $\R^2$ with $\C$). 

By assumption, we know that $u \neq 0$; thus in particular, there exists $\xi' \in \R^{d-1}$ such that $\F_{(h\Z)^d} u (\xi',\cdot) \neq 0$ as an analytic function of $\xi_d \in \R$.
For fixed $\xi' \in \R^{d-1}$ as above, we consider the symbol
\begin{align*}
p_{h,\xi'}(\xi_d):=h^{-2s}\Big(C(\xi') + 4\sin^2(h \xi_d/2) \Big)^s,
\end{align*}
where $C(\xi'):=\sum_{j=1}^{d-1} 4\sin^2(h \xi_j/2) \in \R$. We next observe that 
the function $x\mapsto \sin^2(h x)$ map $[-\pi/h, \pi/h]\times \R \subset \C$ surjectively into $ \C$. Consequently, any realization of the extension of $p_{h,\xi'}(\xi_d)$ into $[-\pi/h, \pi/h]\times \R \subset \C$ must have a branch cut. However, by the analyticity of $\F_{(h\Z)^d} u(\xi',\cdot)$ and $\F_{(h\Z)^d} (-\D_{\dis})^s u(\xi',\cdot)$ as functions of $([-\pi/h, \pi/h]\times \R) \subset \C$ we hence infer that
\begin{align*}
p_{h,\xi'}(\xi_d)\F_{(h\Z)^d} u(\xi', \xi_n) = g(\xi', \xi_d)
\end{align*}
on $([-\pi/h, \pi/h]\times \R) \subset \C$ for some analytic function $g: ([-\pi/h, \pi/h]\times \R) \subset \C\rightarrow \C$. Equivalently, we obtain that  
\begin{align*}
p_{h,\xi'}(\xi_d) = \frac{g(\xi',\xi_d)}{\F_{(h\Z)^d} u (\xi',\xi_d)}.
\end{align*}
While the right hand side of this equation is a meromorphic function (recall that by our contradiction assumption $\F_{(h\Z)^d} u (\xi', \cdot) \neq 0$), the left hand side of this equation has a branch cut, and, in particular, cannot be meromorphic. This yields a contradiction and implies that $u \equiv 0$.
\end{proof}

With a similar argument, we also obtain that the unique continuation property holds from half-spaces as formulated in Proposition \ref{prop:exterior_UCP}.

\begin{proof}[Proof of Proposition \ref{prop:exterior_UCP}]
We first note that we may assume that $u\in H^{d}((h\Z)^d)$. Indeed, this follows by convolution with a smooth convolution kernel and by the linearity of the assumed vanishing condition (possibly after a shift in the $x_d$ direction).
Since $u = 0 = (-\D_{\dis})^s u \mbox{ in } \{x\in (h\Z)^d: \ x_{d}\geq 0\}$, the functions $\F_{(h\Z)^d} u$ and $\F_{(h\Z)^d}((-\D_{\dis})^s u)$ extend as holomorphic functions in the negative complex half-plane. Moreover, by the explicit Fourier representation of the solution and the regularity and integrability of $u, (-\D_{\dis})^s u$, we obtain that the analytic continuations of $\F_{(h\Z)^d} u$ and $\F_{(h\Z)^d}((-\D_{\dis})^s u)$ into the negative complex half-plane are continuous up to the boundary $(h^{-1}\T)^d \times \{0\}$. Again, if one  does not already have that $u\equiv 0$, this yields a contradiction, since the analytic continuation of the function $|\xi|^{s}$ necessarily has a branch cut in the lower complex half-plane.
\end{proof}

\section[Counterexamples to the WUCP]{Counterexamples for the weak UCP for discrete Schrödinger equations with bounded potentials}
\label{sec:wucp}
In this section, we present with the proof of Proposition \ref{prop:slab_domain} the failure of the weak unique continuation property for the fractional discrete Laplacian in a (thin) slab domain. We emphasize that, contrary to the previous construction, we here additionally require the validity of an equation in the whole lattice.
The result of Proposition \ref{prop:slab_domain}, in particular, shows that a ``naive" version of the results from \cite{Rue15} does not longer hold in the discrete set-up.

In what follows, we reduce the result of Proposition \ref{prop:slab_domain} to the following one-dimensional auxiliary result:

\begin{lem}
\label{lem:1D_disc_Schr}
Let $s\in (0,1)$. 
Then there exist non-trivial sequences $\{u_j\}_{j\in \Z} \in \ell_s$ and $\{V_j\}_{j\in \Z} \in \ell^{\infty}_{h\Z}$ (the latter size depending on h) such that
\begin{align*}
(-\D_{\dis})^s u_j = V_j u_j \mbox{ for all } j \in \Z \mbox{ and }
u_j = 0 \mbox{ for all } j \in\{-1,0,1\}. 
\end{align*}
\end{lem}

Analogously as in Remark~\ref{rmk:wucp} it would also be desirable to extend the example from Lemma~\ref{lem:1D_disc_Schr} to arbitrarily thick sets of vanishing of $u$. Here similar issues as the ones outlined in Remark~\ref{rmk:wucp}(ii) would have to be dealt with.

\subsection{Proof of Lemma \ref{lem:1D_disc_Schr}}

Before turning to the proof of Proposition \ref{prop:slab_domain}, we discuss the auxiliary result from Lemma \ref{lem:1D_disc_Schr}.

\begin{proof}[Proof of Lemma \ref{lem:1D_disc_Schr}]
We start by considering the following auxiliary sequence 
\begin{align*}
\tilde{u}_j := 
\left\{
\begin{array}{ll}
1 &\mbox{ if } j \geq 2,\\
-1 &\mbox{ if } j \leq -2,\\
0 &\mbox{ if } j \in \{-1,0,1\}.
\end{array} \right.
\end{align*}
Then by symmetry we have that $(-\D_{\dis})^s \tilde{u}_0 = 0$. For $j\geq 2$ (and by symmetry similarly for $j\leq -2$) we have
\begin{align*}
(-\D_{\dis})^s \tilde{u}_j = \sum\limits_{m\in \Z, m \leq -2} 2 K_s^h(j-m) + \sum\limits_{m=-1}^1 K_s^h(j-m),
\end{align*}
where the kernel $K_s^h$ is as in \eqref{eq:kernelKs}.
Due to the summability of $K_s^h$ (see, for instance, \cite[Remark 1.2]{CRSTV16}), this is finite with a uniform bound independent of $j$ (but depending on $h$). In particular, for these values of $j$ we infer a bound on the quotient
\begin{align*}
\left|\frac{(-\D_{\dis})^s \tilde{u}_j}{\tilde{u}_j}\right| \leq C_h < \infty,
\end{align*}
with an $h\in \R$ dependence of the form $C_h\sim h^{-2s}$.
For $j=1$ (the case $j=-1$ is analogous by symmetry) we obtain
\begin{align}
\label{eq:u1}
(-\D_{\dis})^s \tilde{u}_1 
= \sum\limits_{m\in \Z, m \leq -2} K_s^h(1-m) - \sum\limits_{m \in \Z, m \geq 2} K_s^h(1-m) 
= -K_s^h(-1) - K_s^h(-2).
\end{align}
We seek to ``correct" this. To this end, for some $a\in \R$ we consider the sequence
\begin{equation*}
v_j := \begin{cases}
0 & \mbox{ for } j \in \Z \setminus \{\pm 2\},\\
a & \mbox{ for } j =2,\\
-a & \mbox{ for } j=-2.
\end{cases} 
\end{equation*}
Again the integrability of the kernel $K_s^h$ implies that $|(-\D_{\dis})^s v_j|$ is bounded with a bound which only depends on the choice of $a \in \R$ and $h>0$. For $j=1$ (and by symmetry similarly for $j=-1$) we have
\begin{align}
\label{eq:v1}
(-\D_{\dis})^s v_1 = -a K_s^h(1-2) + a K_s^h(1+2) = a(K_s^h(3)-K_s^h(-1)).
\end{align}
Combining \eqref{eq:u1}, \eqref{eq:v1} and choosing $a\in \R$ such that 
\begin{equation}
\label{eq:a}
a = \frac{K_s^h(-1)+K_s^h(-2)}{K_s^h(3)-K_s^h(-1)} = \frac{K_s^h(1)+K_s^h(2)}{K_s^h(3)-K_s^h(1)} ,
\end{equation}
then implies that the function $u_j=\tilde{u}_j +v_j$ satisfies the following identities and estimates
\begin{align*}
&u_j \in \ell_s, \ (-\D_{\dis})^s u_j = V_j u_j \mbox{ for all } j \in \Z,\\
&u_j = 0 \mbox{ for } j \in \{0,\pm 1\} \mbox{ and }
(-\D_{\dis})^s u_j = 0 \mbox{ for } j\in \{0,\pm 1\}.
\end{align*}
Moreover, $a\ne -1$, which ensures that $V_j:= \frac{(-\D_{\dis})^s \tilde{u}_j + (-\D_{\dis})^s v_j}{\tilde{u}_j + v_j}$ is well-defined (this quotient is set to be equal to zero if both the numerator and the denominator vanish). Since $|(-\D_{\dis})^s \tilde{u}_j + (-\D_{\dis})^s v_j|$ is bounded with a bound $C_{a,h}$ only depending on $h,a$ as $C_{a,h}\sim C_a h^{-2s}$, this yields the desired result.
\end{proof}

\subsection{Proof of Proposition \ref{prop:slab_domain}}

With Lemma \ref{lem:1D_disc_Schr} in hand, we next turn to the proof of Proposition \ref{prop:slab_domain}.

\begin{proof}[Proof of Proposition \ref{prop:slab_domain}]
We use the notation $j=(j_1,j_2)$, $j_i\in \Z$, $i=1,2$. Consider the sequence 
$$
\tilde{u}_j := 
\begin{cases}
1 &\mbox{ if } j_1\ge2,\\
-1 &\mbox{ if } j_1\le -2,\\
0& \mbox{ otherwise, } 
\end{cases}
$$
see Figure \ref{grid2d}.

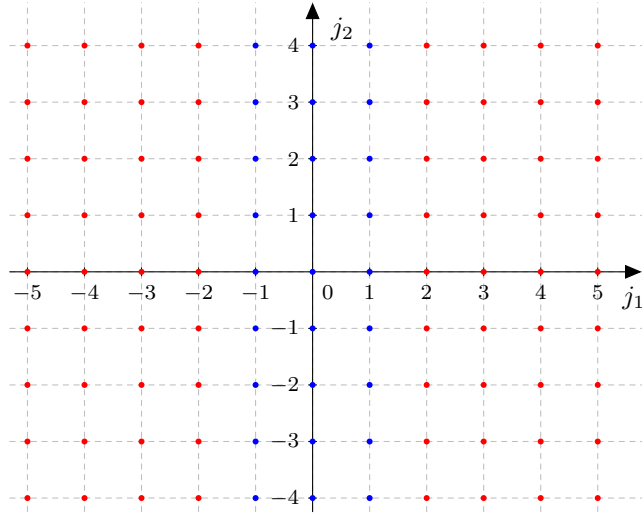
\begin{figure}
\centering
\begin{tikzpicture}[line cap=round,line join=round,>=triangle 45,x=1.0cm,y=1.0cm, scale=0.75]
\draw [color=cqcqcq,dash pattern=on 2pt off 2pt, xstep=1.0cm,ystep=1.0cm] (-5.31,-4.24) grid (5.77,4.76);
\draw[->,color=black] (-5.31,0) -- (5.77,0);
\foreach \x in {-5,-4,-3,-2,-1,1,2,3,4,5}
\draw[shift={(\x,0)},color=black] (0pt,2pt) -- (0pt,-2pt) node[below] {\footnotesize $\x$};
\draw[->,color=black] (0,-4.24) -- (0,4.76);
\foreach \y in {-4,-3,-2,-1,1,2,3,4}
\draw[shift={(0,\y)},color=black] (2pt,0pt) -- (-2pt,0pt) node[left] {\footnotesize $\y$};
\draw[color=black] (0pt,-10pt) node[right] {\footnotesize $0$};
\clip(-5.31,-4.24) rectangle (5.77,4.76);
\draw (5.27,-0.1) node[anchor=north west] {$j_1$};
\draw (0.16,4.66) node[anchor=north west] {$j_2$};
\begin{scriptsize}
\fill [color=qqqqff] (-1,0) circle (1.5pt);
\fill [color=qqqqff] (0,0) circle (1.5pt);
\fill [color=qqqqff] (1,0) circle (1.5pt);
\fill [color=qqqqff] (-1,1) circle (1.5pt);
\fill [color=qqqqff] (0,1) circle (1.5pt);
\fill [color=qqqqff] (1,1) circle (1.5pt);
\fill [color=qqqqff] (-1,2) circle (1.5pt);
\fill [color=qqqqff] (0,2) circle (1.5pt);
\fill [color=qqqqff] (1,2) circle (1.5pt);
\fill [color=qqqqff] (-1,3) circle (1.5pt);
\fill [color=qqqqff] (0,3) circle (1.5pt);
\fill [color=qqqqff] (1,3) circle (1.5pt);
\fill [color=qqqqff] (-1,-1) circle (1.5pt);
\fill [color=qqqqff] (0,-1) circle (1.5pt);
\fill [color=qqqqff] (1,-1) circle (1.5pt);
\fill [color=qqqqff] (-1,-2) circle (1.5pt);
\fill [color=qqqqff] (0,-2) circle (1.5pt);
\fill [color=qqqqff] (1,-2) circle (1.5pt);
\fill [color=qqqqff] (-1,-3) circle (1.5pt);
\fill [color=qqqqff] (0,-3) circle (1.5pt);
\fill [color=qqqqff] (1,-3) circle (1.5pt);
\fill [color=ffqqqq] (-2,3) circle (1.5pt);
\fill [color=ffqqqq] (-3,3) circle (1.5pt);
\fill [color=ffqqqq] (-4,3) circle (1.5pt);
\fill [color=ffqqqq] (-4,2) circle (1.5pt);
\fill [color=ffqqqq] (-3,2) circle (1.5pt);
\fill [color=ffqqqq] (-2,2) circle (1.5pt);
\fill [color=ffqqqq] (-2,1) circle (1.5pt);
\fill [color=ffqqqq] (-3,1) circle (1.5pt);
\fill [color=ffqqqq] (-4,1) circle (1.5pt);
\fill [color=ffqqqq] (-4,0) circle (1.5pt);
\fill [color=ffqqqq] (-3,0) circle (1.5pt);
\fill [color=ffqqqq] (-2,0) circle (1.5pt);
\fill [color=ffqqqq] (-2,-1) circle (1.5pt);
\fill [color=ffqqqq] (-3,-1) circle (1.5pt);
\fill [color=ffqqqq] (-4,-1) circle (1.5pt);
\fill [color=ffqqqq] (-4,-2) circle (1.5pt);
\fill [color=ffqqqq] (-3,-2) circle (1.5pt);
\fill [color=ffqqqq] (-2,-2) circle (1.5pt);
\fill [color=ffqqqq] (-2,-3) circle (1.5pt);
\fill [color=ffqqqq] (-3,-3) circle (1.5pt);
\fill [color=ffqqqq] (-4,-3) circle (1.5pt);
\fill [color=ffqqqq] (2,-3) circle (1.5pt);
\fill [color=ffqqqq] (3,-3) circle (1.5pt);
\fill [color=ffqqqq] (4,-3) circle (1.5pt);
\fill [color=ffqqqq] (4,-2) circle (1.5pt);
\fill [color=ffqqqq] (3,-2) circle (1.5pt);
\fill [color=ffqqqq] (2,-2) circle (1.5pt);
\fill [color=ffqqqq] (2,-1) circle (1.5pt);
\fill [color=ffqqqq] (3,-1) circle (1.5pt);
\fill [color=ffqqqq] (4,-1) circle (1.5pt);
\fill [color=ffqqqq] (4,0) circle (1.5pt);
\fill [color=ffqqqq] (3,0) circle (1.5pt);
\fill [color=ffqqqq] (2,0) circle (1.5pt);
\fill [color=ffqqqq] (2,1) circle (1.5pt);
\fill [color=ffqqqq] (3,1) circle (1.5pt);
\fill [color=ffqqqq] (4,1) circle (1.5pt);
\fill [color=ffqqqq] (4,2) circle (1.5pt);
\fill [color=ffqqqq] (3,2) circle (1.5pt);
\fill [color=ffqqqq] (2,2) circle (1.5pt);
\fill [color=ffqqqq] (2,3) circle (1.5pt);
\fill [color=ffqqqq] (3,3) circle (1.5pt);
\fill [color=ffqqqq] (4,3) circle (1.5pt);
\fill [color=qqqqff] (-1,4) circle (1.5pt);
\fill [color=qqqqff] (0,4) circle (1.5pt);
\fill [color=qqqqff] (1,4) circle (1.5pt);
\fill [color=qqqqff] (-1,-4) circle (1.5pt);
\fill [color=qqqqff] (0,-4) circle (1.5pt);
\fill [color=qqqqff] (1,-4) circle (1.5pt);
\fill [color=ffqqqq] (-4,4) circle (1.5pt);
\fill [color=ffqqqq] (-3,4) circle (1.5pt);
\fill [color=ffqqqq] (-2,4) circle (1.5pt);
\fill [color=ffqqqq] (-5,4) circle (1.5pt);
\fill [color=ffqqqq] (-5,3) circle (1.5pt);
\fill [color=ffqqqq] (-5,2) circle (1.5pt);
\fill [color=ffqqqq] (-5,1) circle (1.5pt);
\fill [color=ffqqqq] (-5,0) circle (1.5pt);
\fill [color=ffqqqq] (-5,-1) circle (1.5pt);
\fill [color=ffqqqq] (-5,-2) circle (1.5pt);
\fill [color=ffqqqq] (-5,-3) circle (1.5pt);
\fill [color=ffqqqq] (-5,-4) circle (1.5pt);
\fill [color=ffqqqq] (-4,-4) circle (1.5pt);
\fill [color=ffqqqq] (-3,-4) circle (1.5pt);
\fill [color=ffqqqq] (-2,-4) circle (1.5pt);
\fill [color=ffqqqq] (2,-4) circle (1.5pt);
\fill [color=ffqqqq] (3,-4) circle (1.5pt);
\fill [color=ffqqqq] (4,-4) circle (1.5pt);
\fill [color=ffqqqq] (5,-4) circle (1.5pt);
\fill [color=ffqqqq] (5,-3) circle (1.5pt);
\fill [color=ffqqqq] (5,-2) circle (1.5pt);
\fill [color=ffqqqq] (5,-1) circle (1.5pt);
\fill [color=ffqqqq] (5,0) circle (1.5pt);
\fill [color=ffqqqq] (5,1) circle (1.5pt);
\fill [color=ffqqqq] (5,2) circle (1.5pt);
\fill [color=ffqqqq] (5,3) circle (1.5pt);
\fill [color=ffqqqq] (5,4) circle (1.5pt);
\fill [color=ffqqqq] (4,4) circle (1.5pt);
\fill [color=ffqqqq] (3,4) circle (1.5pt);
\fill [color=ffqqqq] (2,4) circle (1.5pt);
\end{scriptsize}
\end{tikzpicture}\caption{In the blue points, $\widetilde{u}_j=0$.} \label{grid2d}
\end{figure}

From \eqref{eq:heatn} and \eqref{eq:fracn} we have that
\begin{equation}
\label{eq:Ksm2}
K_s^h(m)=\frac{1}{h^{2s}}\frac{1}{|\Gamma(-s)|}\int_0^{\infty}e^{-4t}\prod_{i=1}^2I_{m_i}(2t)\frac{dt}{t^{1+s}}.
\end{equation}
The following asymptotics are well known: there exist constants $C$, $c>0$ such that
\begin{equation}
\label{eq:asymptotics-zero}
c t^k\le I_k(t)\le C t^k \quad \text{ for } t\to0^+.
\end{equation}
Actually,
\begin{equation}
\label{eq:asymptotics-zero-ctes}
I_k(t)\sim \Big(\frac{t}{2}\Big)^k\frac{1}{\Gamma(k+1)}, \quad \text{ for a fixed } k\neq -1,-2,-3,\ldots \quad \text{ and } t\to0^+,
\end{equation}
see \cite{OlMax}.
Moreover (see \cite{Lebedev})
\begin{equation}
\label{eq:asymptotics-infinite}
I_k(t)=C e^t t^{-1/2}+ R_k(t),
\end{equation}
where
\[
|R_k(t)|\le C_k e^tt^{-3/2}, \quad \text{ for } t\to\infty.
\]
The asymptotics \eqref{eq:asymptotics-zero} and \eqref{eq:asymptotics-infinite} ensures the integrability in \eqref{eq:Ksm2}.
If $j_1=0$, then
\begin{align*}
(-\D_{\dis})^s \tilde{u}_{(0,j_2)}&=-\sum_{\substack{m\in \Z^2\\ m\neq (0,j_2)}} \tilde{u}_{(m_1,m_2)}K_s^h(-m_1,j_2-m_2)\\
&=\sum_{m_1=-\infty}^{-2}\sum_{m_2=-\infty}^{\infty}K_s^h(-m_1,j_2-m_2)-\sum_{m_1=2}^{\infty}\sum_{m_2=-\infty}^{\infty}K_s^h(-m_1,j_2-m_2)\\
&=\sum_{m_1=2}^{\infty}\sum_{m_2=-\infty}^{\infty}K_s^h(m_1,j_2-m_2)-\sum_{m_1=2}^{\infty}\sum_{m_2=-\infty}^{\infty}K_s^h(m_1,j_2-m_2)=0,
\end{align*}
where we used the symmetry of the kernel.

For $j=(j_1,j_2)$ with $j_1\ge 2$ (it works similarly for $j_1\le -2$) we get
\begin{equation}
\label{eq:sums}
(-\Delta_{\dis})^s\tilde{u}_j=\sum_{m_1=-\infty}^{-2}\sum_{m_2=-\infty}^{\infty}2K_s^h(j-m)+\sum_{m_2=-\infty}^{\infty}\sum_{m_1=-1}^{1}K_s^h(j-m).
\end{equation}
The kernel $K_s^h(m)$ can be suitably estimated by Lemma \ref{lem:upperEstimateKernelFL} (see also Remark \ref{rmk:decay}) and the sum in \eqref{eq:sums} is finite with a uniform bound independent of $j$ (but depending on $h$).

In particular, for these values of $j$ we again infer a bound on the quotient
\begin{align*}
\left|\frac{(-\D_{\dis})^s \tilde{u}_j}{\tilde{u}_j}\right| \leq C_h < \infty,
\end{align*}
which only depends on $h\in \R$ as $C_h\sim h^{-2s}$.
 
For $j=(1,j_2)$ (it works similarly for $j=(-1,j_2)$) we get
\begin{align}
\label{eq:1j2}
\notag(-\D_{\dis})^s \tilde{u}_{(1,j_2)}&=\sum_{m_1=-\infty}^{-2}\sum_{m_2=-\infty}^{\infty}K_s^h(1-m_1,j_2-m_2)-\sum_{m_1=2}^{\infty}\sum_{m_2=-\infty}^{\infty}K_s^h(1-m_1,j_2-m_2)\\
&=\sum_{m_1=2}^{\infty}\sum_{m_2=-\infty}^{\infty}\big[K_s^h(1+m_1,j_2-m_2)-K_s^h(1-m_1,j_2-m_2)\big].
\end{align}
It is known that 
\begin{equation}
\label{eq:uno}
\sum_{m\in \Z}e^{-2t}I_{m}(2t)=1.
\end{equation}
The above identity follows from taking $u=1$ and $z=2t$ in the generating function below, which is valid for $z\in \C$ and $u\in \C\setminus \{0\}$,
$$
e^{\frac12z(u+u^{-1})}=\sum_{m\in \Z}u^mI_m(z)
$$
see, for instance, \cite[formula 10.35.1]{OlMax}.
Hence, \eqref{eq:1j2} can be spelled out as 
\begin{align*}
\notag &(-\D_{\dis})^s \tilde{u}_{(1,j_2)}\\
&\quad =\sum_{m_1=2}^{\infty}\sum_{m_2=-\infty}^{\infty}\frac{1}{h^{2s}}\frac{1}{|\Gamma(-s)|}\int_0^{\infty}e^{-4t}\big(I_{1+m_1}(2t)I_{j_2-m_2}(2t)-I_{1-m_1}(2t)I_{j_2-m_2}(2t)\big)\frac{dt}{t^{1+s}}\\
&\quad =\frac{1}{h^{2s}}\frac{1}{|\Gamma(-s)|}\int_0^{\infty}e^{-2t}\sum_{m_2=-\infty}^{\infty}e^{-2t}I_{j_2-m_2}(2t)\sum_{m_1=2}^{\infty}\big(I_{1+m_1}(2t)-I_{1-m_1}(2t)\big)\frac{dt}{t^{1+s}}\\
&\quad =\frac{1}{h^{2s}}\frac{1}{|\Gamma(-s)|}\int_0^{\infty}e^{-2t}\sum_{m_2=-\infty}^{\infty}e^{-2t}I_{j_2-m_2}(2t)\big(-I_{1}(2t)-I_{2}(2t)\big)\frac{dt}{t^{1+s}}\\
&\quad=\frac{1}{h^{2s}}\frac{1}{|\Gamma(-s)|}\int_0^{\infty}e^{-2t}\big(-I_{1}(2t)-I_{2}(2t)\big)\frac{dt}{t^{1+s}}= -K_s^h(-1) - K_s^h(-2),
\end{align*}
where the last expression corresponds to the one with one-dimensional kernels. Above, the interchange of sum and integral is justified in view of the fact that $I_k$ are nonnegative functions and \eqref{eq:uno}, so that the sums involving $e^{-2t}I_{m_i}(2t)$ are absolutely convergent. Moreover, from the second to third equality we used the cancellation due to the symmetry property of Bessel functions $I_{-m_i}=I_{m_i}$,  see Subsection \ref{sub:heat}, which turns the sum in $m_1$ into a telescoping series. 
In order to correct the term above, we consider 
\begin{equation*}
v_j := \begin{cases}
a & \mbox{ for } j_1 =2,\\
-a & \mbox{ for } j_1=-2,\\
0 & \mbox{ otherwise,}
\end{cases} 
\end{equation*}
see Figure \ref{grid2dV}.
\begin{figure}
\centering
\begin{tikzpicture}[line cap=round,line join=round,>=triangle 45,x=1.0cm,y=1.0cm, scale=0.75]
\draw [color=cqcqcq,dash pattern=on 2pt off 2pt, xstep=1.0cm,ystep=1.0cm] (-5.31,-4.24) grid (5.77,4.76);
\draw[->,color=black] (-5.31,0) -- (5.77,0);
\foreach \x in {-5,-4,-3,-2,-1,1,2,3,4,5}
\draw[shift={(\x,0)},color=black] (0pt,2pt) -- (0pt,-2pt) node[below] {\footnotesize $\x$};
\draw[->,color=black] (0,-4.24) -- (0,4.76);
\foreach \y in {-4,-3,-2,-1,1,2,3,4}
\draw[shift={(0,\y)},color=black] (2pt,0pt) -- (-2pt,0pt) node[left] {\footnotesize $\y$};
\draw[color=black] (0pt,-10pt) node[right] {\footnotesize $0$};
\clip(-5.31,-4.24) rectangle (5.77,4.76);
\draw (5.27,-0.1) node[anchor=north west] {$j_1$};
\draw (0.16,4.66) node[anchor=north west] {$j_2$};
\begin{scriptsize}
\fill [color=ffqqqq] (-2,0) circle (1.5pt);
\fill [color=qqqqff] (2,0) circle (1.5pt);
\fill [color=ffqqqq] (-2,1) circle (1.5pt);
\fill [color=qqqqff] (2,1) circle (1.5pt);
\fill [color=ffqqqq] (-2,2) circle (1.5pt);
\fill [color=qqqqff] (2,2) circle (1.5pt);
\fill [color=ffqqqq] (-2,3) circle (1.5pt);
\fill [color=qqqqff] (2,3) circle (1.5pt);
\fill [color=ffqqqq] (-2,-1) circle (1.5pt);
\fill [color=qqqqff] (2,-1) circle (1.5pt);
\fill [color=ffqqqq] (-2,-2) circle (1.5pt);
\fill [color=qqqqff] (2,-2) circle (1.5pt);
\fill [color=ffqqqq] (-2,-3) circle (1.5pt);
\fill [color=qqqqff] (2,-3) circle (1.5pt);
\fill [color=ffqqqq] (-2,4) circle (1.5pt);
\fill [color=qqqqff] (2,4) circle (1.5pt);
\fill [color=ffqqqq] (-2,-4) circle (1.5pt);
\fill [color=qqqqff] (2,-4) circle (1.5pt);
\end{scriptsize}
\end{tikzpicture}
\caption{In the blue points, $v_j=a$; in the red points, $v_j=-a$; in the remaining points of the lattice, $v_j=0$.} \label{grid2dV}
\end{figure}
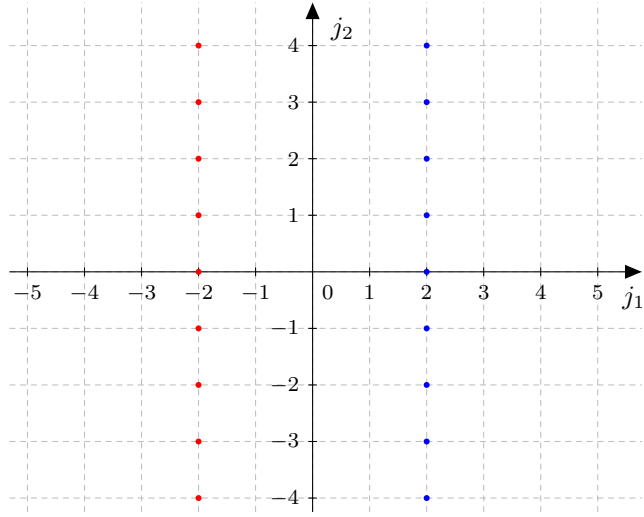
Once again, the integrability of the kernel $K_s^h$ implies that $|(-\D_{\dis})^{s}v_j|$ is bounded, with a bound which only depends on the choice of $a\in \R$ and $h>0$. Moreover, for $j_1=1$ (and by symmetry it works similarly for $j_1=-1$), we have
\begin{align*}
(-\D_{\dis})^sv_{(1,j_2)}&=-\sum_{m_2\in \Z}\big[v_{(2,m_2)}K_s^h(-1,j_2-m_2)+v_{(-2,m_2)}K_s^h(3,j_2-m_2)\big]\\
&=a\sum_{m_2\in \Z}\big[K_s^h(3,j_2-m_2)-K_s^h(-1,j_2-m_2)\big]\\
&=a\frac{1}{h^{2s}}\frac{1}{|\Gamma(-s)|}\int_0^{\infty}\sum_{m_2\in \Z}e^{-2t}I_{j_2-m_2}e^{-2t}\big(I_{3}(2t)-I_{-1}(2t)\big)\frac{dt}{t^{1+s}}\\
&=a(K_s^h(3)-K_s^h(-1)),
\end{align*}
where again the last expression corresponds to the one with one-dimensional kernels. Thus, we have reduced the construction to the one-dimensional case, and by choosing $a$ as in \eqref{eq:a} we can define the function $u_j=\tilde{u}_j+v_j$ so that 
\begin{align*}
&u_j \in \ell_s, \ (-\D_{\dis})^s u_j = V_j u_j \mbox{ for all } j \in \Z^2,\\
&u_j = 0 \mbox{ if and only if } j_1 \in \{0,\pm 1\} \mbox{ and }
(-\D_{\dis})^s u_j = 0 \mbox{ for } j_1\in \{0,\pm 1\},
\end{align*}
and $V_j:= \frac{(-\D_{\dis})^s \tilde{u}_j + (-\D_{\dis})^s v_j}{\tilde{u}_j + v_j}$ (where this quotient is set to be equal to zero if both the numerator and the denominator vanish). We obtain the desired result.
\end{proof}

\begin{rmk}
\label{rmk:nD}
The argument above can be easily extended to higher dimensions, just considering the same kind of sequences and taking into account the fact that the semidiscrete heat kernel separates the components $m_i$ in the variable $m=(m_1,\ldots, m_d)$.
\end{rmk}

\section{Boundary-Bulk Doubling Estimates for the Discrete Half-Laplacian}
\label{sec:bdry_bulk}

Using the Caffarelli--Silvestre extension, we complement the failure of the unique continuation properties for the discrete half-Laplacian by the boundary-bulk unique continuation property for the semidiscrete equation
\begin{align*}
(\p_t^2 + \D_{\dis}) \tilde{u} & = 0 \mbox{ in } (h\Z)^d \times \R_+,\\
\tilde{u} & = u \mbox{ in } (h \Z)^d \times \{0\}.
\end{align*}
More precisely, we deduce a discrete analogue of the boundary-bulk three balls inequality \eqref{eq:bdry_bulk} from Theorem \ref{thm:bdry_bulk} which is valid in the continuum setting.
Quantitative unique continuation estimates of this type are a key ingredient in the quantitative derivation of global unique continuation properties for the half-Laplacian. For the case of general $s\in (0,1)$ it played a major role in deducing stability for the fractional Calder\'on problem \cite{RS20a,Rue21}, see also \cite{GFR20,Rue17,RS20a} for related work. We hope that the estimates presented here may be of use in studying the discrete-to-continuum limits for discrete versions of the fractional Calder\'on problem.

Due to the discussed failure of the global unique continuation property for the fractional Laplacian (Theorem \ref{prop:unique}) and in analogy to the doubling inequalities for the discrete Laplacian from \cite{FBRRS20}, we do not expect an estimate of the form \eqref{eq:bdry_bulk} to be true without an additional (but very weak) $h$-dependent correction term. 

Our main result in this context thus is the estimate \eqref{eq:discrete_bdry_bulk} which provides the discrete analogue of \eqref{eq:bdry_bulk} with a mesh-dependent correction term. As in \cite{FBRRS20} we derive this estimate as a consequence of a corresponding Carleman inequality:

\begin{thm}
\label{thm:Carl_disc}
Let $u\in H^1((h \Z)^d)$, $\tilde{u}\in H^2((h\Z)^d \times \R_+)$  such that $u=\tilde{u}$ in $(h\Z)^d\times\{0\}$. Suppose further that $\supp(u) \subset B_{4/5}^+ \subset (h \Z)^d \times \R_+$.
Define the Carleman weight $\phi: (h \Z)^d \times \R_+ \rightarrow \R$ to be 
\begin{align*}
\phi(jh, t) := -|jh|^2 + c_0 \Big( \frac{1}{2} t^2-  t\Big).
\end{align*}
Then, there exist $h_0 \in (0,1/4)$, $\delta_0>0$ with $h_0<\delta_0$ and $\tau_0>1$ such that for all $h\in (0,h_0)$ and $\tau \in (\tau_0, \delta_0 h_0^{-1})$ it holds
\begin{align}
\label{eq:Carl}
\begin{split}
&\tau^{3/2}\|e^{\tau \phi} \tilde{u}\|_{L^2((h \Z)^d \times \R_+ )} + \tau^{1/2}\|e^{\tau \phi} \nabla_{\dis} \tilde{u}\|_{L^2((h \Z)^d \times \R_+)}+ \tau^{1/2}\|e^{\tau \phi} \p_t \tilde{u}\|_{L^2((h \Z)^d \times \R_+)}\\\
&\leq C (\|e^{\tau \phi} (\D_{\dis} + \p_t^2) \tilde{u}\|_{L^2((h\Z)^d \times \R_+)} + \tau^{3/2}\|e^{\tau \phi} (|u| + |\nabla_{\dis} u| + |\p_t \tilde{u}|)\|_{L^2((h\Z)^d \times \{0\})}).
\end{split}
\end{align}
The constant $C>0$ only depends on $d, c_0$.
\end{thm}

Let us comment on this estimate: The Carleman estimate and the Carleman weight function from Theorem \ref{thm:Carl_disc} are in parallel to the continuous counterpart which can be used to prove the continuous boundary-bulk doubling inequality \eqref{eq:bdry_bulk} (see, for instance, \cite{JL99}). Following the continuous case, our Carleman weight has the property that it is \emph{convex} only in the normal direction while it is \emph{concave} in the tangential directions. The Carleman weight thus only enjoys pseudoconvexity and not full convexity properties. In spite of this difficulty, contrary to our arguments from \cite{FBRRS20}, the polynomial structure of the weight allows us to carry out the conjugation arguments directly also in the discrete setting. In particular, we do not have to resort to freezing arguments as in \cite{FBRRS20} in order to transfer the continuum estimates to the discrete setting. 
We refer to Remark \ref{rmk:tau} for a comment on the restriction on the size of the admissible choices of $\tau>1$ in Theorem \ref{thm:Carl_disc}.

\subsection{Proof of Theorem \ref{thm:Carl_disc}}

In this section, we present the proof of our main auxiliary tool, the Carleman estimate from Theorem \ref{thm:Carl_disc}. The argument follows the continuum strategy with a ``direct'' conjugation argument (instead of a freezing argument as in \cite{FBRRS20}). Only in estimating the contributions in the size of the semi-classical parameter $\tau\geq 1$ and in the bounds for the error contributions, linearization arguments are invoked.

\begin{proof}[Proof of Theorem \ref{thm:Carl_disc}]
The proof of Theorem \ref{thm:Carl_disc} mimics the argument of the continuous setting. In particular, the positivity of the commutator originates from the convexity of the normal contribution of the Carleman weight. 

\emph{Step 1: Conjugation.}
We begin by considering the conjugated operator:
\begin{align*}
L_{\phi} = e^{\tau \phi}(\D_{\dis} + \p_t^2) e^{-\tau \phi}.
\end{align*}
Hence, seeking to deduce the desired Carleman estimate, and, setting $v= e^{\tau \phi} \tilde{u}$, we expand as follows
\begin{align}
\begin{split}
\label{eq:expand}
&\|e^{\tau \phi} (\D_{\dis} + \p_t^2 ) \tilde{u}  \|_{L^2((h\Z)^d \times \R_+)}^2
= \|L_{\phi} v  \|_{L^2((h\Z)^d \times \R_+)}^2\\
&= \|S_{\phi} v\|_{L^2((h\Z)^d \times \R_+)}^2 + \|A_{\phi} v \|_{L^2((h\Z)^d \times \R_+)}^2 + ([S_{\phi}, A_{\phi}]v,v)_{L^2((h\Z)^d \times \R_+)} + \mbox{(BC)},
\end{split}
\end{align}
where (BC) denotes the boundary contributions. Here $S_{\phi}, A_{\phi}$ denote the symmetric and antisymmetric bulk contributions of $L_{\phi}$ in the full space setting, ignoring the boundary contributions.
By virtue of the fact that the weight $\phi$ is given as the sum of functions of each variable individually and since thus the bulk contributions in the tangential and normal directions commute, we may consider the conjugation and splitting in the symmetric and antisymmetric parts in the tangential and normal variables separately. 

In the sequel, we will use different scalar products $(\cdot,\cdot)_{L^2((h\Z)^d\times\R_+)}$, $(\cdot,\cdot)_{L^2((h\Z)^d)}$, $(\cdot,\cdot)_{L^2(\R_+)}$ and $(\cdot,\cdot)_0$. While it is clear on which space are taken the first three scalar products, the last one refers to the scalar product on $(h\Z)^d\times\{0\}$.

\emph{Step 1a: Bulk commutator in the tangential directions.}
We observe that
\begin{align*}
e^{\tau \phi}\D_{\dis} e^{-\tau \phi} = S_{\phi,d} + A_{\phi,d},
\end{align*}
where -- using the notational convention that $u_j := u(hj)$ -- we obtain
\begin{align*}
h^2 S_{\phi, d} v_{j}&:= \sum\limits_{k=1}^{d}
\left( \cosh(\tau (\phi_{j} - \phi_{j +  e_k})) v_{j + e_k} + \cosh(\tau(\phi_{j} -\phi_{j - e_k}) ) v_{j -  e_k} -2 v_{j} \right),\\
h^2 A_{\phi, d} v_{j} &:= 
 \sum\limits_{k=1}^{d}
\left( \sinh(\tau (\phi_{j} - \phi_{j +  e_k})) v_{j + e_k} + \sinh(\tau(\phi_{j} -\phi_{j - e_k}) ) v_{j -  e_k} \right).
\end{align*}
Thus, a short computation shows that the tangential commutator (see, for instance, \cite{FBV17}) turns into  
\begin{align*}
([S_{\phi,d},A_{\phi,d}]v,v)_{L^2((h\Z)^d)} 
& = -4 h^{-4} \sinh(2 \tau h^2) \sum\limits_{j \in \Z^d} \sum\limits_{k=1}^{d} \sinh^2(2\tau j_k h^2) |v_{j}|^2 \\
& \quad -4 h^{-2} \sinh(2 \tau h^2)\sum\limits_{j \in \Z^d} \sum\limits_{k=1}^{d} \Big| \frac{v_{j +  e_k} - v_{j - e_k}}{2h} \Big|^2.
\end{align*}
Now, due to the upper bound on $\tau$, we obtain that $\tau h \leq \delta_0$, which allows to linearize the hyberbolic functions. This yields that 
\begin{align*}
([S_{\phi,d},A_{\phi,d}]v,v)_{L^2((h\Z)^d)}  
& = - 32 \tau^3 \sum\limits_{j \in \Z^d} |h j|^2 |v_{j}|^2 - 8 \tau  \sum\limits_{j \in \Z^d}\sum\limits_{k=1}^{d}\Big| \frac{v_{j +  e_k} - v_{j - e_k}}{2h} \Big|^2 + E_1 + E_2, 
\end{align*}
where $E_1$ and $E_2$ are error terms satisfying the bounds from \eqref{eq:E1} and \eqref{eq:E2} below.

Indeed, using that $|\tau h|\leq \delta_0$, and that in the support of $u$ we have $\ |jh|\le 1$, we first note that
\begin{align*}
\sinh(2\tau h^2) = 2 \tau h^2 + O(\tau^3 h^6), \ 
\sinh^2(2 \tau j_k h^2) = (2 \tau j_k h^2)^2 + O(\tau^4 j_k^4 h^8).
\end{align*}
Thus, 
\begin{align*}
 h^{-4} \sinh(2 \tau h^2)  \sum\limits_{j \in \Z^d}\sum\limits_{k=1}^{d} \sinh^2(2\tau j_k h^2) |v_{j}|^2 = C \tau^3 \sum\limits_{j \in \Z^d} |h j|^2 |v_{j}|^2
 + E_{1},
\end{align*}
with 
\begin{align}
\label{eq:E1}
 |E_1| \leq C \tau^3\delta_0^2 \sum\limits_{j \in \Z^d} |v_{j}|^2.
\end{align}
Similarly, we have that 
\begin{align*}
4 h^{-2} \sinh(2 \tau h^2) \sum\limits_{j \in \Z^d} \sum\limits_{k=1}^{d}\Big| \frac{v_{j +  e_k} - v_{j - e_k}}{2h} \Big|^2
= C  \tau \sum\limits_{j \in \Z^d} \sum\limits_{k=1}^{d}\Big| \frac{v_{j +  e_k} - v_{j - e_k}}{2h}\Big|^2 + E_2 ,
\end{align*}
where
\begin{align}
\label{eq:E2}
|E_2| \leq C \tau\delta_0^2 \sum\limits_{j \in \Z^d} |v_j|^2.
\end{align}

\emph{Step 1b: Bulk commutator in normal directions.}
We next consider the corresponding contributions in the normal directions. For this we have
\begin{align*}
e^{\tau \phi} \p_t^2 e^{-\tau \phi} = \p_t^2 + \tau^2 |\p_t \phi|^2 - 2\tau \p_t - \tau \p_t^2 \phi
= S_{\phi,t} + A_{\phi,t},
\end{align*}
where
\begin{align*}
S_{\phi,t} := \p_t^2 + \tau^2 |\p_t \phi|^2 , \
A_{\phi,t} := -2\tau \p_t \phi \p_t - \tau \p_t^2 \phi.
\end{align*}
Hence, the normal commutator turns into 
\begin{align*}
([S_{\phi,t},A_{\phi,t}]v,v)_{{L^2(\R_+)}} &= 4 \tau^3 ((\p_t \phi)^2 \p_t^2 \phi v, v )_{{L^2(\R_+)}} + 4\tau (\p_t v, \p_t^2 \phi \p_t v)_{{L^2(\R_+)}}-\tau(\partial_t^4\phi v,v )_{{L^2(\R_+)}}\\
&\geq 4 \tau^3 c_0^3 ((1-t)^2 v, v)_{{L^2(\R_+)}} + 4 \tau c_0 (\p_t v, \p_t v)_{{L^2(\R_+)}}.
\end{align*}

\emph{Step 1c: Boundary contributions.}
Returning to \eqref{eq:expand}, we next seek to identify the corresponding boundary contributions. These arise when integrating the bulk terms $2(S_{\phi} v, A_{\phi}v)_{L^2((h\Z)^d \times \R_+)}$ by parts which then result in the bulk commutator. Since $v$ does not vanish on $(h \Z)^d \times \{0\}$ in general, we obtain the following contributions:
\begin{align*}
(BC) &=  2\tau ((\p_t \phi) \p_t v, \p_t v)_0 + 2\tau^3 ((\p_t \phi)^3 v, v)_0 + 2\tau (v, \p_t^2 \phi \p_t v)_0 - \tau (v, (\p_t^3 \phi) v)_0 \\
& \quad -2 (\p_t v, A_{\phi,d}v)_0 + 2 \tau (S_{\phi,d}v, (\p_t \phi) v)_0.
\end{align*}

We next seek to bound these contributions. Using the support assumption for $v$, we can estimate all the terms which do not involve the tangential operators $A_{\phi,d}$ and $S_{\phi,d}$:
\begin{align}
\label{eq:boundary_non_tangential}
\begin{split}
& 2\tau |((\p_t \phi) \p_t v, \p_t v)_0| + 2\tau^3 |( |\p_t \phi|^3 v, v)_0| + 2\tau |(v, \p_t^2 \phi \p_t v)_0| + \tau |(v, (\p_t^3 \phi) v)_0 |\\
&\leq C(1 + c_0^3)(\tau\|\p_t v\|_0^2 + \tau^3\|v\|_0^2).
\end{split}
\end{align}
It hence remains to bound the boundary contributions involving the tangential operators $A_{\phi,d}$ and $S_{\phi,d}$. We first consider the term involving $A_{\phi,d}$. To this end, we note that
\begin{align*}
|(\p_t v, A_{\phi,d}v)_0|& \leq \|\p_t v\|_0 \|A_{\phi,d}v\|_0\\
&\leq \|\p_t v\|_0 \Big(  \sum\limits_{j\in \Z^d}\sum\limits_{k=1}^{d}h^{-4}\big(|\sinh(\tau(\phi_j - \phi_{j+ e_k}))(v_{j+e_k}-v_j)|^{2} \\
& \quad  +
|(\sinh(\tau(\phi_j - \phi_{j+e_k}) ) + \sinh(\tau (\phi_j - \phi_{j-e_k}))v_j)|^{2} \\
&\quad 
+ |\sinh(\tau(\phi_j - \phi_{j- e_k}))(v_{j-e_k}-v_j)|^{2}
 \big) \Big)^{1/2}.
\end{align*}
Expanding the trigonometric functions, using the imposed smallness assumption, we can further bound these contributions 
\begin{align*}
& \Big( \sum\limits_{j\in \Z^d} \sum\limits_{k=1}^{d}h^{-4}\big(|\sinh(\tau(\phi_j - \phi_{j+ e_k}))(v_{j+e_k}-v_j)|^{2}  \\
& \quad  +  
|(\sinh(\tau(\phi_j - \phi_{j+e_k}) ) + \sinh(\tau (\phi_j - \phi_{j-e_k}))v_j)|^{2} 
+ |\sinh(\tau(\phi_j - \phi_{j- e_k}))(v_{j-e_k}-v_j)|^{2}
 \big) \Big)^{1/2}\\
& \leq C \tau \| \nabla_{\dis} v\|_{0} + C\tau^2 \| v\|_{0}, 
\end{align*} 
where, besides the support condition of $u$, we have, for instance, used that 
\begin{align*}
\sinh(\tau(\phi_j - \phi_{j + e_k})) + \sinh(\tau(\phi_j - \phi_{j-e_k}))
= -\tau h^2 \D_{\dis,k}\phi_j + O(\tau^{3} (\phi_j - \phi_{j+e_k})^{3} + \tau^{3} (\phi_j - \phi_{j-e_k})^{3} ),
\end{align*}
(where we have denoted the discrete Laplacian in direction $k$ by $\Delta_{\dis,k}$)
and that 
\begin{align*}
\sum\limits_{j\in \Z^d} h^{-4}\tau^6 (\phi_j - \phi_{j+e_k})^6 |v_j|^2
\leq C \tau^4 \| v\|_{0}^2.
\end{align*}
We argue similarly for the contribution involving $S_{\phi,d}v$: We first rewrite the symmetric tangential operator as
\begin{align}
\label{eq:symmetric}
S_{\phi,d} v_j =  \sum\limits_{k=1}^{d} \Big(h^{-2}(\cosh(\tau (\phi_j - \phi_{j+e_k}))-1)v_{j+e_k} + h^{-2}(\cosh(\tau(\phi_j - \phi_{j-e_k}))-1)v_{j-e_k} + \D_{d,k} v_j\Big).
\end{align}
Expanding the contributions involving the trigonometric functions, we obtain
\begin{align*}
\cosh(\tau (\phi_j - \phi_{j+e_k}))-1 = C \tau^2(\phi_j - \phi_{j+e_k})^2 + O(\tau^4(\phi_j - \phi_{j+e_k})^4).
\end{align*}
Inserting the explicit expression for $\phi_j$, we can further bound this by
\begin{align*}
\cosh(\tau (\phi_j - \phi_{j+e_k}))-1 \le C \tau^2(jh)^2 h^2 + O(\tau^4 (jh)^4 h^4).
\end{align*}
Hence, we infer 
\begin{align*}
2\tau |((\p_t \phi) v, S_{\phi,d}v)_0| 
&\leq C\tau^3 \||\p_t \phi|^{\frac{1}{2}}(jh) v\|^2_0 +  C\tau^3 \delta_0^2 \||\p_t \phi|^{\frac{1}{2}} (jh)^2 v\|^2_0 + C \tau \||\p_t \phi|^{\frac{1}{2}} \nabla_{\dis} v\|^2_{0}\\
& \leq C\tau^3 \| v\|^2_0 +  C\tau^3 \delta_0^2 \| v\|^2_0 + C \tau \| \nabla_{\dis} v\|^2_{0}.
\end{align*}

\emph{Step 1d: Combination of the lower bounds from conjugation.}
Combining the above arguments (and in particular also using the support assumption for $u$ and hence $v$), we obtain that
\begin{align*}
& \|S_{\phi} v\|_{L^2}^2  + \|A_{\phi} v\|_{L^2}^2 + 4 \tau^3 c_0^3 ((1-t)^2 v, v) + 4 \tau c_0(\p_t v, \p_t v) \\
& \leq \|L_{\phi} v\|_{L^2}^2  + 32 \tau^3 \int_{\R_+} \sum\limits_{j \in \Z^d} |h j|^2 |v_{j}|^2\,dt + 8 \tau \int_{\R_+} \sum\limits_{j \in \Z^d}\sum\limits_{k=1}^{d}\Big| \frac{v_{j +  e_k} - v_{j - e_k}}{2h} \Big|^2\,dt\\
& \quad + C \tau (\|\nabla_{\dis} v\|_{0} ^2+ \tau^2\|v\|_0^2 + \|\p_t v\|_0^2) + |E_{1}| + |E_2|.
\end{align*}
Here the notation $(\cdot,\cdot)$ refers to the full $L^2((h\Z)^d\times \R_+)$ scalar product.

Recalling the bounds for $E_1, E_2$ from \eqref{eq:E1}, \eqref{eq:E2}, choosing $c_0>0$ sufficiently large and recalling the support assumption $\supp(v)\subset B_{4/5}^+$, we infer that these error contributions can be absorbed into the left hand side of the inequality. As a consequence,
\begin{align}
\label{eq:first_est_Carl}
\begin{split}
& \|S_{\phi} v\|_{L^2}^2  + \|A_{\phi} v\|_{L^2}^2 + 4 \tau^3 c_0^3 ((1-t)^2 v, v) + 4 \tau c_0(\p_t v, \p_t v) \\
& \leq \|L_{\phi} v\|_{L^2}^2  + 32 \tau^3 \int_{\R_+}\sum\limits_{j \in \Z^d} |h j|^2 |v_{j}|^2\,dt + 8 \tau \int_{\R_+}\sum\limits_{j \in \Z^d}\sum\limits_{k=1}^{d}\Big| \frac{v_{j +  e_k} - v_{j - e_k}}{2h} \Big|^2\,dt\\
& \quad + C \tau (\|\nabla_{\dis} v\|_{0} ^2+ \tau^2\|v\|_0^2 + \|\p_t v\|_0^2).
\end{split}
\end{align}

We next seek to remove the two remaining bulk contributions on the right hand side of \eqref{eq:first_est_Carl} by absorbing them into the left hand side of \eqref{eq:first_est_Carl}. To this end, we first observe that by choosing $c_0>1$ sufficiently large and recalling that $\supp(u)\subset B_1^+$, it is possible to absorb the $L^2$ bulk contribution from the right hand side into the left hand side $L^2$ contribution of \eqref{eq:first_est_Carl}. Hence, if $c_0>1$ is sufficiently large, we arrive at
\begin{align}
\label{eq:Carl_aux}
\begin{split}
&\|S_{\phi} v\|_{L^2}^2  + \|A_{\phi} v\|_{L^2}^2 + 2 \tau^3 c_0^3 ((1-t)^2 v, v) + 4 \tau c_0(\p_t v, \p_t v) \\
& \leq \|L_{\phi} v\|_{L^2}^2  +  8 \tau \int_{\R_+}\sum\limits_{j \in \Z^d}\sum\limits_{k=1}^{d} \Big| \frac{v_{j +  e_k} - v_{j - e_k}}{2h} \Big|^2\,dt
 + C \tau (\|\nabla_{\dis} v\|_{0} ^2+ \tau^2 \|v\|_0^2 + \|\p_t v\|_0^2).
\end{split}
\end{align}
In the next step, we discuss how the remaining bulk contribution can be removed from the right hand side.

\emph{Step 2: Compensation of the negative tangential contributions.}
We discuss how to bound the contribution
\begin{align*}
8 \tau \int_{\R_+}\sum\limits_{j \in \Z^d}\sum\limits_{k=1}^{d} \Big| \frac{v_{j +  e_k} - v_{j - e_k}}{2h} \Big|^2\,dt
\end{align*}
on the right hand side of \eqref{eq:Carl_aux}. To this end, we again recall the representation \eqref{eq:symmetric}. From this we obtain that
\begin{align}
\label{eq:control_neg_der}
\begin{split}
-(S_{\phi} v, v) & = \int_{\R_+} \sum\limits_{j \in \Z^d}\sum\limits_{k=1}^{d}\Big| \frac{v_{j +  e_k} - v_{j - e_k}}{2h} \Big|^2\,dt 
- \int_{\R_+}\sum\limits_{j \in \Z^d} \sum\limits_{k=1}^{d} \Big[ h^{-2}((\cosh(\tau (\phi_j - \phi_{j+e_k}))-1)v_{j+e_k},v_j) \\
& \quad + h^{-2}((\cosh(\tau(\phi_j - \phi_{j-e_k}))-1)v_{j-e_k},v_j) \Big]\,dt + (\p_t v, \p_t v) - \tau^2(|\p_t \phi|^2 v, v ).
\end{split}
\end{align}
Again using the smallness condition $\tau \leq \delta_0 h^{-1}_0$ and the support condition for $v$, we hence deduce that
\begin{align}
\label{eq:symmetric_term}
-8 \tau (S_{\phi} v, v) = 8 \tau\int_{\R_+} \sum\limits_{j \in \Z^d} \sum\limits_{k=1}^{d}\Big| \frac{v_{j +  e_k} - v_{j - e_k}}{2h} \Big|^2\,dt 
+ 8 \tau (\p_t v, \p_t v) + E_{3} .
\end{align}
with $|E_3|\leq C\tau^3 c_0^2 (v,v)$ for some sufficiently large constant $C>1$ (independent of $c_0,\tau$). Indeed, in order to infer this error bound, it suffices to expand the contributions in \eqref{eq:control_neg_der} involving the $\cosh$ contributions. To this end, we observe that
\begin{align*}
\cosh(\tau(\phi_j - \phi_{j+e_k}))-1 = C \tau^2 (\phi_j - \phi_{j+e_k})^2 + O(\tau^4(\phi_j - \phi_{j+ e_k})^4).
\end{align*}
Inserting the definition of $\phi_j$, using the smallness condition $\tau h_0 \leq \delta_0$ and the support condition for $v$, this implies that 
\begin{align*}
& \tau \int_{\R_+}\sum\limits_{j \in \Z^d} \sum\limits_{k=1}^{d} \left[ h^{-2}((\cosh(\tau (\phi_j - \phi_{j+e_k}))-1)v_{j+e_k},v_j) \right.\\
& \quad \left. + h^{-2}((\cosh(\tau(\phi_j - \phi_{j-e_k}))-1)v_{j-e_k},v_j) \right]\,dt
\leq C \tau^3 \| v\|^2_{L^2}.
\end{align*}
Thus, for $c_0>1$ sufficiently large, the bound from \eqref{eq:symmetric_term} yields
\begin{align}
\label{eq:grad_bound}
\begin{split}
8 \tau \int_{\R_+}\sum\limits_{j \in \Z^d}\sum\limits_{k=1}^{d} \Big| \frac{v_{j +  e_k} - v_{j - e_k}}{2h} \Big|^2\,dt  + 8 \tau \|\p_t v\|^2_{L^2}
&\leq 8 \|S_{\phi}v\|_{L^2} \|v\|_{L^2} + C \tau^3 c_0^2 \|v\|^2_{L^2}\\
&\leq \frac{1}{2}\|S_{\phi}v\|^2_{L^2} + 32 \|v\|^2_{L^2} +  C\tau^3 c_0^2 \|v\|_{L^2}^2.
\end{split}
\end{align}
Returning to the Carleman estimate from \eqref{eq:Carl_aux} and inserting the bound \eqref{eq:grad_bound} for the gradient contribution, we hence deduce that 
\begin{align}
\label{eq:Carl_aux0a}
\begin{split}
&\|S_{\phi} v\|_{L^2}^2  + \|A_{\phi} v\|_{L^2}^2 + 2 \tau^3 c_0^3 ((1-t)^2 v, v) + 4 \tau c_0(\p_t v, \p_t v) + \tau \|\nabla_{\dis} v\|_{L^2}^2 \\
& \leq \|L_{\phi} v\|_{L^2}^2  + \frac{3}{4}\|S_\phi v\|_{L^2}^2 + 64 \|v\|_{L^2}^2 +  C \tau^3 c_0^2 \|v\|_{L^2}^2
 + C \tau (\|\nabla_{\dis} v\|_{0} ^2+ \tau^2 \|v\|_0^2 + \|\p_t v\|_0^2).
\end{split}
\end{align}
Now, using that for $c_0>1$, $\tau \geq \tau_0>1$ sufficiently large the new bulk terms on the right hand side of \eqref{eq:Carl_aux0a} can be absorbed into the left hand side, we obtain that for  $c_0>1$, $\tau \geq \tau_0>1$ and some $C>1$ (independent of $\tau, h$)
\begin{align}
\label{eq:Carl_aux1}
\begin{split}
&\|S_{\phi} v\|_{L^2}^2  + \|A_{\phi} v\|_{L^2}^2 +  2\tau^3 c_0^3 \|(1-t) v\|_{L^2}^2 + 4 \tau \|\p_t v \|_{L^2}^2  + \tau \|\nabla_{\dis} v\|_{L^2}^2\\
& \leq C\|L_{\phi} v\|_{L^2}^2  + C \tau (\|\nabla_{\dis} v\|_{0} ^2+ \tau^2 \|v\|_0^2 + \|\p_t v\|_0^2).
\end{split}
\end{align}
In order to conclude the argument, we finally return from the function $v$ to the original function $\tilde{u}= e^{-\tau \phi}v$. To this end, we observe that
\begin{align*}
\|\p_t v\|_{L^2}^2 &\geq \|e^{\tau \phi}\p_t \tilde  u\|_{L^2}^2  - 2\tau^2\|(\p_t \phi) \tilde u\|_{L^2}^2
\geq \|e^{\tau \phi}\p_t \tilde u\|_{L^2}^2  - 2c_0^2 \tau^2 \| \tilde  u\|_{L^2}^2,\\
\|\nabla_{\dis} v\|_{L^2}^2 & \geq \|e^{\tau \phi} \nabla_{\dis} \tilde  u\|_{L^2}^2 - 2\|(\nabla_{\dis} e^{\tau\phi }) \tilde  u\|_{L^2}^2
 \geq \|e^{\tau \phi} \nabla_{\dis} \tilde  u\|_{L^2}^2 - C\tau^2 \|e^{\tau \phi}|\nabla_{\dis} \phi| \tilde  u\|_{L^2}^2. 
\end{align*}
For $c_0>1$ sufficiently large, all of these contributions can be absorbed into the bulk $L^2$ term on the left hand side of \eqref{eq:Carl_aux1}. Arguing similarly for the boundary contributions, we obtain the desired result
\begin{align*}
&\tau^3 c_0^3 \| e^{\tau \phi} \tilde  u\|_{L^2}^2 + \tau \| e^{\tau \phi} \p_t \tilde  u\|^2 + \tau \|e^{\tau \phi}\nabla_{\dis} \tilde  u\|_{L^2}^2\\
& \leq C\|e^{\tau \phi} \D_{\dis} \tilde  u \|_{L^2}^2  + C \tau (\|e^{\tau \phi} \nabla_{\dis} u\|_{0} ^2+ \tau^2 \|e^{\tau \phi} u\|_0^2 + \|e^{\tau \phi} \p_t \tilde  u\|_0^2),
\end{align*}
which concludes the argument for the Carleman inequality from Theorem \ref{thm:Carl_disc}.
\end{proof}

\begin{rmk}
\label{rmk:tau}
Considering the conjugated operator and its symbol explains the restriction $\tau h\leq \delta_0$: Indeed, this exactly corresponds to the regime in which (by linearization) the conjugated operator is a linear perturbation of the continuum one (in a precise sense) which allows for the application of the explained methods. We also refer to \cite{BHLR10a} for similar constraints in the context of control theory of (semi-)discrete equations.
\end{rmk}

\subsection{Proof of Theorem \ref{thm:boundary_bulk_discrete}}

Using the Carleman estimate from Theorem \ref{thm:Carl_disc}, we next present the proof of our main boundary-bulk doubling inequality from Theorem \ref{thm:boundary_bulk_discrete}:

\begin{proof}[Proof of Theorem \ref{thm:boundary_bulk_discrete}]
The proof of Theorem \ref{thm:boundary_bulk_discrete} follows from Theorem \ref{thm:Carl_disc} by standard cut-off arguments using that the Carleman weight is largest at $t=0$. 
The error term arises from the fact that we can only consider semi-classical parameters $\tau \leq \delta_0 h_0^{-1}$.

We discuss the details of this.\\

\emph{Step 1: A cut-off argument.}
We begin by noting that for $t \in (0,1)$ the function $\phi$ as a function of $x\in (h\Z)^d$ and $t$ is monotone decreasing with level sets of finite length in the upper half-plane.
We next define $\bar{r}_0:= \max\{r\in (0,2/3): 2\varphi_{-}(\bar{r}_0)> \varphi_+(2/3)\}>0$. 
Here $\varphi_{-}(r) = \min\{\phi(x,t): |(x,t)|=r, \ (x,t)\in (h\Z)^d \times \R_+ \}$ and $\varphi_{+}(r) = \max\{\phi(x,t): |(x,t)|=r, \ (x,t)\in (h\Z)^d \times \R_+ \}$. 
Such a number exists due to the fact that $\phi(0,0)=0$, the monotonic decay of the weight function for $|(x,t)|\leq 1$, and the structure of its level sets.

Next we consider $w:= \eta \tilde{u}$, where $\eta$ is a cut-off function supported in $B_{3/4}^+$ which is equal to one on $B_{7/10}^+$,  and satisfies uniform bounds for its gradient and second derivatives. Moreover, we may assume that $\p_t \eta|_{(h\Z)^d \times \{0\}} = 0$.
Using the equation for $\tilde{u}$, we obtain that $w$ satisfies a similar equation but now involves a bulk inhomogeneity $f$. More precisely, we have that 
\begin{align*}
(-\p_t^2 - \D_{\dis}) w &=Vw+ f  \mbox{ in } (h \Z)^d \times \R_+,
\end{align*}
where 
\[
f(jh,t)=  -2 \p_t \eta \p_t \tilde{u} -  \p_t^2 \eta \tilde{u} - 2 \nabla_{\dis} \eta \cdot \nabla_{\dis} \tilde{u} - h^{-2}\sum_{k=1}^d\Big(\frac{u_{j+e_k}+u_{j-e_k}}{2}(\eta_{j+e_k}+\eta_{j-e_k}-2\eta_j)\Big).
\]
We observe that all contributions of this error term may be assumed to be supported in $B_{4/5}^+\setminus B_{2/3}^+$ by choosing the support of $\eta$ appropriately.
Inserting this into \eqref{eq:Carl}, we obtain
\begin{align}
\label{eq:Cut_Carl}
\begin{split}
&\tau^{3/2}\|e^{\tau \phi} w\|_{L^2((h \Z)^d \times \R_+ )} + \tau^{\frac{1}{2}}\|e^{\tau \phi} \nabla_{\dis} w\|_{L^2((h \Z)^d \times \R_+)}\\
&\leq C \tau^{3/2}(\|e^{\tau \phi} V w\|_{L^2((h\Z)^d \times \R_+)} + \|e^{\tau \phi} f\|_{L^2((h\Z)^d \times \R_+)}  \\
& \quad  + \tau^{\frac{3}{2}}\|e^{\tau \phi} (|w| + |\nabla_{\dis} w| + |\p_t w|)\|_{L^2((h\Z)^d \times \{0\})}).
\end{split}
\end{align}
Next, we observe that $\|e^{\tau \phi} V w\|_{L^2((h\Z)^d \times \R_+)} \leq C \|V\|_{L^{\infty}}\|e^{\tau \phi} w\|_{L^2((h\Z)^d \times \R_+)}$. For $\tau^{3/2}\geq 2C \|V\|_{L^{\infty}}$, we may hence absorb this term involving the bulk potential into the left hand side of the estimate \eqref{eq:Cut_Carl}.

Due to the support assumption for $\eta$ and the monotonicity of the weight (it decays with $|(x,t)|$ in a neighbourhood of zero), we may thus further bound this as follows 
\begin{align*}
&\tau^{3/2}e^{\tau \varphi_-(\bar{r}_0)} \| w\|_{L^2(B_{\bar{r}_0}^+ )} + \tau^{\frac{1}{2}}e^{\tau \varphi_{-}(\bar{r}_0)}\| \nabla_{\dis} w\|_{L^2(B_{\bar{r}_0}^+)}\\
&\leq C \tau^{3/2}(e^{\tau \varphi_{+}(2/3)}\| f\|_{L^2((h\Z)^d \times \R_+)} +  \| |w| + |\nabla_{\dis} w| + |\p_t w|\|_{L^2((h\Z)^d \times \{0\})}),
\end{align*}
where $\varphi_{-}(r)$ and $\varphi_+(r)$ are as above. By virtue of the structure of the weight, for $c_0>1$ sufficiently large, we have that, by definition, $0>\varphi_-(\bar{r}_0)>\varphi_{+}(2/3)$.  
Using the support conditions encoded in the cut-off function $\eta$ together with Caccioppoli's inequality (see for instance \cite[Lemma 4.1]{FBRRS20}) and the definition of $f$, we then further deduce that
\begin{align*}
& \| \tilde{u}\|_{L^2(B_{\bar{r}_0}^+)}
\leq C (e^{\tau (\varphi_{+}(2/3)- \varphi_-(\bar{r}_0))} \| \tilde{u} \|_{L^2(B_{1}^+)} +  e^{-\tau  \varphi_-(\bar{r}_0)} \| |u| + |\nabla_{\dis} u| + |\p_t \tilde{u}|\|_{L^2(B_1')}).
\end{align*}
We recall that by construction
\begin{align*}
\varphi_{+}(2/3)- \varphi_-(\bar{r}_0) \leq 0 \mbox{ and } - \varphi_-(\bar{r}_0) > 0. 
\end{align*}
Based on this and the restrictions on the size of $\tau$, we consider two cases: 
\begin{itemize}
\item If $\max\{\| \tilde{u} \|_{L^2(B_{1}^+)} ,\| |u| + |\nabla_{\dis} u| + |\p_t \tilde{u}|\|_{L^2(B_1')} \} = \| \tilde{u} \|_{L^2(B_{1}^+)}$, then we choose 
\begin{align*}
\tau = \min\Big\{\delta_0 h^{-1}, \frac{1}{-\varphi_{+}(2/3)}\log\Big(\frac{\| \tilde{u} \|_{L^2(B_{1}^+)} }{\| |u| + |\nabla_{\dis} u| + |\p_t \tilde{u}|\|_{L^2(B_1')}}\Big)\Big\}.
\end{align*}
We then obtain that
\begin{align*}
\| \tilde{u}\|_{L^2(B_{\bar{r}_0}^+)}
& \leq C e^{-\delta_0 h^{-1} |\varphi_+(2/3)- \varphi_-(\bar{r}_0)|} \| \tilde{u} \|_{L^2(B_{1}^+)} \\
& \quad + C \|\tilde{u}\|_{L^2(B_1^+)}^{\alpha} \left( \| |u| + |\nabla_{\dis} u| + |\p_t \tilde{u}| \|_{L^2(B_1')} \right)^{1-\alpha},
\end{align*} 
where $\alpha = \frac{\varphi_-(\bar{r}_0)}{\varphi_{+}(2/3)} \in (0,1)$.
\item If $\max\{\| \tilde{u} \|_{L^2(B_{1}^+)} ,\| |u| + |\nabla_{\dis} u| + |\p_t \tilde{u}|\|_{L^2(B_1')} \} > \| \tilde{u} \|_{L^2(B_{1}^+)}$, we then directly choose $\tau = \tau_0> \max\{1,\|V\|^{2/3}_{L^{\infty}}\}$ and bound the expression by
\begin{align*}
\| \tilde{u}\|_{L^2(B_{\bar{r}_0}^+)}
& \leq C\| |u| + |\nabla_{\dis} u| + |\p_t \tilde{u}| \|_{L^2(B_1')},
\end{align*} 
where $C >1$ is a constant depending on $\|V\|_{L^{\infty}}$.
\end{itemize}
This completes the proof of Theorem \ref{thm:boundary_bulk_discrete}.
\end{proof}

\subsection{A stability estimate for a linear inverse problem}
\label{sec:stab_est}

We deduce the estimate of Theorem~\ref{thm:application} as a consequence of a propagation of a smallness argument. This in turn relies on the doubling inequality from \cite{FBRRS20}, the result of Theorem \ref{thm:bdry_bulk} together with Caccioppoli's inequality and trace estimates.

\begin{proof}[Proof of Theorem \ref{thm:application}]
We denote by $\tilde{u}$ the semi-discrete Caffarelli-Silvestre extension of $f$.
With this, for $E= \Omega \times [0,C]$ and some large constant $C>1$ to be chosen below but fixed, we first invoke Lemma \ref{lem:L2_bd} in order to obtain the following estimate: $\|\tilde{u}\|_{L^2(E)} \leq C \|f\|_{H^1(W)}$.

Next, we argue similarly as in \cite{Rue21} and propagate information through the upper half plane. To this end, we consider a chain of balls in the upper half-plane connecting $W$ and $\Omega$. More precisely, let $x_0 \in \Omega \subset (h\Z)^d \times \{0\}$, $r_0= \bar{r}_0$ and let $(x_j)_{j\in \{1,\dots,N\}}$ and $(r_j)_{j\in \{1,\dots,N\}}$ denote points and radii such that 
\begin{itemize}
\item $x_j \in B_{r_{j-1}}^+(x_{j-1})$ and $B_{r_{j}/2}^+(x_j)\subset B_{r_{j-1}}^+(x_{j-1})$, where this denotes balls which are fully contained in the upper half plane if $j>1$,
\item for $j\in\{1,\dots,N\}$ the balls $B_{5 r_{j}}(x_j)$ are all contained in the upper half plane and only finitely many of them have overlap,
\item the balls cover $W \times [\mu,1]$ for $\mu \in (0,1)$ to be determined.
\end{itemize}
Without loss of generality, by scaling we assume that $B'_{1}(x_0)\subset \Omega$.

\begin{figure}
\centering
\begin{tikzpicture}[line cap=round,line join=round,>=triangle 45, scale=0.75]
\begin{axis}[
x=1.0cm,y=1.0cm,
axis lines=middle,
grid style=dashed,
ymajorgrids=true,
xmajorgrids=true,
xmin=-0.7821672074769789,
xmax=16.119827039555762,
ymin=-0.9758464598153493,
ymax=5.0225395582028804,
xtick={-0.0,1.0,...,16.0},
ytick={-0.0,1.0,...,5.0},
xticklabels=\empty,
yticklabels=\empty,
]
\clip(-0.7821672074769789,-0.9758464598153493) rectangle (16.119827039555762,5.0225395582028804);
\fill[line width=0.8pt,color=qqqqff,fill=qqqqff,fill opacity=0.10000000149011612] (3.,0.2) -- (5.,0.2) -- (5.,2.) -- (3.01,2.) -- cycle;
\fill[line width=0.8pt,color=zzttqq,fill=zzttqq,fill opacity=0.10000000149011612] (14.,2.) -- (11.,2.) -- (11.,0.) -- (14.,0.) -- cycle;
\draw [line width=0.8pt,color=qqqqff] (3.,0.2)-- (5.,0.2);
\draw [line width=0.8pt,color=qqqqff] (5.,0.2)-- (5.,2.);
\draw [line width=0.8pt,color=qqqqff] (5.,2.)-- (3.01,2.);
\draw [line width=0.8pt,color=qqqqff] (3.01,2.)-- (3.,0.2);
\draw [line width=0.8pt,color=zzttqq] (14.,2.)-- (11.,2.);
\draw [line width=0.8pt,color=zzttqq] (11.,2.)-- (11.,0.);
\draw [line width=0.8pt,color=zzttqq] (14.,0.)-- (14.,2.);
\draw [line width=0.4pt] (11.,2.5) circle (1.2cm);
\draw [line width=2.8pt,color=ffqqqq] (10.,0.)-- (14.997985579888068,0.);
\draw [line width=0.4pt] (9.,2.5) circle (1.2cm);
\draw [line width=0.4pt] (10.,2.5) circle (1.2cm);
\draw [line width=0.4pt] (3.,2.5) circle (1.2cm);
\draw [line width=0.4pt] (4.,2.5) circle (1.2cm);
\draw [line width=0.4pt] (5.,2.5) circle (1.2cm);
\draw [line width=0.4pt] (6.,2.5) circle (1.2cm);
\draw [line width=0.4pt] (7.,2.5) circle (1.2cm);
\draw [line width=0.4pt] (8.,2.5) circle (1.2cm);
\draw [line width=0.4pt] (2.5,1.25) circle (0.7cm);
\draw [line width=0.4pt] (3.49920443637167,1.2535356214921929) circle (0.7cm);
\draw [line width=0.4pt] (3.9981888336334075,1.2535356214921929) circle (0.7cm);
\draw [line width=0.4pt] (4.502983373885995,1.2516758379755484) circle (0.7cm);
\draw [line width=0.4pt] (3.254632076587796,0.7499719010165523) circle (0.4cm);
\draw [line width=0.4pt] (3.5039935212070126,0.7530571411983982) circle (0.4cm);
\draw [line width=0.4pt] (3.7504281300709352,0.7498977231360402) circle (0.4cm);
\draw [line width=0.4pt] (4.002218704792949,0.7514294937378291) circle (0.4cm);
\draw [line width=0.4pt] (4.253119939133888,0.7488951378353954) circle (0.4cm);
\draw [line width=0.4pt] (4.498952461669958,0.7463607819329616) circle (0.4cm);
\draw [line width=0.4pt] (4.747319340108462,0.7539638496402629) circle (0.4cm);
\draw [line width=0.4pt] (3.1996924648965543,0.30208629704338746) circle (0.25cm);
\draw [line width=0.4pt] (3.4043910889426026,0.2977310071700673) circle (0.25cm);
\draw [line width=0.4pt] (3.6003791332420105,0.2977310071700673) circle (0.25cm);
\draw [line width=0.4pt] (3.7963671775414185,0.30208629704338746) circle (0.25cm);
\draw [line width=0.4pt] (4.001065801587466,0.30208629704338746) circle (0.25cm);
\draw [line width=0.4pt] (4.201409135760195,0.30208629704338746) circle (0.25cm);
\draw [line width=0.4pt] (4.397397180059603,0.2977310071700673) circle (0.25cm);
\draw [line width=0.4pt] (4.59774051423233,0.30208629704338746) circle (0.25cm);
\draw [line width=0.4pt] (4.798083848405058,0.30644158691670764) circle (0.25cm);
\draw [line width=0.4pt] (3.000545882469815,0.29825884456667495) circle (0.25cm);
\draw [line width=0.4pt] (5.,0.2996424622646916) circle (0.25cm);
\draw [line width=2.8pt,color=qqqqff] (3.,0.)-- (5.,0.);
\draw [color=qqqqff](3.106863211329494,-0.03271315403143788) node[anchor=north west] {$W=\operatorname{supp}f$};
\draw [color=ffqqqq](11,-0.07) node[anchor=north west] {data $(-\D_{\dis})^{1/2}f|_{\Omega}$};
\draw [color=ffqqqq](14.3,0) node[anchor=north west] {$\Omega$};
\draw [line width=0.4pt] (12.28815181095648,2.123105809069128) circle (0.4cm);
\draw [line width=0.4pt] (12.591315778940743,1.599592754433063) circle (0.4cm);
\draw [line width=0.4pt] (12.64935564655909,1.0481807747991918) circle (0.4cm);
\draw (15,-0.04626335613604268) node[anchor=north west] {$(h\mathbb{Z})^d$};
\draw (0.12581874831642018,4.6150061678480085) node[anchor=north west] {$t>0$};
\draw [color=qqqqff](5.220694739647856,1.1461544290691799) node[anchor=north west] {$W\times [\mu,1]$};
\end{axis}
\end{tikzpicture}
\caption{The setting of the inverse problem from Theorem \ref{thm:application} and an illustration of its proof. We measure the data $(-\D_{\dis})^{1/2}f|_{\Omega}$ restricted to $\Omega \subset (h\Z)^d$ for a function $f$ which is compactly supported in $W \subset (h\Z)^d$, a set disjoint from $\Omega$. The associated inverse problem of recovering $f$ from the measured data is highly ill-posed. Corresponding to this, we deduce a logarithmic modulus of continuity up to an exponentially (in the lattice size $h$) small correction term which takes into account the discrete nature of the problem. As in the continuum setting, the key idea of deducing this relies on a chain of balls argument transferring information from $\Omega$ to $W$ through the upper half-space.} \label{fig:inverse}
\end{figure}
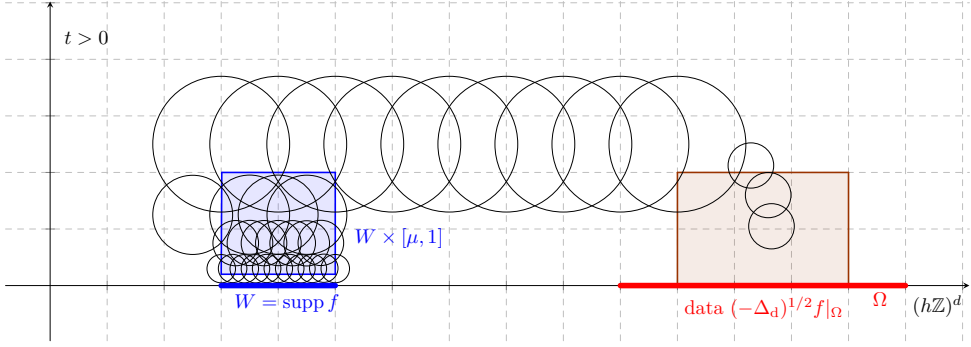

Now, using that $\supp(f) \subset W$, by the estimate from Theorem \ref{thm:boundary_bulk_discrete}, we have that 
\begin{align*}
\|\tilde{u}\|_{L^2(B_{\bar{r}_0}^+(x_0))} / \|f\|_{H^1(W)} 
\leq C \left( \|\p_t \tilde{u}\|_{L^2(B_1'(x_0))} / \|f\|_{H^1(W)} \right)^{\alpha} + C e^{-C h^{-1}} ,
\end{align*}
where we have used that $\|\tilde{u}\|_{L^2(B_1^+(x_0))} \leq C \|f\|_{H^1(W)}$ by Lemma \ref{lem:L2_bd} and the fact that $\tilde{u}|_{B_1'(x_0)} = 0$ since $\supp(f)\subset W$ and $B_1'(x_0)\subset \Omega$ by assumption.
We combine this with the interior three balls inequality in its scaled form (see \cite[Corollary 5.2]{FBRRS20}) and by using second condition in the chain of balls, to next obtain  
\begin{align*}
\|\tilde{u}\|_{L^2(B^+_{r_1}(x_1))} / \|f\|_{H^1(W)} 
&\leq C\big( \|\tilde{u}\|_{L^2(B^+_{r_1/2}(x_1))}/ \|f\|_{H^1(W)} \big)^{\alpha} + C e^{-C h^{-1}}\\
& \leq C\big( \|\tilde{u}\|_{L^2(B_{r_0}^+(x_0))}/ \|f\|_{H^1(W)}\big)^{\alpha} + C e^{-C h^{-1}} \\
& \leq C\big( \|\p_t \tilde{u}\|_{L^2(B_{1}'(x_0))}/ \|f\|_{H^1(W)} \big)^{\alpha^2} + 2 C e^{-C\alpha h^{-1}}.
\end{align*}
Here, invoking Lemma \ref{lem:L2_bd}, we estimated that $\|\tilde{u}\|_{L^2(B^+_{2r_1})}/\|f\|_{H^1(W)}\leq C$ and have used that $\supp_{\R^d}(\tilde{u}) = \supp(f) \subset W$.
We now iterate this estimate, use that the balls in the chain cover $W\times [\mu,1]$ and obtain that for $N \sim C(d, \Omega, W)(|\log(\mu)|+1)$ 
\begin{align}
\label{eq:iterated}
\|\tilde{u}\|_{L^2(W \times [\mu,1])}/ \|f\|_{H^1(W)} \leq C^{N}C\left( \|\p_t \tilde{u}\|_{L^2(B_{1}'(x_0))}/ \|f\|_{H^1(W)} \right)^{\alpha^N} + N C^N e^{-c \alpha^{N-1} h^{-1}}.
\end{align}
We remark that this leads to the requirement that $\mu\geq C h_0$ for $C>1$ sufficiently large, in order to still be able to apply the scaled three balls inequality from \cite[Corollary 5.2]{FBRRS20}.
Further, by Poincar\'e's inequality, trace estimates (see Lemma \ref{lem:trace_aux}),
we deduce that 
\begin{align*}
\|f\|_{L^2(W)} &\leq C \mu^{1/2}\|\p_t \tilde{u}\|_{L^2(W \times [0,1])} + C \mu^{-\frac{1}{2}}\|\tilde{u}\|_{L^2(W \times [\mu,1])} \\
&\leq C \mu^{1/2}\|f\|_{H^1(W)} + C \mu^{-\frac{1}{2}}\|\tilde{u}\|_{L^2(W \times [\mu,1])} .
\end{align*}
Here, in the second line, we have used that by the properties of the Caffarelli--Silvestre extension we have that $\|\p_t \tilde{u}\|_{L^2((h\Z)^d \times \R_+)} \leq C \|f\|_{H^1((h\Z)^d)}$.
Combining this with the above bound \eqref{eq:iterated}, implies that for $N \sim C(d,\Omega,W)(|\log(\mu)|+1)$
\begin{align*}
\|f\|_{L^2(W)} 
&\leq C \mu^{1/2}\|f\|_{H^1(W)} +  \mu^{-\frac{1}{2}} C^{N}\left( \|\p_t\tilde{u}\|_{L^2(\Omega \times \{0\})}/ \|f\|_{H^1(W)} \right)^{\alpha^N} \|f\|_{H^1(W)}\\
& \quad + \mu^{-\frac{1}{2}}N C^N e^{-c \alpha^{N-1} h^{-1}} \|f\|_{H^1(W)}.
\end{align*}
Possibly enlarging the constants $C>1$ we may estimate this as follows:
\begin{align}
\label{eq:mu_bd}
\begin{split}
\|f\|_{L^2(W)} 
&\leq C \mu^{1/2}\|f\|_{H^1(W)} +   C^{N}C\big( \|\p_t\tilde{u}\|_{L^2(\Omega \times \{0\})}/ \|f\|_{H^1(W)} \big)^{\alpha^N} \|f\|_{H^1(W)}\\
& \quad + C^N e^{-c \alpha^{N-1} h^{-1}} \|f\|_{H^1(W)}.
\end{split}
\end{align}

We next recall that $\|\p_t \tilde{u}\|_{L^2(\Omega \times \{0\})} = \|(-\D_{\dis})^{\frac{1}{2}}f\|_{L^2(\Omega)}$ and choose
\begin{align*}
\mu \sim \max\Big\{C h_0,\Big|\log\Big( \frac{\|(-\D_{\dis})^{\frac{1}{2}} f\|_{L^2(\Omega)}}{\|f\|_{H^1(W)}} \Big)\Big|^{-2 \nu} \Big\}
\end{align*}
for some constant $\nu \in (0,\frac{1}{3})$ to be determined below. We set  $\epsilon:= \Big( \frac{\|(-\D_{\dis})^{\frac{1}{2}} f\|_{L^2(\Omega)}}{\|f\|_{H^1(W)}}\Big) \in (0,1)$ and observe, that after this choice of $\nu$, the smallness hypothesis on $h_0$ \eqref{eq:continuum_aprox}, which reads
\begin{align*}
C h_0 \leq \big|\log( \epsilon) \big|^{-1+{\nu}} |\log(-C\log(\epsilon))|^{-1}.
\end{align*}
implies that in fact $\mu=C|\log(\epsilon)|^{-2 \nu}$ for some large constant $C>1$. Inserting this into estimate \eqref{eq:mu_bd}, due to the dependence of $N$ on $\mu$ we thus obtain
\begin{align}
\label{eq:est_combine}
\begin{split}
\|f\|_{L^2(W)}
&\leq C|\log(\epsilon)|^{-\nu} \|f\|_{H^1(W)}
 + C\exp\big(-C|\log(
\epsilon)|^{1-2C_1|\log(\alpha)|{\nu}}
 \big) \|f\|_{H^1(W)}\\
& \quad + C\exp\big(C|\log(-C\log(\epsilon))|-C h^{-1}
|\log(\epsilon)|^{-2 C_1|\log(\alpha)|{\nu}}
 \big)  \|f\|_{H^1(W)},
\end{split} 
\end{align}
where $C_1>0$ is a constant depending on $d, \Omega, W$. 
Indeed, with $\epsilon\in (0,1)$ as above, using the relation between $N$ and $\mu$, we have rewritten the second contribution  on the right hand side of \eqref{eq:mu_bd} as
\begin{align*}
\exp\big( C|\log(\mu)| + C \log(\epsilon) \mu^{C|\log(\alpha)|} \big)
&= \exp\big( 2C\nu |\log(-C\log(\epsilon))| + C \log(\epsilon) |\log(\epsilon)|^{-2\nu C_1|\log(\alpha)|} \big) \\
& \leq \tilde{C}\exp\big( \tilde{C} \log(\epsilon) |\log(\epsilon)|^{-2\nu C_1|\log(\alpha)|} \big),
\end{align*}
for some appropriate constant $\tilde{C}>1$. Indeed, observe that the $2C\nu|\log(-C\log(\epsilon))|$ term in the exponential above is absorbed by the second term (note also that, since $\epsilon\in (0,1)$ then $|\log(\epsilon)|=-\log(\epsilon)$), whenever $1-2\nu C_1|\log(\alpha)|>0$. For the third right hand side contribution in \eqref{eq:mu_bd} we have argued similarly.
Choosing $\nu\in (0,1)$ such that $1-2C_1|\log(\alpha)|\nu= \nu$ and using that as $\epsilon \rightarrow 0$ the second right hand side term in \eqref{eq:est_combine} is dominated by the first right hand side term, then implies 
\begin{align}
\label{eq:est_combinea}
\begin{split}
\|f\|_{L^2(W)}
&\leq C|\log(\epsilon)|^{-\nu} \|f\|_{H^1(W)} + C\exp\big(C|\log(-C\log(\epsilon))|-C h^{-1}
|\log(\epsilon)|^{-1+\nu}
 \big)  \|f\|_{H^1(W)}\\
&\leq C|\log(\epsilon)|^{-\nu} \|f\|_{H^1(W)} + C\exp\big(-\tilde{C}h^{-1}
|\log(\epsilon)|^{-1+ \nu}
 \big)  \|f\|_{H^1(W)}
\end{split} 
\end{align}
with $\nu = \frac{1}{2C_1|\log(\alpha)| + 1}>0$ and a possibly enlarged constant $C>1$ and a new constant $\tilde{C}>0$. In the last line, we have used the smallness condition for $h_0>0$.
\end{proof}

\section[Fractional Discrete Laplacian in the Torus]{Fractional Powers of the Laplacian in the Discrete Torus}
\label{sec:torus}

Let $M\in \N$ and define $\mathbb{T}_{\frac{2\pi}{M}}$ to be the circle of circumference $2\pi/M$, which we identify with the interval $[-\frac{\pi}{M},\frac{\pi}{M})$. The Pontryagin dual is $\Z M$. Thus, the Pontryagin dual of $\mathbb{T}_{2\pi}$ is $\Z$. If we discretise $\T_{2\pi}$ by taking $M$ equispaced points, we have that the Pontryagin dual of $\Z\frac{2\pi}{M}/\Z 2\pi=\frac{2\pi}{M}\Z_M$ is $\Z_M$.
The following dictionary might help:

\begin{center}
{\tabulinesep=1.3mm
\begin{tabu}{ c c } \hline
\textbf{Group} &
	\textbf{Dual} \\
\hline\hline
 $\R/\Z 2\pi=\mathbb{T}_{2\pi}$& $\Z$\\
$\Z\frac{2\pi}{M}/\Z 2\pi=\frac{2\pi}{M}\Z_M$& $\Z_M$ \\
$\mathbb{T}_{2\pi/M}$& $\Z M$  \\
\hline
\end{tabu}
}
\end{center}

Let us take $M=2N+1$, $N\in \N$. So if in particular we discretise $\mathbb{T}_{2\pi}$ by taking $2N+1$ equispaced mesh points with distance $h=\dfrac{2\pi}{2N+1}$, we can define the \textit{discrete torus} as the set of points 
$$
A_h^N:=h\Z_{2N+1}:=h\{-N,\ldots,N\}.
$$ 
The Pontryagin dual of $h\Z_{2N+1}$ is $\Z_{2N+1}$. The $d$-dimensional discrete torus is $A_{h,d}^N=(h\{-N,\ldots,N\})^d$

The discrete Fourier transform of a finite bounded sequence $\{u_j\}_{j\in\{-N,\ldots,N\}^d}$ is given by 
\begin{equation}
\label{eq:DFT}
\widehat{u}^D_{m}=\frac{1}{(2N+1)^d}\sum_{j\in\{-N,\ldots,N\}^d}u_je^{-2\pi i m\cdot\frac{ j}{2N+1}},\quad m\in\{-N,\ldots,N\}^d.
\end{equation}
Recall that the notation is $u_j:=u(jh)$.

The fractional Laplacian on $A_{h,d}^N$ is initially defined as 
$$
(-\Delta_{A_{h,d}^N})^s(\varphi)_j:=\sum_{k\in \{-N,\ldots,N\}^d}\Big(\sum_{i=1}^d\frac{4}{h^2}\sin^2\Big(k_i\frac{\pi}{2N+1}\Big)\Big)^sc_k(\varphi)e^{ik\cdot\frac{2\pi j}{2N+1}}
$$
where $c_k(\varphi)=\widehat{\varphi}^D_{k}$ is the Fourier coefficient of $\varphi:A_{h,d}^N\to \R$ given in \eqref{eq:DFT}. Observe that, alike the situation with the fractional Laplacian on the continuous torus, see \cite{RS_trans}, $(-\Delta_{A_{h,d}^N})^s$ preserves the class of smooth functions on $A_{h,d}^N$. By symmetry, we can define this operator for any function $v$ which is a periodic distribution on $A_{h,d}^N$, namely 
$$
\langle (-\Delta_{A_{h,d}^N})^sv,\phi\rangle_{C^{\infty}(A_{h,d}^N)}:=\sum_{j\in\{-N,\ldots,N\}^d} v_j(-\Delta_{A_{h,d}^N})^s\phi_j, \quad \phi \in C^{\infty}(A_{h,d}^N).
$$

\subsection{A transference formula for fractional discrete Laplacians}
\label{subsec:transf}

Our main objective in this section is the derivation of a transference formula for the fractional Laplacian on the torus. To this end, we rely on a suitable Poisson summation formula on the \textit{dual} of the discrete torus. While this is known, we briefly recall it with our normalization conventions and in our set-up: Let $f:(h\Z)^d \to \R$, $f \in \mathcal{S}((h\Z)^d)$. For $x\in \{-N,\ldots, N\}^d$, we compute the Fourier coefficients of $\sum_{m\in \Z^d}f\big(h(x+m(2N+1))\big)$ namely, for $j\in\{-N,\ldots,N\}^d$,
\begin{align*}
&\frac{1}{(2N+1)^d}\sum_{x\in\{-N,\ldots,N\}^d}\Big[\sum_{m\in \Z^d}f\big(h(x+m(2N+1))\big)\Big]e^{-ix\cdot\frac{2\pi j}{2N+1}}\\
&=\frac{1}{(2N+1)^d}\sum_{m\in \Z^d}\sum_{x\in \{-N,\ldots, N\}^d}f\big(h(x+m(2N+1))\big)e^{-ix\cdot\frac{2\pi j}{2N+1}}\\
&=\frac{1}{(2N+1)^d}\sum_{m\in \Z^d}\sum_{y\in\{- N,\ldots,N\}^d+m(2N+1)}f_ye^{-iy\cdot\frac{2\pi j}{2N+1}}\\
&=\frac{1}{(2N+1)^d}\sum_{y\in \Z^d}f_ye^{-iy\cdot\frac{2\pi j}{2N+1}}=\frac{1}{(2N+1)^d}\mathcal{F}_{(h\Z)^d}(f)\Big(\frac{2\pi j}{2N+1}\Big).
\end{align*}
Inverting this, we obtain the following Poisson summation formula: for $x\in \{-N,\ldots,N\}^d$,
\begin{equation}
\label{eq:PSF}
\sum_{m\in \Z^d}f(h(x+m(2N+1)))=\frac{1}{(2N+1)^d}\sum_{ j\in \{-N,\ldots, N\}^d}\mathcal{F}_{(h\Z)^d}(f)\Big(\frac{2\pi j}{2N+1}\Big)e^{ix\cdot\frac{2\pi j}{2N+1}}. 
\end{equation}

Now for a function $u:(h \Z)^d\to \R$ we define its \textit{periodisation} as the function $p_{\Sigma}u: A_{h,d}^N\to \R$ given formally by
\begin{equation}
\label{eq:perio}
p_{\Sigma}u_j:=\sum_{l\in \Z^d}u\big(h(l(2N+1)+j)\big), \quad j\in \{-N,\ldots, N\}^d.
\end{equation}
For a function $v: A_{h,d}^N\to \R$ we define its \textit{repetition} $Rv: (h\Z)^d\to \R$ by 
$$
Rv_m:=\sum_{l\in \Z^d}v\big(h(m-l(2N+1))\big)\chi_{A_{h,d}^N}\big(h(m-l(2N+1))\big). 
$$
This is nothing but the $A_{h,d}^N$-periodic function on $(h\Z)^d$ that coincides with $v$ on $A_{h,d}^N$. The periodicity is easy to check: let $m\in \{-N,\ldots, N\}^d$, $j\in (2N+1)\Z^d$. We have
\begin{align*}
Rv_{m+j}&=\sum_{l\in \Z^d}v\big(h(m+j-l(2N+1))\big)\chi_{A_{h,d}^N}\big(h(m+j-l(2N+1))\big)\\
&=\sum_{l\in \Z^d}v\big(h(m-l(2N+1))\big)\chi_{A_{h,d}^N}\big(h(m-l(2N+1))\big)=Rv_m.
\end{align*}

Recall that, for $u:A_{h,d}^N\to \R$ with $h=\dfrac{2\pi}{2N+1}$ we are considering the Laplacian on the dual of the discrete torus
$$
\Delta_{A_{h,d}^N}u_j:=\frac{1}{h^2}\sum_{i=1}^d\big(u(h(j+e_i))-2u(hj)+u(h(j-e_i))\big), \quad j\in \{-N,\ldots, N\}^d.
$$
 We aim to prove the following. 
\begin{thm}
\label{thm:transF}
Let $v$ be a bounded function on $A_{h,d}^N$ and $h=\dfrac{2\pi}{2N+1}$. Then its repitition $Rv$ is a function in $\ell_s$ which defines a distribution in $\mathcal{S}_s'$ and such that
\begin{equation}
\label{eq:idtransF}
\sum_{l\in \Z^d}Rv_l(-\Delta_{\dis})^s\varphi_l=\sum_{j\in \{-N,\ldots,N\}^d}v_j(-\Delta_{A_{h,d}^N})^s(p_{\Sigma}\varphi)_j, \quad \varphi\in\mathcal{S}((h\Z)^d).
\end{equation}
\end{thm}
\begin{proof}
Let us check that $Rv\in \ell_s$. We compute
\begin{align*}
&\sum_{l\in \Z^d}|Rv_l|(1+|l|)^{-(d+2s)}\\
&=\sum_{l\in \Z^d}\Big[\sum_{k\in \Z^d}|v\big(h(l-k(2N+1))\big)|\chi_{A_{h,d}^N}\big(h(l-k(2N+1))\big)\Big](1+|l|)^{-(d+2s)}\\
&=\sum_{k\in \Z^d}\sum_{l\in \{-N,\ldots,N\}^d+k(2N+1)}|v\big(h(l-k(2N+1))\big)|(1+|l|)^{-(d+2s)}\\
&=\sum_{k\in \Z^d}\sum_{j\in \{-N,\ldots,N\}^d}|v_j|\big(1+|j+k(2N+1|)^{-(d+2s)}\big)\\
&=\sum_{j\in \{-N,\ldots,N\}^d}|v_j|p_{\Sigma}[(1+|\cdot|)^{-(d+2s)}]_j.
\end{align*}
The function $\big(1+|\cdot|\big)^{-(d+2s)}$ is summable, and so is $|p_{\Sigma}[\big(1+|\cdot|\big)^{-(d+2s)}]$, (see  \cite[Chapter VII]{SW}, where this is shown for functions in $\R^d$ and their periodisation), so the last expression above is finite whenever $v$ is bounded. Hence, $Rv\in \ell_s$ and the left hand side of \eqref{eq:idtransF} is absolutely convergent.

Let, for $f \in \mathcal{S}((h\Z)^d)$,
$$
f_j=\sum_{k\in \{-N,\ldots,N\}^d}c_k(f)e^{ik\cdot\frac{2\pi j}{2N+1}}, \quad j\in \{-N,\ldots,N\}^d
$$
where
$$
c_k(f)=\frac{1}{(2N+1)^d}\sum_{j\in \{-N,\ldots,N\}^d}f_je^{-ik\cdot\frac{2\pi j}{2N+1}}, \quad k\in \{-N,\ldots,N\}^d.
$$
For $\varphi\in \mathcal{S}((h\Z)^d)$, we have that $p_{\Sigma}\varphi$ is summable on $A_{h,d}^N$ and
\begin{align*}
c_k(p_{\Sigma}\varphi)&=\frac{1}{(2N+1)^d}\sum_{j\in \{-N,\ldots,N\}^d}(p_{\Sigma} \varphi)_je^{-ik\cdot\frac{2\pi j}{2N+1}}\\
&=\frac{1}{(2N+1)^d}\sum_{j\in \{-N,\ldots,N\}^d}\sum_{l\in \Z^d}\varphi\big(h(j+l(2N+1))\big)e^{-ik\cdot\frac{2\pi j}{2N+1}}\\
&=\frac{1}{(2N+1)^d}\sum_{l\in \Z^d}\sum_{n\in \{-N,\ldots,N\}^d+l(2N+1)}\varphi_ne^{-in\cdot\frac{2\pi k}{2N+1}}\\
&=\frac{1}{(2N+1)^d}\sum_{n\in \Z^d}\varphi_ne^{-in\cdot\frac{2\pi k}{2N+1}}=\frac{1}{(2N+1)^d}\mathcal{F}_{(h\Z)^d}(\varphi)\Big(\frac{2\pi k}{2N+1}\Big), \quad k\in \{-N,\ldots,N\}^d.
\end{align*}
Then we have proven
\begin{equation}
\label{eq:coeff}
c_k(p_{\Sigma}\varphi)=\frac{1}{(2N+1)^d}\mathcal{F}_{(h\Z)^d}(\varphi)\Big(\frac{2\pi k}{2N+1}\Big)=\frac{1}{(2N+1)^d}\mathcal{F}_{(h\Z)^d}(\varphi)(kh), \quad k\in \{-N,\ldots,N\}^d.
\end{equation}
Further, since $\varphi, \mathcal{F}_{(h\Z)^d}(\varphi)\in  \mathcal{S}((h\Z)^d)$, we have
$$
(p_{\Sigma}\varphi)_j:=\frac{1}{(2N+1)^d}\sum_{k\in \Z^d}\mathcal{F}_{(h\Z)^d}(\varphi)\Big(\frac{2\pi k}{2N+1}\Big)e^{ik\cdot\frac{2\pi j}{2N+1}},
$$
and this series converges absolutely, see \cite[Chapter VII]{SW}. It follows that $p_{\Sigma}\varphi$ is a smooth function on $A_{h,d}^N$ and hence $(-\Delta_{A_{h,d}^N})^s(p_{\Sigma}\varphi)$ is also smooth, thus the right hand side of \eqref{eq:idtransF} is absolutely convergent.

Moreover, for $j\in\{-N,\ldots, N\}^d$, by using \eqref{eq:coeff} in the first equality below and \eqref{eq:PSF} in the third one,
\begin{align*}
(-\Delta_{A_{h,d}^N})^s(p_{\Sigma}\varphi)_j&:=\sum_{k\in \{-N,\ldots,N\}^d}\Big(\sum_{i=1}^d\frac{4}{h^2}\sin^2\Big(k_i\frac{\pi}{2N+1}\Big)\Big)^sc_k(p_{\Sigma}\varphi)e^{ik\cdot\frac{2\pi j}{2N+1}}\\
&=\frac{1}{(2N+1)^d}\sum_{k\in \{-N,\ldots,N\}^d}\Big(\sum_{i=1}^d\frac{4}{h^2}\sin^2\Big(k_i\frac{\pi}{2N+1}\Big)\Big)^s\mathcal{F}_{(h\Z)^d}(\varphi)(kh)e^{ik\cdot j\frac{2\pi}{2N+1}}\\
&=\frac{1}{(2N+1)^d}\sum_{k\in \{-N,\ldots,N\}^d}\mathcal{F}_{(h\Z)^d}[(-\Delta_{\dis})^s\varphi](k)e^{ik\cdot j\frac{2\pi}{2N+1}}\\
&=\sum_{m\in \Z^d}[(-\Delta_{\dis})^s\varphi]\big(h(m(2N+1)+j)\big)\\
&=p_{\Sigma}[(-\Delta_{\dis})^s\varphi]_j,
\end{align*}
thus
\begin{equation}
\label{eq:lapla}
(-\Delta_{A_{h,d}^N})^s(p_{\Sigma}\varphi)_j=p_{\Sigma}[(-\Delta_{\dis})^s\varphi]_j,\quad j\in\{-N,\ldots,N\}^d.
\end{equation}
We use the above to prove the desired transference formula. Indeed,
\begin{align*}
&\sum_{l\in \Z^d}Rv_l(-\Delta_{\dis})^s\varphi_l\\
&=\sum_{l\in \Z^d}\Big[\sum_{k\in \Z^d}v\big(h(l-k(2N+1))\big)\chi_{A_{h,d}^N}\big(h(l-k(2N+1))\big)\Big](-\Delta_{\dis})^s\varphi_l\\
&=\sum_{k\in \Z^d}\sum_{l\in \{-N,\ldots,N\}^d+k(2N+1)}v\big(h(l-k(2N+1))\big)(-\Delta_{\dis})^s\varphi_l\\
&=\sum_{k\in \Z^d}\sum_{j\in \{-N,\ldots,N\}^d}v_j(-\Delta_{\dis})^s\varphi\big(h(j+k(2N+1))\big)\\
&=\sum_{j\in \{-N,\ldots,N\}^d}v_jp_{\Sigma}[(-\Delta_{\dis})^s\varphi]_j\\
&=\sum_{j\in \{-N,\ldots,N\}^d}v_j(-\Delta_{A_{h,d}^N})^s(p_{\Sigma}\varphi)_j,
\end{align*}
where we used \eqref{eq:lapla} in the last equality. 
\end{proof}

\subsection{First application: Proof of Theorem \ref{thm:transP}}

Via transference from the pointwise formula for the fractional discrete Laplacian we will deduce the pointwise formula for fractional powers of the Laplacian on the discrete torus. Recall the notations for $A_{h,d}^N$ and $A_h^N$ at the beginning of this section.

We are going to prove that, for any $\phi\in C^{\infty}(A_{h,d}^N)$, the following identity holds
\begin{equation}
\label{eq:desired}
\langle (-\Delta_{A_{h,d}^N})^sv,\phi\rangle_{C^{\infty}(A_{h,d}^N)}=\sum_{j\in\{-N,\ldots,N\}^d}q_j\phi_j,
\end{equation}
where the function $q$ is given by 
$$
q(j)=\sum_{\substack{k\in \{-N,\ldots,N\}^d\\
k\neq j}}\big(v(j)-v(k)\big)K_s^{A_{h,d}^N}(j-k), \quad j\in \{-N,\ldots,N\}^d.
$$
 First observe that $Rv\in \ell_s((h\Z)^d)$. We have, for  $j\in \{-N,\ldots,N\}^d$,
\begin{align}
\label{eq:hj}
\notag &(-\Delta_{\dis})^s(Rv)_j\\
&=\sum_{\substack{k\in \Z^d\\
k\neq j}}\big(Rv_j-Rv_k\big)K_s^h(j-k)\\
&\notag=\sum_{l\in \Z^d}\sum_{\substack{m\in\{-N,\ldots,N\}^d\\
m\neq j}}\big(Rv_j-Rv\big(h(m-l(2N+1))\big)\big)K_s^h(j-m+l(2N+1))\\
&\notag=\sum_{l\in \Z^d}\sum_{\substack{m\in\{-N,\ldots,N\}^d\\
m\neq j}}\big(v_j-v_m\big)K_s^h(j-m+l(2N+1))\\
&\notag=\sum_{\substack{m\in \{-N,\ldots,N\}^d\\
m\neq j}}\big(v_j-v_m\big)\sum_{l\in \Z^d}K_s^h(j-m+l(2N+1))\\
&\notag=\sum_{\substack{m\in\{-N,\ldots,N\}^d\\
m\neq j}}\big(v_j-v_m\big)K_s^{A_{h,d}^N}(j-m)=q_j.
\end{align}
It is easy to check that $(-\Delta_{\dis})^s(Rv)$ is a $(h(2N+1))^d$-periodic function. Indeed, let $j\in\{-N,\ldots,N\}^d$, $l\in \Z^d$. By the periodicity of $Rv$ then
\begin{align*}
&(-\Delta_{\dis})^s(Rv)\big(h(j+l(2N+1))\big)\\
&=\sum_{\substack{k\in\Z^d\\
k\neq j+l(2N+1)}}\big(Rv\big(j+l(2N+1)\big)-Rv_k\big)K_s^h(j+l(2N+1)-k)\\
&=\sum_{\substack{k\in\Z^d\\
k\neq j+l(2N+1)}}\big(Rv_j-Rv_k\big)K_s^h(j+l(2N+1)-k)\\
&=\sum_{\substack{k\in\Z^d\\
k\neq j}}\big(Rv_j-Rv(h(k+l(2N+1)))\big)K_s^h(j-k)\\
&=\sum_{\substack{k\in\Z^d\\
k\neq j}}\big(Rv_j-Rv_k\big)K_s^h(j-k)\\
&=(-\Delta_{\dis})^s(Rv)_j.
\end{align*}
Let $\phi\in C^{\infty}(A_{h,d}^N)$. Moreover, it can be proved analogously as in \cite[Lemma 2.3]{RS_trans} that there exists $\varphi\in \mathcal{S}((h\Z)^d)$ such that 
\begin{equation}
\label{eq:phij}
\phi_j=p_{\Sigma}\varphi_j, \quad j\in \{-N,\ldots,N\}^d.
\end{equation} 
Then, by Theorem \ref{thm:transF} and \eqref{eq:hj},
\begin{align*}
\langle (-\Delta_{A_{h,d}^N})^sv,\phi\rangle_{C^{\infty}(A_{h,d}^N)}&=\langle (-\Delta_{A_{h,d}^N})^sv,p_{\Sigma}\varphi\rangle_{C^{\infty}(A_{h,d}^N)}\\
&=\langle (-\Delta_{\dis})^s(Rv),\varphi\rangle_{(\mathcal{S}((h\Z)^d))'}\\
&=\sum_{j\in \Z^d}(-\Delta_{\dis})^s(Rv)_j\varphi_j\\
&=\sum_{k\in \Z^d}\sum_{j\in\{-N,\ldots,N\}^d+k(2N+1)}(-\Delta_{\dis})^s(Rv)_j\varphi_j\\
&=\sum_{k\in \Z^d}\sum_{j\in\{-N,\ldots,N\}^d}(-\Delta_{\dis})^s(Rv)\big(h(j+k(2N+1))\big)\varphi\big(h(j+k(2N+1))\big)\\
&=\sum_{k\in \Z^d}\sum_{j\in\{-N,\ldots,N\}^d}(-\Delta_{\dis})^s(Rv)_j\varphi\big(h(j+k(2N+1))\big)\\
&=\sum_{j\in\{-N,\ldots,N\}^d}(-\Delta_{\dis})^s(Rv)_jp_{\Sigma}\varphi_j\\
&=\sum_{j\in\{-N,\ldots,N\}^d}q_j\phi_j,
\end{align*}
as desired.

The explicit expression for the case $d=1$ easily follows from the expression for the kernel in~\eqref{eq:kernelKs}.

\subsection{Second application: extension problem}

The Caffarelli--Silvestre extension problem characterisation is valid also for the fractional discrete Laplacian, see \cite{CRSTV16}. By the general theory \cite{ST10}, it is also valid for the fractional discrete torus. We can derive the latter from the result of $(h\Z)^d$ in \cite{CRSTV16}.  

\begin{proof}[Proof of Theorem \ref{thm:transEO}]
Consider $u=Rv\in \ell_s$. Let $U$ be the solution to the extension problem for $u$
$$
\begin{cases}
\Delta_{\dis}U+\frac{1-2s}{t}U_t+U_{tt}=0, \quad &\text{ in } (h\Z)^d\times (0,\infty),\\
U(hm,0)=u(hm), \quad &\text{ on }(h\Z)^d. 
\end{cases}
$$
From \cite{ST10} it is known that $U(hm,t)=P_t^s\ast u(hm)$ for a suitable Poisson kernel $P_t^s(hm)$. Observe that, for $j\in \{-N,\ldots, N\}^d$,
\begin{align*}
U(hj,t)&=\sum_{k\in \Z^d}P_t^s(h(j-k))u_k\\
&=\sum_{k\in \Z^d}\sum_{m\in \{-N,\ldots,N\}^d}P_t^s(h(j-m+k(2N+1)))u(h(m+k(2N+1)))\\
&=\sum_{k\in \Z^d}\sum_{m\in \{-N,\ldots,N\}^d}P_t^s(h(j-m+k(2N+1)))Rv(h(m+k(2N+1)))\\
&=\sum_{k\in \Z^d}\sum_{m\in \{-N,\ldots,N\}^d}P_t^s(h(j-m+k(2N+1)))v(hm)\\
&=\sum_{m\in \{-N,\ldots,N\}^d}(p_{\Sigma}P_t^s)_{j-m}v_m=v\ast_{A_{h,d}^N}(p_{\Sigma}P_t^s)_j.
\end{align*}
By uniqueness, it follows that $V(\cdot,t)=v\ast_{A_{h,d}^N}(p_{\Sigma}P_t^s)$, for each $t>0$. Moreover, by Theorem \ref{thm:transF} and the extension problem for the fractional discrete Laplacian, 
\begin{align*}
&c_s\sum_{j\in \{-N,\ldots,N\}^d}v_j(-\Delta_{A_{h,d}^N})^s(p_{\Sigma}\varphi)_j=c_s\sum_{l\in \Z^d}u_l(-\D_{\dis})^s\varphi_l\\
&\quad =-\lim_{t\to0^+}\sum_{l\in \Z^d}t^{1-2s}U_t(hl,t)\varphi_l\\
&\quad=-\lim_{t\to0^+}\sum_{k\in \Z^d}\sum_{j\in \{-N,\ldots,N\}^d}t^{1-2s}U_t(h(j+k(2N+1)),t)\varphi(h(j+k(2N+1)))\\
&\quad=-\lim_{t\to0^+}\sum_{j\in \{-N,\ldots,N\}^d}t^{1-2s}U_t(hj,t)(p_{\Sigma}\varphi)_j\\
&\quad=-\lim_{t\to0^+}\sum_{j\in \{-N,\ldots,N\}^d}t^{1-2s}V_t(hj,t)(p_{\Sigma}\varphi)_j.
\end{align*}
Taking \eqref{eq:phij} into account, \eqref{eq:DtoN} follows. 
\end{proof}

\subsection{Failure of global unique continuation: proof of Theorem~\ref{prop:UCP_torus_new}}

We are going to prove Theorem~\ref{prop:UCP_torus_new}. 
If $j\in X$, the requirement $u_j=0$ implies
$$
(-\D_{A_h^N})^s u_j=-\sum_{\substack{m=-N \\ m \neq j}}^N u_m K_s^{A_h^N}(j-m).
$$
We compute $(-\D_{\dis})^s u_j$ for all $j\in X$, which yields $M$ equations. On the other hand, set $u_j=0$ for all $j\in \{-N,\dots,N\}\setminus Y$, where $Y$ is a set of cardinality $M+1$ such that $Y\cap X=\emptyset$. This is only possible if $M+(M+1)\le 2N+1$ or, equivalently, $M\le N$. We have that setting $(-\D_{\dis})^su_j=0$ for all $j\in X$ gives then a homogeneous system of $M$ equations and $M+1$ unknowns. Since there are more unknowns than equations, there are infinitely many nontrivial solutions.

\begin{rmk}
If we have a function $u$ such that $u_j=0=(-\D_{A_h^N})^su_j$ for $j\in X$, $X$ being a set of cardinality $m>N$, by the same reasoning as above, we have $m$ equations coming from the fact that the fractional Laplacian vanishes in $X$, and, since $u_j=0$ in $X$, there are $2N+1-m$ values of $u$ where we do not have a priori information. The assumption that $m>N$ implies that we have a homogeneous system with more equations than unknowns. Hence, we can no longer argue as above, in order to conclude the failure of the global UCP. It would be interesting to study the unique continuation properties in this situation further, and to know whether/ in which situation the global UCP holds.
\end{rmk}

\appendix

\section{An Estimate for the $d$-Dimensional Kernel}
\label{subsec:technical-lemmas}

In this subsection we aim to prove an upper estimate for the kernel $K_s^h(m)$, $m\in \Z^d$, which will be useful for the counterexamples in higher dimensions. First we need some technical lemmas.

\begin{lem}
\label{lem:integral}
Let $\alpha\ge0$, $\beta>\alpha+1$, $A,B>0$. Then
$$
 \int_0^1\frac{A^{\alpha}(1-A)^\alpha}{(A+B)^{\beta}}\,dA\le \frac{\Gamma(\alpha+1)\Gamma(\beta-\alpha-1)}{\Gamma(\beta)}
\frac{1}{B^{\beta-\alpha-1}}.
$$
\end{lem}

\begin{proof}
Observe that 
\begin{align*}
 \int_0^1\frac{A^{\alpha}(1-A)^\alpha}{(A+B)^{\beta}}\,dA
\le \int_0^1\frac{A^{\alpha}}{(A+B)^{\beta}}\,dA&=\frac{1}{B^{\beta}}
 \int_0^1\frac{A^\alpha}{(1+\frac{A}{B})^{\beta}}\,dA\\ &=\frac{1}{B^{\beta}} \int_0^{1/B}\frac{(Bz)^\alpha B}{(1+z)^{\beta}}\,dz\\
 &\le \frac{1}{B^{\beta}} \int_0^{\infty}\frac{(Bz)^\alpha B}{(1+z)^{\beta}}\,dz\\
 &\le
\frac{\Gamma(\alpha+1)\Gamma(\beta-\alpha-1)}{\Gamma(\beta)}
\frac{1}{B^{\beta-\alpha-1}},
\end{align*}
where we used the definition of the Beta function
$$
B(x,y):=\int_0^{\infty}\frac{t^{x-1}}{(1+t)^{x+y}}\,dt=\frac{\Gamma(x)\Gamma(y)}{\Gamma(x+y)},
$$
whenever $\operatorname{Re} x, \operatorname{Re} y>0$.
\end{proof}

Let us use the notation
$\|m\|_1:=\sum_{i=1}^d|m_i|$.

\begin{lem}
\label{lem:horror1}
Let $m=(m_1,\ldots,m_d)\in \N^d$. Define, for $i=1,\ldots,d$,
\begin{equation}
\label{ed:U}
U_i:=\frac{u_i^{m_i-1/2}(1-u_i)^{m_i-1/2}}{\Gamma(m_i+1/2)}.
\end{equation}
Let $1\le r\le d$ be fixed and $b>\frac{d}{2}$. Then, for $C>0$,
\begin{equation*}
 \int_{[0,1]^{d-r+1}}
\frac{\prod_{i=r}^dU_i\,du_i}
{(C+u_r+\cdots+u_d)^{\|m\|_1+b}}\le
\frac{\Gamma(\sum_{i=1}^{r-1}m_i+b-\frac{d-(r-1)}{2})}
{C^{\sum_{i=1}^{r-1}m_i+b-\frac{d-(r-1)}{2}}\Gamma(\|m\|_1+b)}.
\end{equation*}
\end{lem}

\begin{proof}
Without loss of generality, we begin with the integral in the variable~$u_r$. We use Lemma~\ref{lem:integral}, with $A=u_r$, $B=(C+u_{r+1}+\cdots+u_d)$, $\alpha=m_r-1/2$ and $\beta=\|m\|_1+b$. Let us denote by $I$ the integral in the statement of the lemma. Then,
\begin{align*}
I&\le\frac{\Gamma(m_r+1/2)\Gamma(\|m\|_1+b-(m_r+1/2))}{\Gamma(\|m\|_1+b)}
\frac{1}{\Gamma(m_r+1/2)} \\
&\,\,\,\times \int_{[0,1]^{d-r}}
\frac{\prod_{i={r+1}}^dU_i\,du_i}
{(C+u_{r+1}+\cdots+u_d)^{\|m\|_1+b-(m_r+1/2)}}.
\end{align*}
We apply again Lemma~\ref{lem:integral} in the integral with the variable $u_{r+1}$, with $A=u_{r+1}$, $B=C+u_{r+2}+\cdots+u_d$, $\alpha=m_{r+1}-1/2$ and $\beta=\|m\|_1+b-(m_r+1/2)$. Observe that $\beta-\alpha=\|m\|_1-m_r+b+1/2>1$.
Hence,
\begin{align*}
I&\le\frac{\Gamma(\|m\|_1+b-(m_r+\frac12))}{\Gamma(\|m\|_1+b)}\frac{\Gamma(m_{r+1}+\frac12)
\Gamma(\|m\|_1+b-(m_r+\frac12)-(m_{r+1}+\frac12))}
{\Gamma(\|m\|_1+b-(m_r+\frac12))}\frac{1}{\Gamma(m_{r+1}+\frac12)} \\
&\qquad \qquad\times \int_{[0,1]^{d-r-1}}
\frac{\prod_{i=r+2}^dU_i\,du_i}
{\big(C+u_{r+2}+\cdots+u_n\big)^{\|m\|_1+b-(m_r+1/2)-(m_{r+1}+1/2)}}.
\end{align*}
Then, by following the same procedure $d-r-1$ times, we finally obtain that the integral in the statement is bounded by
\begin{equation*}
\frac{\Gamma(\|m\|_1+b-\sum_{i=r}^d(m_i+1/2))}
{\Gamma(\|m\|_1+b)}\frac{1}{C^{\|m\|_1+b-\sum_{i=r}^d(m_i+1/2)}}.
\end{equation*}
 The conclusion follows easily.
\end{proof}

\begin{lem}
\label{lem:integralUs}
Let $m=(m_1,\ldots,m_d)\in \N^d$. Let $U_i$ be as in \eqref{ed:U}, for $i=2,\ldots,d$.
Let $a>0$, $b>\frac{d}{2}$. Then,
$$
\int_{[0,1]^d}
\frac{u_1^{m_1-\frac12+a}(1-u_1)^{m_1-\frac12}}
{(u_1+\cdots+u_d)^{\|m\|_1+b}} \prod_{i=2}^dU_i\,du_i\,du_1
\le \frac{\Gamma(b+m_1-\frac{d-1}{2})\Gamma(a-b+\frac{d}{2})\Gamma(m_1+1/2)}
{\Gamma(\|m\|_1+b)\Gamma(a-b+m_1+\frac{d+1}{2})}.
$$
\end{lem}

\begin{proof}
First, by Lemma~\ref{lem:horror1} with $r=2$ and $C=u_1$ we obtain that
$$
\int_{[0,1]^{n-1}}
\frac{\prod_{i=2}^dU_i\,du_i}
{(u_1+\cdots+u_d)^{\|m\|_1+b}}
\le \frac{\Gamma(\|m\|_1+b-\sum_{i=2}^d(m_i+\frac12))}
{\Gamma(\|m\|_1+b)} \frac{1}{u_1^{\|m\|_1+b-\sum_{i=2}^d(m_i+\frac12)}}.
$$
Thus, for the remaining integral in the variable $u_1$, we have
$$
\int_0^1 u_1^{a-b+\frac{d}{2}-1}(1-u_1)^{m_1-\frac12}
\,du_1
= \frac{\Gamma(a-b+\frac{d}{2})\Gamma(m_1+\frac12)}{\Gamma(a-b+m_1+\frac{d+1}{2})}.
$$
This, together with the estimate above, yields the desired result.
\end{proof}

The next identity is known as Schl\"afli's integral
representation of Poisson type for modified Bessel functions (see \cite[(5.10.22)]{Lebedev}), and it is valid for a real number $\nu>-\frac12$:
\begin{equation}
\label{eq:Schlafli}
I_{\nu}(z) = \frac{z^{\nu}}{\sqrt{\pi}\,2^{\nu}\Gamma(\nu+1/2)}
\int_{-1}^1e^{-zs} (1-s^2)^{\nu-1/2}\,ds, \quad
|\arg z|<\pi, \quad \nu>-\frac12.
\end{equation}
Now we are ready to prove the upper estimate for the kernel.

\begin{lem}
\label{lem:upperEstimateKernelFL}
The discrete kernel $K_s^h(m)$ defined in \eqref{Ksh} is bounded from above as follows
\begin{equation}
  \label{eq:estimateAbove}
  K_s^h(m)\leq  \frac{1}{h^{2s}}\frac{2^{d(d+2s-1)}4^s\Gamma(\frac{d}{2}+s)\Gamma(\|m\|_1-s)}{\pi^{d/2}|\Gamma(-s)|\Gamma(\|m\|_1+d+s)},
  \end{equation}
  for $m\in \Z^d$, $m\neq(0,\ldots,0)$.
\end{lem}

\begin{proof}
Let $m\in \N^d$.
We use Schl\"afli's integral \eqref{eq:Schlafli}
to obtain
\begin{align*}
\int_0^\infty G(m,t)\,\frac{dt}{t^{1+s}}&=\pi^{-d/2}\int_0^{\infty}e^{-2dt}\int_{[-1,1]^d}\prod_{i=1}^d\frac{(2t)^{m_i}}{2^{m_i}\Gamma(m_i+1/2)}e^{-2tw_i}(1-w_i^2)^{m_i-1/2}\,dw_i\frac{dt}{t^{1+s}}\\
&=\pi^{-d/2}\int_{[-1,1]^d}\prod_{i=1}^d\int_0^{\infty}e^{-2dt}e^{-2t(w_1+\cdots+w_d)}t^{\|m\|_1-1-s}\,dt\frac{(1-w_i^2)^{m_i-1/2}\,dw_i}{\Gamma(m_i+1/2)}.
\end{align*}
We make the change of variables $(1+w_i)=2u_i$, $i=1,\ldots,d$, then the last integral equals
\begin{align}
\label{eq:integralU}
\notag&\pi^{-d/2}\int_{[0,1]^d}\prod_{i=1}^d\int_0^{\infty}4^{\|m\|_1}
e^{-4t(u_1+\cdots+u_d)}t^{\|m\|_1-1-s}\,dt\frac{u_i^{m_i-1/2}(1-u_i)^{m_i-1/2}\,du_i}
{\Gamma(m_i+1/2)}\\
&\quad=\frac{4^s\Gamma(\|m\|_1-s)}{\pi^{d/2}} \int_{[0,1]^d}
\frac{\prod_{i=1}^dU_i\,du_i}
{(u_1+\cdots+u_d)^{\|m\|_1-s}}
\end{align}
 where $U_i$ is that in \eqref{ed:U} (observe that the integral in $t$ converges when $s<\|m\|_1$, with $\|m\|_1\ge 1$). If we multiply and divide by $(u_1+\cdots+u_d)^{d+2s}$ in the integral above, then 
 \begin{align*}
 & \int_{[0,1]^d}
\frac{(u_1+\cdots+u_d)^{d+2s}}
{(u_1+\cdots+u_d)^{\|m\|_1+d+s}}\prod_{i=1}^dU_i\,du_i\\
&= \frac{1}{\Gamma(m_k+1/2)}\int_{[0,1]^d}
\frac{u_k^{m_k-1/2}(1-u_k)^{m_k-1/2}(u_1+\cdots+u_d)^{d+2s}}
{(u_1+\cdots+u_d)^{\|m\|_1+d+s}}\prod_{\substack{i=1\\ i\neq k}}^dU_i\,du_i\,du_k\\
&\le\frac{2^{d(d+2s)}}{\Gamma(m_k+1/2)}\\
&\times \int_{[0,1]^d}
\frac{\big(u_k^{m_k-1/2+d+2s}(1-u_k)^{m_k-1/2}+(u_1^{d+2s}+\cdots u_{k-1}^{d+2s}+u_{k+1}^{d+2s}+\cdots+u_d^{d+2s})\big)}
{(u_1+\cdots+u_d)^{\|m\|_1+d+s}}\prod_{\substack{i=1\\ i\neq k}}^dU_i\,du_i\,du_k,
 \end{align*}
 where we used the convexity of the function $|x|^p$, $p>1$, repeatedly, so that $(u_1+(u_2+\cdots+u_d))^{d+2s}\le 2^{d+2s-1}\big(u_1^{d+2s}+(u_2+\cdots+u_d)^{d+2s}\big)$. Thus,
 $K_s^h(m)\le \frac{2^{d(d+2s-1)}4^s\Gamma(\|m\|_1-s)}{h^{2s}\pi^{d/2}|\Gamma(-s)|} \sum_{k=1}^dT_k$, where
\begin{equation*}
T_k=
\frac{1}{\Gamma(m_k+\frac12)} \int_{[0,1]^d}
\frac{u_k^{m_k-\frac12+d+2s}(1-u_k)^{m_k-\frac12}}
{(u_1+\cdots+u_d)^{\|m\|_1+d+s}}\prod_{\substack{i=1 \\ i\neq k}}^dU_i\,du_i\,du_k
\end{equation*}
If $k=1$, by using Lemma~\ref{lem:integralUs} with $a=d+2s$ and $b=d+s$, we obtain
\begin{equation*}
T_1\le \frac{\Gamma(m_1+\frac{d+1}{2}+s)}{\Gamma(m_1+\frac12)\Gamma(\|m\|_1+d+s)}
\frac{\Gamma(\frac{d}{2}+s)\Gamma(m_1+\frac12)}
{\Gamma(m_1+\frac{d+1}{2}+s)}.
\end{equation*}
For the rest of $T_k$, we obtain analogous estimates.
With this, we conclude that
$$
K_s^h(m)\leq  \frac{1}{h^{2s}}\frac{2^{d(d+2s-1)}4^s\Gamma(\frac{d}{2}+s)\Gamma(\|m\|_1-s)}{\pi^{d/2}|\Gamma(-s)|\Gamma(\|m\|_1+d+s)}, 
$$
and we are done.
\end{proof}

\begin{rmk}
\label{rmk:decay}
In view of the asymptotics \eqref{eq:asymG}, one deduces that 
$$
K_s^h(m)\lesssim_{d,s} \frac{1}{h^{2s}}\|m\|_1^{-d-2s} \quad \text{ as } \|m\|_1\to \infty.
$$
\end{rmk}
\section{Fractional Laplacian on the discrete torus via the heat semigroup}
\label{subsec:heatDT}
We will discuss how the pointwise formula for the fractional powers of the Laplacian on the discrete torus shown in Theorem \ref{thm:transP}  can be also obtained via the semigroup language with the corresponding heat kernel.

First, we are interested in solving 
\begin{equation}
\label{eq:semidT}
\begin{cases}
\partial_tw_j=\Delta_{A_{h,d}^N}w_j, \quad &\text{ in } A_{h,d}^N\times (0,\infty),\\
w_j(0)=u_j, \quad &\text{ on } A_{h,d}^N,
\end{cases}
\end{equation}
where $\Delta_{A_{h,d}^N}$ is defined in \eqref{eq:deltaTorus}.
Let us compute the fundamental solution, i.e., the heat kernel. We can rewrite \eqref{eq:DFT} as
$$
\widehat{u}_{m}^D=\frac{h^d}{(2\pi)^d}\sum_{j\in\{-N,\ldots,N\}^d}u_je^{- i m\cdot jh},\quad m\in\{-N,\ldots, N\}^d
$$
and the latter solves 
$$
\partial_t \widehat{u}^D_{m}=\frac{1}{h^{2}}\Big(\sum_{j=1}^d\big(2\cos(m_j h)-2\big)\Big)\widehat{u}^D_{m}.
$$
The fundamental solution $u_j:=G(jh,t)$ satisfies $\widehat{u}^D_{m}=1$ for all $m$, so 
\begin{align*}
G(jh,t)&=\frac{h^d}{(2\pi)^d}\sum_{m\in \{-N,\ldots,N\}^d}\Big(\prod_{j=1}^de^{\frac{2t}{h^2}(\cos(m_j h)-1)}\Big)e^{ i m\cdot jh}\\
&=\frac{1}{(2N+1)^d}\sum_{m\in \{-N,\ldots,N\}^d}e^{ij\cdot mh}\Big(\prod_{j=1}^de^{\frac{2t}{h^2}(\cos(m_j h)-1)}\Big), 
\end{align*}
for $j\in\{-N,\ldots, N\}^d$.
Now, for $\rho \in\{-N,\ldots,N\}^d$, we use the Poisson summation formula
$$
\frac{1}{(2N+1)^d}\sum_{x\in A_{h,d}^N}e^{i\rho\cdot x}f(x)=\sum_{k\in \Z^d }\widehat{f}(\rho+k(2N+1)),
$$
where
$$
\widehat{f}(k):=\frac{1}{(2\pi)^{d}}\int_{[0,2\pi]^d}e^{-ik\cdot x}f(x)\,dx.
$$
Then, taking $\rho=j$ and noting that $\dfrac{2\pi}{h}=2N+1$, we write
$$
G(jh,t)=\sum_{\ell\in \Z^d}\widehat{f}\Big(\frac{2\pi \ell}{h}+j\Big).
$$
Here,
$$
\widehat{f}\Big(\frac{2\pi \ell}{h}+j\Big)=\frac{1}{(2\pi)^d}\int_{[0,2\pi]^d}e^{-i(\frac{2\pi \ell}{h}+j)\cdot x}\prod_{k=1}^de^{\frac{2t}{h^2}(\cos x_k-1)}\,dx=e^{-2dt/h^2}\prod_{k=1}^dI_{\ell_k(2N+1)+j_k}\Big(\frac{2t}{h^2}\Big),
$$
where $I_k$ is the modified Bessel function of the first kind and order $k\in \Z$,
and we used the identity (see \cite[10.32.3]{OlMax})
$$
I_n(z)=\frac{1}{\pi}\int_0^{\pi}e^{z\cos \theta}\cos(n\theta)\,d\theta, \quad n\in \Z.
$$
Therefore,
\begin{align*}
G(jh,t)&=\frac{h^d}{(2\pi)^d}\sum_{m\in \{-N,\ldots,N\}^d}\Big(\prod_{j=1}^de^{\frac{2t}{h^2}(\cos(m_j h)-1)}\Big)e^{ i m\cdot jh}\\
&=e^{-2dt/h^2}\sum_{\ell\in\Z^d}\prod_{k=1}^dI_{\ell_k(2N+1)+j_k}\Big(\frac{2t}{h^2}\Big).
\end{align*}
The above identity, in the case $d=1$, was obtained in \cite{KN}.

Observe that 
\begin{align*}
\sum_{j\in\{-N,\ldots,N\}^d}G(jh,t)&=\sum_{j\in\{-N,\ldots,N\}^d}\sum_{\ell\in\Z^d}\prod_{i=1}^de^{-2t/h^2}I_{\ell_i(2N+1)+j_i}\Big(\frac{2t}{h^2}\Big)\\
&=\prod_{i=1}^d\sum_{k_i\in \Z}e^{-2t/h^2}I_{k_i}\Big(\frac{2t}{h^2}\Big)=1
\end{align*}
where we used that for all $k_i\in \Z$ there exists a unique $\ell_i\in \Z$ such that $k_i=\ell_i(2N+1)+j_i$, with $j_i\in \{-N,\ldots,N\}$ and that $\sum_{k\in \Z}e^{-2t/h^2}I_k\big(\frac{2t}{h^2}\big)=1$. Hence, the solution of the equation \eqref{eq:semidT} is given by 
\begin{align*}
e^{t\Delta_{A_{h,d}^N}}u_j&=\sum_{m\in \{-N,\ldots,N\}^d}G(jh-mh,t)u_m\\
&=\sum_{m\in \{-N,\ldots,N\}^d}e^{-2dt/h^2}\sum_{\ell\in \Z^d}\prod_{k=1}^dI_{\ell_k(2N+1)+j_k-m_k}\Big(\frac{2t}{h^2}\Big)u_m.
\end{align*}

Now we can define the fractional powers
\begin{align*}
(-\Delta_{A_{h,d}^N})^su_j&=\frac{1}{\Gamma(-s)}\int_0^{\infty}(e^{t\Delta_{A_{h,d}^N}}u_j-u_j)\frac{dt}{t^{1+s}}\\
&=\frac{1}{\Gamma(-s)}\int_0^{\infty}\sum_{\substack {m\in \{-N,\ldots,N\}^d\\ m\neq j}}e^{-2dt/h^2}\sum_{\ell\in \Z^d}\prod_{k=1}^dI_{\ell_k(2N+1)+j_k-m_k}\Big(\frac{2t}{h^2}\Big)(u_m-u_j)\frac{dt}{t^{1+s}}\\
&=\sum_{\substack {m\in \{-N,\ldots,N\}^d\\ m\neq j}}(u_m-u_j)\sum_{\ell\in \Z^d}\frac{1}{\Gamma(-s)}\int_0^{\infty}e^{-2dt/h^2}\prod_{k=1}^dI_{\ell_k(2N+1)+j_k-m_k}\Big(\frac{2t}{h^2}\Big)\frac{dt}{t^{1+s}}\\
&=:\sum_{\substack {m\in \{-N,\ldots,N\}^d\\ m\neq j}}(u_j-u_m)K_s^{A_{h,d}^N}(j-m),
\end{align*}
In view of \eqref{Ksh} we conclude 
\begin{equation}
\label{eq:pformulaDiscreteT}
(-\Delta_{A_{h,d}^N})^s u_j = \sum_{\substack{m\in \{-N,\ldots,N\}^d \\ m \neq j}} (u_j - u_m) K_s^{A_{h,d}^N}(j-m), \quad j\in \{-N,\ldots,N\}^d.
\end{equation}
where 
$$
K_s^{A_{h,d}}(j):=\sum_{\ell\in \Z^d}K_s^h(j+\ell(2N+1)).
$$

In the case $d=1$, in view of the formula (see for instace \cite[p. 305]{PBM2}), valid for $\operatorname{Re} c>0$, $-\operatorname{Re} \nu<\operatorname{Re} \alpha <1/2$,
$$
\int_0^{\infty}e^{-ct}I_{\nu}(ct)t^{\alpha-1}\,dt=\frac{(2c)^{-\alpha}}{\sqrt \pi}\frac{\Gamma(1/2-\alpha)\Gamma(\alpha+\nu)}{\Gamma(\nu+1-\alpha)},
$$
and using that $I_{-m}(t)=I_m(t)$ we have
\begin{align*}
&\sum_{\ell=-\infty}^{\infty}I_{\ell(2N+1)+j-m}\Big(\frac{2t}{h^2}\Big)\\
&=\sum_{\substack{ \ell \in \Z \\ \ell(2N+1)>-(j-m)}}I_{\ell(2N+1)+j-m}\Big(\frac{2t}{h^2}\Big)+\sum_{\substack{ \ell \in \Z \\ \ell(2N+1)<-(j-m)}}I_{\ell(2N+1)+j-m}\Big(\frac{2t}{h^2}\Big)\\
&=\sum_{\substack{ \ell \in \Z \\ \ell(2N+1)>-(j-m)}}I_{\ell(2N+1)+j-m}\Big(\frac{2t}{h^2}\Big)+\sum_{\substack{ \ell \in \Z \\ \ell(2N+1)<-(j-m)}}I_{-\ell(2N+1)-(j-m)}\Big(\frac{2t}{h^2}\Big)\\
\end{align*}
and thus, since $m\neq j$
\begin{align*}
&\frac{1}{|\Gamma(-s)|}\int_0^{\infty}e^{-2t/h^2}\sum_{\ell=-\infty}^{\infty}I_{\ell(2N+1)+j-m}\Big(\frac{2t}{h^2}\Big)\frac{dt}{t^{1+s}}\\
&\quad =\frac{1}{|\Gamma(-s)|}\sum_{\substack{ \ell \in \Z \\ \ell(2N+1)>-(j-m)}}\frac{1}{\sqrt \pi}\Big(\frac{4}{h^2}\Big)^s\frac{\Gamma(1/2+s)\Gamma(j-m+\ell(2N+1)-s)}{\Gamma(j-m+\ell(2N+1)+1+s)}\\
&\qquad +\frac{1}{|\Gamma(-s)|}\sum_{\substack{ \ell \in \Z \\ \ell(2N+1)<-(j-m)}}\frac{1}{\sqrt \pi}\Big(\frac{4}{h^2}\Big)^s\frac{\Gamma(1/2+s)\Gamma(-\ell(2N+1)-(j-m)-s)}{\Gamma(-\ell(2N+1)-(j-m)+1+s)}\\
&\quad=\frac{1}{|\Gamma(-s)|}\sum_{\ell \in \Z }\frac{1}{\sqrt \pi}\Big(\frac{4}{h^2}\Big)^s\frac{\Gamma(1/2+s)\Gamma(|j-m+\ell(2N+1)|-s)}{\Gamma(|j-m+\ell(2N+1)|+1+s)}=:K^{A_{h,1}^N}_s(j-m).
\end{align*}
We conclude that 
\begin{equation}
\label{eq:pformulaDiscreteT}
(-\Delta_{A_{h,1}^N})^s u_j = \sum_{\substack{m=-N \\ m \neq j}}^{N} (u_j - u_m) K_s^{A_{h,1}^N}(j-m), \quad j\in \{-N,\ldots,N\}.
\end{equation}



\begin{thebibliography}{999999999}

\bibitem[AS22]{AS22}
Giovanni~S Alberti and Matteo Santacesaria.
\newblock {\em Infinite-dimensional inverse problems with finite measurements.}
\newblock {Arch. Ration. Mech. Anal.} \textbf{243} (2022), no. 1, 1--31. 

\bibitem[AdHGS17]{AdHGS17}
Giovanni Alessandrini, Maarten de~Hoop, Romina Gaburro, and Eva Sincich.
\newblock {\em Lipschitz stability for the electrostatic inverse boundary value
  problem with piecewise linear conductivities.}
\newblock {J. Math. Pures Appl.} (9) \textbf{107} (2017), no. 5, 638--664.
   
\bibitem[AV05]{AV05}
Giovanni Alessandrini and Sergio Vessella.
\newblock {\em Lipschitz stability for the inverse conductivity problem.}
\newblock {Adv. in Appl. Math.} \textbf{35} (2005), no.2, 207--241.

\bibitem[BEO15]{BEO15}
Lucie Baudoin, Sylvain Ervedoza, and Axel Osses.
\newblock {\em Stability of an inverse problem for the discrete wave equation and convergence results.}
\newblock {J. Math. Pures Appl. (9)} \textbf{103} (2015), no. 6, 1475--1522.

\bibitem[BL15]{BL15}
Katar{\'\i}na Bellov{\'a} and Fang-Hua Lin.
\newblock {\em Nodal sets of {S}teklov eigenfunctions.}
\newblock  {Calc. Var. Partial Differential Equations} \textbf{54} (2015), no. 2, 2239--2268.

\bibitem[BDHQ13]{BDHQ13}
Elena Beretta, Maarten~V De~Hoop, and Lingyun Qiu.
\newblock {\em Lipschitz stability of an inverse boundary value problem for a
  {S}chr\"odinger-type equation.}
\newblock {\em SIAM J. Math. Anal.} \textbf{45} (2013), no. 2, 679--699.
  
\bibitem[BHLR10a]{BHLR10a}
Franck Boyer, Florence Hubert, and J{\'e}r{\^o}me Le~Rousseau.
\newblock {\em Discrete {C}arleman estimates for elliptic operators and uniform
  controllability of semi-discretized parabolic equations.}
\newblock {J. Math. Pures Appl.} \textbf{93} (2010), no. 3, 240--276.

\bibitem[BHLR10b]{BHR10}
Franck Boyer, Florence Hubert, and J{\'e}r{\^o}me Le~Rousseau.
\newblock {\em Discrete {C}arleman estimates for elliptic operators in arbitrary
  dimension and applications.}
\newblock {SIAM J. Control Optim.} \textbf{48} (2010), no. 8, 5357--5397.


\bibitem[BLMS21]{BLMS17}
Lev Buhovski, Alexander Logunov, Eugenia Malinnikova, and Mikhail Sodin.
\newblock {\em A discrete harmonic function bounded on a large portion of
  $\mathbb{Z}^2$ is constant.}
\newblock To appear in {Duke Math. J.} (2021), preprint \href{https://arxiv.org/abs/1712.07902}{arXiv:1712.07902}.

\bibitem[CS07]{CS07}
Luis Caffarelli and Luis Silvestre.
\newblock {\em An extension problem related to the fractional {L}aplacian.}
\newblock {Comm. Partial Differential Equations} \textbf{32} (2007), no. 8,1245--1260.

\bibitem[CGR{\etalchar{+}}17]{CGRTV17}
\'{O}scar Ciaurri, T.~Alastair Gillespie, Luz Roncal, Jos\'{e}~L. Torrea, and
  Juan~Luis Varona.
\newblock {\em Harmonic analysis associated with a discrete {L}aplacian.}
\newblock {J. Anal. Math.} \textbf{132} (2017), 109--131.

\bibitem[CRS{\etalchar{+}}18]{CRSTV16}
\'{O}scar Ciaurri, Luz Roncal, Pablo~Ra\'{u}l Stinga, Jos\'{e}~L. Torrea, and
  Juan~Luis Varona.
\newblock {\em Nonlocal discrete diffusion equations and the fractional discrete
  {L}aplacian, regularity and applications.}
\newblock {Adv. Math.} \textbf{330} (2018), 688--738.

\bibitem[CGFR21]{CGFR21}
Giovanni Covi, Mar{\'\i}a~{\'A}ngeles Garc{\'\i}a-Ferrero, and Angkana
  R{\"u}land.
\newblock {\em On the {C}alder{\'o}n problem for nonlocal {S}chr{\"o}dinger
  equations with homogeneous, directionally antilocal principal symbols.}
\newblock {arXiv} preprint (2021), \href{http://arxiv.org/abs/2109.14976}{arXiv:2109.14976}.

\bibitem[CR21]{CR20}
Giovanni Covi and Angkana R{\"u}land.
\newblock {\em On some partial data {C}alder{\'o}n type problems with mixed boundary
  conditions.}
\newblock {J. Differential Equations} \textbf{288} (2021), 141--203.

\bibitem[EDG11]{E11}
Sylvain Ervedoza and Fr{\'e}d{\'e}ric De~Gournay.
\newblock {\em Uniform stability estimates for the discrete {C}alder{\'o}n problems.}
\newblock {Inverse problems} \textbf{27} (2011), no. 12, 125012, 37 pp.

\bibitem[FF14]{FF14}
Mouhamed~Moustapha Fall and Veronica Felli.
\newblock {\em Unique continuation property and local asymptotics of solutions to
  fractional elliptic equations.}
\newblock {Comm. Partial Differential Equations} \textbf{39} (2014), no. 2, 354--397.

\bibitem[FF15]{FF15}
Mouhamed~Moustapha Fall and Veronica Felli.
\newblock {\em Unique continuation properties for relativistic schr{\"o}dinger
  operators with a singular potential.}
\newblock { Discrete Contin. Dyn. Syst.} \textbf{35} (2015), no. 12, 5827--5867.

\bibitem[FB19]{FB19}
Aingeru Fern{\'a}ndez-Bertolin.
\newblock {\em A discrete {H}ardy’s uncertainty principle and discrete evolutions.}
\newblock {J. Anal. Math.} \textbf{137} (2019), no. 2, 507--528.

\bibitem[FBM21]{FBM21}
Aingeru Fern{\'a}ndez-Bertolin and Eugenia Malinnikova.
\newblock {\em Dynamical versions of {H}ardy’s uncertainty principle: {A} survey.}
\newblock {Bull. Amer. Math. Soc.} \textbf{58} (2021), no. 3, 357--375.

\bibitem[FBRRS21]{FBRRS20}
Aingeru Fern\'{a}ndez-Bertolin, Luz Roncal, Angkana R{\"u}land, and Diana Stan.
\newblock {\em Discrete {C}arleman estimates and three balls inequalities.}
\newblock {Calc. Var. Partial Differential Equations} \textbf{60} (2021), no. 6, paper no. 239.

\bibitem[FBV17]{FBV17}
Aingeru Fern\'andez-Bertolin and Luis Vega.
\newblock {\em Uniqueness properties for discrete equations and {C}arleman
  estimates.}
\newblock {J. Funct. Anal.} \textbf{272} (2017), no. 11, 4853--4869.

\bibitem[GFR19]{GFR19}
Mar{\'\i}a~{\'A}ngeles Garc{\'\i}a-Ferrero and Angkana R{\"u}land.
\newblock {\em Strong unique continuation for the higher order fractional
  {L}aplacian.}
\newblock {Math. Eng.} \textbf{1} (2019), no. 4, 715--774.

\bibitem[GFR20]{GFR20}
Mar{\'\i}a~{\'A}ngeles Garc{\'\i}a-Ferrero and Angkana R{\"u}land.
\newblock {\em On two methods for quantitative unique continuation results for some
  nonlocal operators.}
\newblock {Comm. Partial Differential Equations} \textbf{45} (2020), no. 11, 1512--1560.

\bibitem[GRSU20]{GRSU18}
Tuhin Ghosh, Angkana R{\"u}land, Mikko Salo, and Gunther Uhlmann.
\newblock {\em Uniqueness and reconstruction for the fractional {C}alder{\'o}n
  problem with a single measurement.}
\newblock {J. Funct. Anal.} \textbf{279} (2020), no. 1, 108505, 42 pp.

\bibitem[GSU20]{GSU16}
Tuhin Ghosh, Mikko Salo, and Gunther Uhlmann.
\newblock {\em The {C}alder\'{o}n problem for the fractional {S}chr\"{o}dinger
  equation.}
\newblock {Anal. PDE}, \textbf{13} (2020), no. 2, 455--475.

\bibitem[GM13]{GM13}
Maru Guadie and Eugenia Malinnikova.
\newblock {\em Stability and regularization for determining sets of discrete
  {L}aplacian.}
\newblock {Inverse Problems} \textbf{29} (2013), no. 7, 075018.

\bibitem[GM14]{GM14}
Maru Guadie and Eugenia Malinnikova.
\newblock {\em On three balls theorem for discrete harmonic functions.}
\newblock {Comput. Methods Funct. Theory} \textbf{14} (2014), no. 4, 721--734.

\bibitem[Isa90]{Isakov}
Victor Isakov.
\newblock {\em Inverse source problems.} 
\newblock {Mathematical Surveys
  and Monographs,} \textbf{34}.
\newblock {American Mathematical Society, Providence, RI,} 1990.

\bibitem[JL99]{JL99}
David Jerison and Gilles Lebeau.
\newblock {\em Nodal sets of sums of eigenfunctions.}
\newblock {Harmonic analysis and partial differential equations (Chicago,
  IL, 1996),} 223--239, {Chicago Lectures in Math,} 1999.

\bibitem[KN06]{KN}
Anders Karlsson and Markus Neuhauser.
\newblock {\em Heat kernels, theta identities, and zeta functions on cyclic groups.}
\newblock  { Topological and asymptotic aspects of group theory}, 177--189,  {Contemp. Math.}, 394, {Amer. Math. Soc., Providence,
  RI}, 2006.

\bibitem[Leb72]{Lebedev}
N.~N. Lebedev.
\newblock {\em Special functions and their applications}.
\newblock Revised edition, translated from the Russian and edited by Richard A.
  Silverman, Unabridged and corrected republication.
\newblock {Dover Publications, Inc., New York,} 1972.

\bibitem[Lie82]{L82}
Otto Liess.
\newblock {\em Antilocality of complex powers of elliptic differential operators
  with analytic coefficients.}
\newblock {Ann. Scuola Norm. Sup. Pisa Cl. Sci.} \textbf{9} (1982), no. 1, 1--26.

\bibitem[LM15]{LM15}
Gabor Lippner and Dan Mangoubi.
\newblock {\em Harmonic functions on the lattice: absolute monotonicity and
  propagation of smallness.}
\newblock {Duke Math. J.} \textbf{164} (2015), no. 13, 2577--2595.

\bibitem[LM17]{LM17}
Gabor Lippner and Dan Mangoubi.
\newblock {\em On the sharpness of a three circles theorem for discrete harmonic
  functions.}
\newblock {Int. Math. Res. Not. IMRN} 2017, no. 5, 1487--1503.

\bibitem[Olv97]{Olver}
Frank W.~J. Olver.
\newblock {\em Asymptotics and special functions}.
\newblock Reprint of the 1974 original.
\newblock {AKP Classics. A K Peters, Ltd., Wellesley, MA,} 1997.


\bibitem[OM10]{OlMax}
F.~W.~J. Olver and L.~C. Maximon.
\newblock {\em Bessel functions.}
\newblock  {N{IST} handbook of mathematical functions}, 215-286,
  {U.S. Dept. Commerce, Washington, DC}, 2010. Available online in \href{http://dlmf.nist.gov/10}{{http://dlmf.nist.gov/10}}.

\bibitem[PBM88]{PBM2}
A.~P. Prudnikov, Yu.~A. Brychkov, and O.~I. Marichev.
\newblock {\em Integrals and series. {V}ol. 2}.
\newblock {Special functions, Translated from the Russian by N. M. Queen}. Second edition.
\newblock {Gordon \& Breach Science Publishers}, New York, 1988.


\bibitem[Rie38]{R38}
Marcel Riesz.
\newblock {\em Int{\'e}grales de {R}iemann-{L}iouville et potentiels}.
\newblock { Acta Szeged} \textbf{9} (1938), 1--42.

  \bibitem[RSV19]{RSV19}
Luz Roncal, Diana Stan, and Luis Vega.
\newblock {\em Carleman type inequalities for fractional relativistic operators.}
\newblock To appear in {Rev. Mat. Complut.} (2019).

\bibitem[RS14]{RS_trans}
Luz Roncal and Pablo~Ra{\'{u}}l Stinga.
\newblock {\em Transference of fractional {L}aplacian regularity.}
\newblock  { Special functions, partial differential equations, and
  harmonic analysis}, 203--212, { Springer Proc. Math. Stat.}, \textbf{108}, Springer, Cham, 2014.
  
  \bibitem[Ron06]{R06}
Luca Rondi.
\newblock {\em A remark on a paper by {G}. {A}lessandrini and {S}. {V}essella}.
\newblock {Adv. in Appl. Math.} \textbf{36} (2006), no. 1, 67--69.
  
\bibitem[R{\"u}l15]{Rue15}
Angkana R{\"u}land.
\newblock {\em Unique continuation for fractional {S}chr{\"o}dinger equations with
  rough potentials.}
\newblock {Comm. Partial Differential Equations} \textbf{40} (2015), no. 1, 77--114.

\bibitem[R{\"u}l18]{Rue18}
Angkana R{\"u}land.
\newblock {\em Unique continuation, {R}unge approximation and the fractional
  {C}alder{\'o}n problem.}
\newblock {Journ{\'e}es {\'e}quations aux d{\'e}riv{\'e}es partielles} (2018), expos\'e no. 8, 10pp.

\bibitem[R{\"u}l19]{Rue17}
Angkana R{\"u}land.
\newblock {\em Quantitative invertibility and approximation for the truncated
  {H}ilbert and {R}iesz transforms.}
\newblock {Rev. Mat. Iberoam.} \textbf{35} (2019), no. 7, 1997--2024.

\bibitem[R{\"u}l21]{Rue21}
Angkana R{\"u}land.
\newblock {\em On single measurement stability for the fractional {C}alder{\'o}n
  problem.}
\newblock {SIAM J. Math. Anal.} \textbf{53} (2021), no. 5, 5094--5113.


\bibitem[RS18]{RS18}
Angkana R{\"u}land and Mikko Salo.
\newblock {\em Exponential instability in the fractional {C}alder{\'o}n problem.}
\newblock {\em Inverse Problems} \textbf{34} (2018), no. 4, 045003, 21 pp.

\bibitem[RS20a]{RS20}
Angkana R{\"u}land and Mikko Salo.
\newblock {\em The fractional {C}alder{\'o}n problem: low regularity and stability.}
\newblock { Nonlinear Anal.} \textbf{193} (2020), 111529, 56 pp.

\bibitem[RS20b]{RS20a}
Angkana R{\"u}land and Mikko Salo.
\newblock {\em Quantitative approximation properties for the fractional heat
  equation.}
\newblock {\em Math. Control Relat. Fields} \textbf{10} (2020), no. 1, 1--26.

\bibitem[RS19b]{RSi19}
Angkana R{\"u}land and Eva Sincich.
\newblock {\em Lipschitz stability for the finite dimensional fractional
  {C}alder{\'o}n problem with finite {C}auchy data.}
\newblock {Inverse Prob. Imaging} \textbf{13} (2019), no. 5, 1023, 2019.

\bibitem[Sal17]{S17}
Mikko Salo.
\newblock {\em The fractional {C}alder{\'o}n problem.}
\newblock {Journ{\'e}es {\'e}quations aux d{\'e}riv{\'e}es partielles} (2017), 
  expos\'e no. 8, 8pp.
  
  \bibitem[Seo15]{S15}
Ihyeok Seo.
\newblock {\em Unique continuation for fractional {S}chr\"odinger operators in three and higher dimensions.}
\newblock {Proc. Amer. Math. Soc.} \textbf{143} (2015), no. 4, 1661--1664.

\bibitem[Sil07]{Sil07}
Luis Silvestre.
\newblock {\em Regularity of the obstacle problem for a fractional power of the Laplace operator.}
\newblock {Comm. Pure Appl. Math.} \textbf{60} (2007), no. 1, 67--112.

\bibitem[Sin07]{S07}
Eva Sincich.
\newblock {\em Lipschitz stability for the inverse {R}obin problem.}
\newblock {Inverse problems} \textbf{23} (2007), no. 3, 1311.

\bibitem[SW]{SW}
Elias M. Stein and Guido Weiss.
\newblock {\em Introduction to Fourier Analysis on Euclidean Spaces}.
\newblock Princeton Mathematica Series \textbf{32}.
\newblock {Princeton Univ. Press, Princeton, New Jersey}, 1971.


\bibitem[ST10]{ST10}
Pablo~Ra{\'u}l Stinga and Jos{\'e}~Luis Torrea.
\newblock {\em Extension problem and {H}arnack's inequality for some fractional
  operators.}
\newblock {Comm. Partial Differential Equations} \textbf{35} (2010), no. 11, 2092--2122.

\bibitem[SVW02]{SVW02}
Alexander Strohmaier, Rainer Verch, and Manfred Wollenberg.
\newblock {\em Microlocal analysis of quantum fields on curved space--times: {A}nalytic wave front sets and {R}eeh--{S}chlieder theorems.}
\newblock {J. of Math. Phys.} \textbf{43} (2002), no. 11, 5514--5530.

\bibitem[Ver93]{V93}
Rainer Verch.
\newblock {\em Antilocality and a {R}eeh-{S}chlieder theorem on manifolds.}
\newblock {Lett. Math. Phys.} \textbf{28} (1993), no. 2, 143--154.

\bibitem[Yu17]{Yu17}
Hui Yu.
\newblock {\em Unique continuation for fractional orders of elliptic equations.}
\newblock {Ann. PDE} \textbf{3} (2017), no. 2, paper no. 16, 21pp.

\end{thebibliography}
\end{document}